\documentclass[a4paper]{amsart}

\usepackage[matrix,arrow,curve]{xy}
\usepackage[english]{babel}
\usepackage{amsmath}
\usepackage{stackrel}
\usepackage{amssymb, indentfirst}
\usepackage{amsthm}
\usepackage{verbatim}
\usepackage{mathrsfs}
\usepackage{mathabx}
\usepackage{hyperref} 

\numberwithin{equation}{section}

\DeclareMathOperator{\Hom}{Hom}
\DeclareMathOperator{\Ext}{Ext}
\DeclareMathOperator{\Ker}{Ker}
\DeclareMathOperator{\Img}{Im}
\DeclareMathOperator{\Coker}{Coker}
\DeclareMathOperator{\cone}{cone}
\DeclareMathOperator{\Mcolim}{Mcolim}
\DeclareMathOperator{\hocolim}{hocolim}
\DeclareMathOperator{\Mor}{Mor}
\DeclareMathOperator{\Iso}{Iso}
\DeclareMathOperator{\Ob}{Ob}
\DeclareMathOperator{\stpMor}{\underline{Mor}}
\DeclareMathOperator{\stiMor}{\overline{Mor}}
\DeclareMathOperator{\fp}{fp}
\DeclareMathOperator{\pres}{\!-pres}
\DeclareMathOperator{\add}{add}
\DeclareMathOperator{\Add}{Add}

\DeclareMathOperator{\Gen}{Gen}
\DeclareMathOperator{\Prod}{Prod}

\DeclareMathOperator{\Cont}{Cont}

\DeclareMathOperator{\thick}{thick}
\DeclareMathOperator{\Susp}{Susp}
\DeclareMathOperator{\Loc}{Loc}
\DeclareMathOperator{\HFun}{HFun}
\DeclareMathOperator{\res}{res}

\newcommand{\modf}{{\rm mod}\text{-}}
\newcommand{\Mod}{{\rm Mod}\text{-}}
\newcommand{\Proj}[1]{{\rm Proj}(#1)}
\newcommand{\Inj}[1]{{\rm Inj}(#1)}
\newcommand{\PInj}[1]{{\rm PInj}(#1)}

\newcommand{\ai}[1]{\tau_{\mathbf{t}}^{\leq{#1}}}
\newcommand{\coa}[1]{\tau_{\mathbf{t}}^{\geq{#1}}}
\renewcommand{\H}{{H_\mathbf{t}^0}}
\newcommand{\y}{\mathbf{y}}
\newcommand{\pure}{\mathrm{pure}}
\newcommand{\Ab}{\mathrm{Ab}}
\newcommand{\Set}{\mathrm{Set}}

\newcommand{\mcA}{\mathcal{A}}
\newcommand{\mcB}{\mathcal{B}}
\newcommand{\mcC}{\mathcal{C}}
\newcommand{\mcD}{\mathcal{D}}
\newcommand{\mcE}{\mathcal{E}}
\newcommand{\mcF}{\mathcal{F}}
\newcommand{\mcG}{\mathcal{G}}
\newcommand{\mcH}{\mathcal{H}}
\newcommand{\mcI}{\mathcal{I}}

\newcommand{\mcL}{\mathcal{L}}

\newcommand{\mcN}{\mathcal{N}}

\newcommand{\mcP}{\mathcal{P}}
\newcommand{\mcQ}{\mathcal{Q}}

\newcommand{\mcS}{\mathcal{S}}
\newcommand{\mcT}{\mathcal{T}}
\newcommand{\mcU}{\mathcal{U}}
\newcommand{\mcV}{\mathcal{V}}
\newcommand{\mcW}{\mathcal{W}}
\newcommand{\mcX}{\mathcal{X}}
\newcommand{\mcY}{\mathcal{Y}}

\newcommand{\la}{\longrightarrow}
\newcommand{\laplus}{\stackrel{+}{\la}}
\newcommand{\op}{\mathrm{op}}
\newcommand{\dd}{\colon}

\newtheorem{thm}{Theorem}[section]
\newtheorem{lem}[thm]{Lemma}
\newtheorem{prop}[thm]{Proposition}
\newtheorem{cor}[thm]{Corollary}

\theoremstyle{definition}
\newtheorem{opr}[thm]{Definition}

\theoremstyle{remark}
\newtheorem{rem}[thm]{Remark}

\newtheorem{ex}[thm]{Example}
\newtheorem{exs}[thm]{Examples}

\begin{document}

\title{$t$-structures with Grothendieck hearts via~functor categories}

\author{Manuel Saor\'{\i}n}
\address{%
Departamento de Matem\'aticas\\
Universidad de Murcia\\
30100 Espinardo, Murcia\\
SPAIN
}

\author{Jan {\v S}{\v t}ov\'\i{\v c}ek}
\address{%
Charles University in Prague\\
Faculty of Mathematics and Physics\\
Department of Algebra\\
Sokolovsk\'a 83\\
186 75 Praha\\
CZECH REPUBLIC
}

\thanks{The first-named author was supported by the Grant PID2020-113206GB-I00, funded by MCIN/AEI/10.13039/501100011033, and the project 22004/PI/22, funded by the Fundaci\'on 'S\'eneca' of Murcia, both  with a part of FEDER funds.
The second-named author was supported by the grant GA~\v{C}R 20-13778S from the Czech Science Foundation.}
\date{\today}

\begin{abstract}
We study when the heart of a $t$-structure in a triangulated category $\mcD$ with coproducts is AB5 or a Grothendieck category. If $\mcD$ satisfies Brown representability, a $t$-structure has an AB5 heart with an injective cogenerator and coproduct-preserving associated homological functor if, and only if, the coaisle has a pure-injective $t$-cogenerating object.
If $\mcD$ is standard well generated, such a heart is automatically a Grothendieck category.
For compactly generated $t$-structures (in any ambient triangulated category with coproducts), we prove that the heart is a locally finitely presented Grothendieck category.

We use functor categories and the proofs rely on two main ingredients. Firstly, we express the heart of any $t$-structure in any triangulated category as a Serre quotient of the category of finitely presented additive functors for suitable choices of subcategories of the aisle or the co-aisle that we, respectively, call $t$-generating or $t$-cogenerating subcategories.
Secondly, we study coproduct-preserving homological functors from $\mcD$ to complete AB5 abelian categories with injective cogenerators and classify them, up to a so-called computational equivalence, in terms of pure-injective objects in~$\mcD$.
This allows us to show that any standard well generated triangulated category $\mcD$ possesses a universal such coproduct-preserving homological functor, to develop a purity theory and to prove that pure-injective objects always cogenerate $t$-structures in such triangulated categories.
\end{abstract}

\maketitle

\setcounter{tocdepth}{2} 
\tableofcontents


\section{Introduction}

The main motivation of this paper is the study of $t$-structures in triangulated categories with coproducts whose hearts are AB5 abelian or  Grothendieck categories. Along the way, we initiate a theory of purity (which is a concept from the model theory of modules over a ring) for not necessarily compactly generated triangulated categories. In this context, purity is very closely related to the study of covariant coproduct-preserving homological functors and representability theorems for them and, at the end of the day, we apply these results to the (co)homological functors induced by $t$-structures. Our results are mostly independent of any particular model or enhancement for the triangulated categories.

\smallskip

The problem of identifying the $t$-structures whose heart is a Grothendieck category has deserved a lot of attention since it first arose for the Happel-Reiten-Smal\o{} $t$-structure associated to a torsion pair in a Grothendieck or module category \cite{CGM,CMT}. For the general question,  several strategies have been used to tackle the problem, including ad hoc arguments \cite{PS-hearts,B-tstr}, functor categories \cite{Bo16,Bo} and suitable enhancements of the ambient triangulated category, such as stable $\infty$-categories \cite{Lur-HA,Lur-SAG} or derivators \cite{SSV,Lak}.

When the ambient triangulated category is compactly generated, the well-de\-vel\-oped theory of purity in this type of categories, initiated in~\cite{K}, has been also used \cite{St-cotilt,AMV,Lak}.
One of the most common strategies here consisted in expressing the heart of a well-behaved $t$-structure (e.g.\ compactly generated or smashing) as Gabriel quotient of the category $\Mod\mcD^c$ of additive functors $(\mcD^c)^\op\la\Ab$, where $\mcD^c$ is the subcategory of compact objects.
A key limitation of this approach so far, which we aim to overcome here, is that it is in contrast to enhancement-based arguments fully dependent on the existence of compact objects---an assumption which may easily fail even for derived categories of sheaves~\cite{N-manifolds}.
Albeit a higher-cardinal generalization of the purity theory has been developed in connection with Verdier quotients and semiorthogonal decompositions of triangulated categories~\cite{N,K-triang}, it is not suitable for us (with the exception of Proposition~\ref{prop.t-structure from pi}) for the following reasons:
\begin{enumerate}
\item the higher-cardinal version of purity seems not to be well-suited for studying exactness of all direct limits and
\item various arguments about localizations of triangulated categories do not seem to directly generalize to $t$-structures.
\end{enumerate}

Although we do follow the trend of using functor categories in this paper, we do so in a different (and initially much more general) way.
We start working in an arbitrary triangulated category $\mcD$ with a $t$-structure $\mathbf{t}=(\mcU,\mcV)$ and
we replace the no longer suitable or even well-defined category $\Mod\mcD^c$ by the category $\modf\mcX$ (or $\modf\mcX^\op$) of finitely presented functors $\mcX^\op\la\Ab$ (or $\mcX\la\Ab$), for a suitable subcategory $\mcX$ of $\mcD$ that is linked to $\mathbf{t}$. Normally $\mcX$ will be the aisle or the co-aisle of~$\mathbf{t}$ or a suitable subcategory of them.
If $\mcD$ is a triangulated category with products, we can very abstractly define pure-injective objects in $\mcD$, choose $\mcX$ to be a class of pure-injective objects, and use this approach together with a recent criterion for the AB5 condition given by Positselski and the second-named author~\cite{PoSt} in terms of pure-injectivity. This turns out to be a very efficient strategy to study the AB5 and Grothendieck property of the heart of $\mathbf{t}$. The advantage is that one gets rid of any model enhancing the ambient triangulated category, thus obtaining completely general results.

\smallskip

Let us now describe  the contents of the paper, in the course of which the main results will be explained. In Section~\ref{sec:prelim} we introduce most of the concepts and terminology to be used in the paper. Already there  we  take some care of the results which are crucial for the paper. In particular,  we show how to reconstruct an abelian category with enough projectives from its subcategory of projectives, we revisit the notions of localization and Serre quotient functors and we recall  criteria for the property of being locally finitely presented to be inherited via Gabriel localization of a Grothendieck category.  

In Section~\ref{sec.t-structures and localization} we show how the heart of a $t$-structure appears as Serre quotient  of the category $\modf\mcU$ of finitely presented functors $\mcU^\op\la\Ab$, where $\mcU$ is the aisle of the $t$-structure,  and give some ideas on how to get rid of degeneracies of $t$-structures.  In Section~\ref{sec:t-generating} we go one step further and show that if $\mcP$ is a suitable precovering subcategory of $\mcU$, then the Serre quotient functor $\modf\mcU\la\mcH$, where $\mcH$ is the heart of the $t$-structure, factors as $\modf\mcU\stackrel{\res}{\la}\modf\mcP\stackrel{F}{\la}\mcH$, where $F$ is again a Serre quotient functor. This gives the following first main result of the paper (see Theorem \ref{thm.t-generating left Kan extensions} for an extended version), that together with its dual give one of our most powerful tools to study the AB5 condition of the heart of a $t$-structure in a triangulated category with coproducts, although the result is valid for all $t$-structures in any triangulated category:
 
\begin{thm} \label{thm.t-generating left Kan extensions-intro}
Let $\mcD$ be a triangulated category and $\mathbf{t}=(\mcU,\mcV)$ be a $t$-structure with heart $\mcH$ and the associated cohomological functor $\H\dd\mcD\la\mcH$. Let further $\mcP\subseteq\mcU$ be a precovering subcategory and denote by  $\y_\mcP$ the generalized Yoneda functor
\begin{align*}
\y_\mcP\dd \mcU &\la \modf\mcP, \\
U &\rightsquigarrow \Hom_\mcU(?,U)_{|\mcP}.
\end{align*}
The following assertions are equivalent:
\begin{enumerate}
\item The functor $\H\dd\mcU\la\mcH$ factors as a composition $\mcU\stackrel{\y_\mcP}{\la} \modf\mcP\stackrel{F}{\la}\mcH$, for some right exact functor $F$.
\item The subcategory $\mcP\subseteq\mcD$ is $t$-generating, i.e.\ for each $U\in\mcU$ there is a triangle  $U'\la P\stackrel{f}{\la} U\laplus$, where  $P\in\mcP$ and $U'\in\mcU$.
\end{enumerate}
In such a case, $F$ is a Serre quotient functor  and $G:=(\y_\mcP)_{| \mcH}\dd\mcH\la \modf\mcP$ is its fully faithful right adjoint. 
\end{thm}

In Section~\ref{sec.pure injectives}
we introduce the key notion of pure-injective object in an arbitrary additive category with products, which extends the corresponding existing notion in locally finitely presented additive categories and in compactly generated triangulated categories. We then revisit a recent result by Positselski and the second-named author from \cite{PoSt}, stating that an AB3* abelian category $\mcA$ with an injective cogenerator $E$ is AB5 if and only if $E$ is pure-injective. We further show that $\mcA$ is a Grothendieck category precisely when $\Prod E=\Inj\mcA$ has a generator, i.e.\ if and only if there is $E'\in\Inj\mcA$ such that  $\Hom_{\Inj\mcA}(E',?)\dd\Inj\mcA\la\Ab$ is a faithful functor.

In Section~\ref{sec:representability} we prove  the following  theorem for coproduct-preserving homological functors whose targets are AB3* abelian categories with an injective cogenerator.
The reader is referred to Definition \ref{def.computationally equivalent} for the precise definition of computationally equivalent coproduct-preserving homological functors whose domain is a given triangulated category with coproducts $\mcD$. A fortiori, when $\mcD$ satisfies Brown representability theorem,   two such functors are computationally equivalent exactly when the morphisms in $\mcD$ that are killed by  one of them are also killed by the other (see Corollary \ref{cor.clean-definition-comp.equivalent}).

\begin{thm} \label{thm.classify homological functors-intro}
Let $\mcD$ be a triangulated category which has arbitrary (set-indexed) coproducts and satisfies Brown representability theorem. Then there is a bijective correspondence between
\begin{enumerate}
\item Computational equivalence classes of coproduct-preserving homological functors $H\dd\mcD\la\mcA$, where $\mcA$ is an AB3* abelian category with an injective cogenerator.
\item Product-equivalence classes of objects in $\mcD$.
\end{enumerate}
The bijection restricts to another one, where in \emph{(1)} we only consider those homological functors with AB5 target and in \emph{(2)} we only consider product-equivalence classes of pure-injective objects.

Moreover, each computational equivalence class in \emph{(1)} has unique initial object $H\dd\mcD\la\mcA$. If $Q\in\mcD$ represents the corresponding product equivalence class as in~\emph{(2)}, then one can take $\mcA=\big(\modf\Prod(Q)^\op\big)^\op$ and for any $D\in\mcD$,
\[ H(D) = \Hom_\mcD(D,?)_{|\Prod(Q)}\dd \Prod(Q) \la \Ab. \]
\end{thm}

The main significance of the latter theorem is that it allows us to initiate a theory of purity for non-compactly generated triangulated categories.
As a consequence, it turns out that pure-injective objects in practice always cogenerate $t$-structures. See Proposition~\ref{prop.t-structure from pi} which substantially generalizes a result of similar nature in~\cite{LV}.

So far, two different approaches to purity appeared in the literature in the absence of finitely presented or compact objects:
\begin{enumerate}
\item via abstractly defined pure-injective objects (as discussed above) in~\cite{CS,PoSt} and
\item via colimit-preserving functors with AB5 target categories in~\cite[\S6]{BaPo}.
\end{enumerate}
Theorem~\ref{thm.classify homological functors-intro} says that if we replace, in the context of triangulated categories, the functors in~(2) by the class of coproduct-preserving homological functors to complete AB5 abelian categories with injective cogenerators, the two approaches become equivalent.

In Section~\ref{sect.universal copr-pres} we further develop the purity theory for standard well generated triangulated categories and show that any such category~$\mcD$ has an associated Grothendieck category $\mcA_\pure(\mcD)$ and a coproduct-preserving homological functor $h_\pure\dd\mcD\la\mcA_\pure(\mcD)$, uniquely determined up to equivalence, that are universal. This means that if $h\dd\mcD\la\mcA$ is any other coproduct-preserving homological functor with AB5 target, then there is a coproduct-preserving exact functor $F\dd\mcA_\pure(\mcD)\la\mcA$, unique up to natural isomorphism, such that $F\circ h_\pure =h$.
Then we can simply define pure triangles and identify pure-injective objects in terms of this universal functor $h_\pure$.

Section~\ref{sec.t-str with Grothendieck heart} is the one specifically dedicated to the study of $t$-structures with an AB5 or  Grothendieck heart. The first general result of the section is the following (see Theorem \ref{thm.AB5 heart} for a more detailed version). 

\begin{thm} \label{thm.AB5 heart-intro}
Let $\mcD$ be a triangulated category with  coproducts that satisfies Brown representability theorem, and let $\mathbf{t}=(\mcU,\mcV)$ be a $t$-structure with heart $\mcH$. The following assertions are equivalent

\begin{enumerate}
\item There exists a pure-injective object $Q\in\mcV$ such that  $\Hom_\mcD(?,Q)$ vanishes on $\mcV[-1]$  and  $\Hom_\mcD(M,Q)\neq 0$, for all $0\neq M\in\mcH$.
\item There is a pure-injective object $\hat{Q}\in\mathcal{V}$ such that, for each $V\in\mcV$, there is a triangle  $V\longrightarrow\hat{Q}_V\longrightarrow V'\stackrel{+}{\longrightarrow}$, where $\hat{Q}_V\in\Prod(Q)$ and $V'\in\mathcal{V}$.
\item  $\mcH$ is an AB5 abelian category with an injective cogenerator and the cohomological functor $\H\dd\mcD\la\mcH$ preserves coproducts. 
\end{enumerate}

When $\mcD$ is standard well generated, they are also equivalent to:

\begin{enumerate}
\item[(4)] $\mcH$ is a Grothendieck category and the cohomological functor $\H\dd\mcD\la\mcH$ preserves coproducts. 
\end{enumerate}
\end{thm}

Condition~(2) of Theorem~\ref{thm.AB5 heart-intro} is very closely related to widely studied finiteness conditions on the co-aisle of $\mathbf{t}$. When $\mcD$ is compactly generated, condition~(2) is satisfied for instance in the following situations (see Theorem~\ref{thm.Groth.heart-for-definable-coaisle}):
\begin{itemize}
\item if the co-aisle $\mcV$ is definable (this in particular holds if $\mathbf{t}=(\mcU,\mcV)$ is compactly generated as a $t$-structure) or
\item if $\mcD$ has a suitable enhancement (as explained in Remarks~\ref{rem:definable coaisles} and~\ref{def.t-cogenerated coaisle}) and $\mcV$ is closed under taking directed homotopy colimits.
\end{itemize}
In this way, we generalize various results which appeared in the literature before---for $t$-structures in presentable stable $\infty$-categories (\cite[Remark 1.3.5.23]{Lur-HA}, \cite[Remark C.1.4.6]{Lur-SAG}), for homotopically smashing $t$-structures in nice enough stable derivators (\cite{SSV}, \cite[Theorem 4.6]{Lak}) or for compactly generated $t$-structures (\cite[Theorem 0.2]{Bo}).

The special case where $\mathbf{t}$ is a semiorthogonal decomposition was also thoroughly studied in~\cite{K,K-cohq} and, in particular, such decompositions $\mathbf{t}$ were classified in terms of certain (so called exact) ideals of the subcategory of all compact objects in $\mcD$
and it was proved that they give rise to recollements in the sense of~\cite[\S1.4]{BBD}.
In Theorems~\ref{thm.classification definable} and~\ref{thm.definable-have-adjacent-cotstr}, we establish a completely analogous classification of all $t$-structures with definable co-aisle
and show that they possess right adjacent co-$t$-structures.
The latter is an analogy of recollements in the context of $t$-structures (as explained e.g.\ in the introduction of~\cite{StPo}).

\begin{thm} \label{thm.classification definable-intro}
Let $\mcD$ be a compactly generated triangulated subcategory. Then there is a bijective correspondence between
\begin{enumerate}
\item the $t$-structures $\mathbf{t}=(\mcU,\mcV)$ in $\mcD$ with $\mcV$ definable, and
\item suspended two-sided ideals $\mcI\subseteq\mcD^c$, i.e.\ ideals which satisfy $\mcI[1]\subseteq\mcI=\mcI^2$ and are saturated (see Definition~\ref{def.suspended ideal}).
\end{enumerate}
Moreover, any $t$-structure as in~\emph{(1)} admits a right adjacent co-$t$-structure $(\mcV,\mcW)$.
\end{thm}


Still in Section~\ref{sec.t-str with Grothendieck heart}, in the yet more special case of compactly generated $t$-structures, we go one step further and prove the following result (see Theorem \ref{thm.locally-fp-heart}):
 
\begin{thm} \label{thm.locally-fp-heart-intro}
Let $\mcD$ a triangulated category with coproducts, let $\mathbf{t}=(\mcU,\mcV)$ be a compactly generated $t$-structure in $\mcD$, with heart $\mcH$,  and put $\mcU_0=\mcU\cap\mcD^c$. Then  $\mcH$ is a locally finitely presented Grothendieck category and its subcategory of finitely presented objects is
$\fp(\mcH)=\H(\mcU_0)$. 

When in addition $\mathbf{t}$ restricts to the subcategory $\mcD^c$ of compact objects, the heart $\mcH$ is also locally coherent.
\end{thm}

 In the final Section~\ref{sec.cosilting} we show relations between $t$-structures with Grothendieck heart and various versions of partial cosilting objects that recently appeared in the literature.

\subsection*{Acknowledgment}

We would like to thank Michal Hrbek for helpful discussions and Rosie Laking and Jorge Vit\'{o}ria  for some clarifications.


\section{Preliminaries}
\label{sec:prelim}

Unless otherwise specified, all categories in this paper will be pre-additive and all functors are additive. All subcategories will be full and closed under taking isomorphisms. When we say that such a category, say $\mcA$, has (co)products we will mean that it has arbitrary set-indexed (co)products. When $\mcA$ is additive, for a given subcategory $\mcS$, we shall denote by $\add_\mcA(\mcS)$ and   $\Add_\mcA(\mcS)$ the subcategories consisting of the direct summands, respectively, of finite and arbitrary coproducts of objects in $\mcS$.
Dually $\text{Prod}_\mcA(\mcS)$ will stand for the subcategory of direct summands of products of objects of $\mcS$. The group of morphisms between objects $X$ and $Y$ will be indistinctly denoted by $\mcA(X,Y)$ or $\Hom_\mcA (X,Y)$.  We will denote by $\mcS^\perp$ (resp. $_{}^\perp\mcS$) the subcategory of $\mcA$  consisting of the objects $X$ such that $\Hom_\mcA(S,X)=0$ (resp.  $\Hom_\mcA(X,S)=0$), for all $S\in\mcS$. 

A morphism $f\dd S\la X$ in $\mcA$ is called an \emph{$\mcS$-precover} if $S\in\mcS$ and any morphism $f'\dd S'\la X$ with $S'\in\mcS$ factors through $f$. The subcategory $\mcS$ is called \emph{precovering} is each $X\in\mcA$ admits an $\mcS$-precover $f\dd S\la X$. Dually, one defines an \emph{$\mcS$-preenvelope} $f\dd X\la S$ and \emph{preenveloping} subcategories of $\mcA$.

We refer the reader to \cite{Pop} and \cite{St} for the basic notions concerning abelian categories, in particular for the terminology AB$n$ and AB$n^*$, for $n=3,4,5$, introduced in \cite{Gro}. Recall that an AB5 abelian category with a set of generators (equivalently, a generator) is called a \emph{Grothendieck category}.


\subsection{Abelian categories with enough projective objects}
\label{subsec:abel-enough-proj}

We start by recalling basic and mostly well known facts about how to reconstruct an abelian category from its subcategory of projective objects, provided we have enough of these. All the results formally dualize to abelian categories with enough injective objects as well.
If $\mcA$ is an abelian category, we will denote by $\Proj\mcA$ the full subcategory of projective objects and by $\Inj\mcA$ the full subcategory of injective objects.

For any (not necessarily small) additive category $\mcP$, we denote by $\modf\mcP$ the category of finitely presented functors $\mcP^\op\la\Ab$, which are by definition functors $F$ with a presentation
\[ \Hom_\mcP(-,Q) \la \Hom_\mcP(-,P) \la F \la 0 \]
given by a map $f\dd Q\la P$ in $\mcP$. We will also frequently use the shorthand notation $\widehat{\mcP}:=\modf\mcP$. 
Observe that, thanks to the Yoneda lemma, the collection of natural transformations between any pair of finitely presented functors forms a set.
For the following  well known lemma (see e.g.~\cite[Corollary 1.5]{Freyd} or \cite[\S2]{K-functors-lfp}), we need the notion of \emph{weak kernel} of a morphism $f\dd X\la Y$ in an additive category $\mcA$. It is just a morphism $u\dd K\la X$ such that the associated sequence of functors $\Hom_\mcA(?,K)\stackrel{u_*}{\la}\Hom_\mcA(?,X)\stackrel{f_*}{\la}\Hom_\mcA(?,Y)$ is exact. \emph{Weak cokernels} are defined dually. 

\begin{lem} \label{lem.coherent functors are abelian}
The Yoneda embedding 
\begin{align*}
\y_\mcP\dd \mcP &\la \modf\mcP, \\
P &\rightsquigarrow \Hom_\mcP(?,P), \\
\end{align*}
has the following universal property: Any additive functor $F\dd\mcP\la\mcA$, where $\mcA$ is an abelian category, extends uniquely up to natural isomorphism over $\y_\mcP$ to a right exact functor $\widehat{F}\dd\modf\mcP\la\mcA$, and any natural transformation $\alpha\dd F\la F'$ between such additive functors uniquely extends to a natural transformation $\widehat{\alpha}\dd\widehat{F}\la\widehat{F'}$.

Moreover, the category $\modf\mcP$ is itself abelian if and only if the kernel of any map of finitely presented functors is finitely presented if and only if the category $\mcP$ has weak kernels.
\end{lem}

\begin{rem}
Note that the previous statement says that the precomposition functor $\y_\mcP^*\dd [\widehat{\mcP},\mcA]_{\textsf{rex}}\stackrel{\simeq}{\la}[\mcP,\mcA]_{\textsf{add}}$ induces an equivalence, which is in this case even surjective on objects, between the category of right exact functors $\widehat{\mcP}\la\mcA$ and the category of additive functors $\mcP\la\mcA$. In particular, the lifting of $F$ to $\widehat{F}$ is, as is well known, unique up to a canonical natural isomorphism.
\end{rem}

Given an additive category $\mcP$, we also denote by $\Mor(\mcP)$ the category of morphisms in $\mcP$ (see \cite[Section I.2]{ARei}) and we denote by $\stpMor(\mcP)$ the quotient of $\Mor(\mcP)$ by the ideal of projectively trivial morphisms, in the terminology of [op.cit]. More in detail, we factor out the two-sided ideal of $\Mor(\mcP)$ of all maps which factor through a split epimorphism in $\mcP$, when viewed as an object of $\Mor(\mcP)$ (what we denote $\stpMor(\mcP)$ is denoted by $\Mod\mcP$ in \cite{ARei}).
The following result is standard and provides two ways to reconstruct an abelian category from the subcategory of projective objects.

\begin{prop} \label{prop.reconstructing}
Let $\mcB$ be an abelian category with enough projective objects and denote by $\mcP$ the full subcategory of projective objects. Then
\[ \stpMor(\mcP) \simeq \mcB \simeq \modf\mcP, \]
where the left hand side equivalence sends $(f\dd Q\la P)\in\stpMor(\mcP)$ to $\Coker(f)$ and the second equivalence sends $B\in\mcB$ to $\Hom_\mcB(?,B)_{|\mcP}$.
\end{prop}

\begin{proof}
The first equivalence was proved in \cite[Section I.2]{ARei}, while the second one essentially follows from \cite[Proposition 2.3]{K-functors-lfp} as the assignment $B\rightsquigarrow\Hom_\mcB(?,B)_{|\mcP}$ restricts by the Yoneda lemma to an equivalence between the projective objects in $\modf\mcP$ and $\mcB$, respectively.
\end{proof}

We will also need a perhaps less well known version of this result involving AB3 categories $\mcB$ with a projective generator. This has been worked out in~\cite[\S6]{PoSt-TCC} in the language of monads, but we will use a more direct formulation which will be convenient for us. It in fact instantiates $\mcB$ as the category of models of an algebraic theory in the sense of~\cite{Wraith}.

For this purpose, suppose that $\mcA$ is an additive category with arbitrary (set-indexed) products with the property that $\mcA=\Prod_\mcA(A)$ for some $A\in\mcA$. We denote by $\Cont(\mcA,\Ab)$ the category of all product-preserving additive functors $\mcA\la\Ab$. Note that again, there is only a set of natural transformations between any pair of functors in $\Cont(\mcA,\Ab)$, as any transformation is determined by its value on $A\in\mcA$.

\begin{lem} \label{lem.Cont(P)}
Let $\mcP$ be an additive category with coproducts and $P\in\mcP$ such that $\mcP=\Add_\mcP(P)$. Then $\mcP$ has weak kernels, and $\modf\mcP = \Cont(\mcP^\op,\Ab)$. In particular, $\Cont(\mcP^\op,\Ab)$ is an abelian category with coproducts and $\Hom_\mcP(-,P)$ is a projective generator.
\end{lem}

\begin{proof}
Suppose that $f\dd P_1\la P_0$ is a morphism in $\mcP$. If we consider the set $Z$ of all morphism $g\dd P\la P_1$ such that $fg=0$, then the canonical morphism $P^{(Z)} \la P_1$ is easily seen to be a weak kernel of $f$. Hence $\modf\mcP$ is abelian. Moreover, since $\mcP$ has coproducts, so have them both $\Mor(\mcP)$ and $\stpMor(\mcP)\simeq\modf\mcP$.
	
It remains to establish the equality $\modf\mcP=\Cont(\mcP^\op,\Ab)$. Clearly, any finitely presented functor $\mcP^\op\la\Ab$ preserves products as all representable functors do and products are exact in $\Ab$.
	
For the converse, choose $F\in\Cont(\mcP^\op,\Ab)$ and denote by $X$ the underlying set of $F(P)$. Then $F(P^{(X)})\cong F(P)^X = X^X$ by the assumption and, hence, we can consider the canonical element $c\in F(P^{(X)})$ whose $x$-th component under the latter identification equals $x$. By the Yoneda lemma, $c$ determines a morphism $\phi\dd\Hom_\mcP(?,P^{(X)})\la F$. Observe that $\phi_P\dd\Hom_\mcP(P,P^{(X)})\la F(P)$ is surjective as the $x$-th coproduct inclusion $P\rightarrowtail P^{(X)}$ maps to $x$ for each $x\in X$. Since both $\Hom_\mcP(?,P^{(X)})$ and $F$ commute with products, $\phi(Q)$ is in fact surjective for any $Q\in\mcP$ and, hence, $F$ is a quotient of $\Hom_\mcP(?,P^{(X)})$. Iterating the argument one more time with $K = \Ker(\phi) \in \Cont(\mcP^\op,\Ab)$, we obtain the required presentation for $F$.
\end{proof}

By combining Proposition~\ref{prop.reconstructing} and Lemma~\ref{lem.Cont(P)}, we obtain:

\begin{cor} \label{cor.Cont(P)}
Let $\mcB$ be an AB3 category with a projective generator $P$, and denote $\mcP=\Add(P)$ the full subcategory of projective objects. Then $\mcB\simeq\Cont(\mcP^\op,\Ab)$ via the restricted Yoneda functor $B\rightsquigarrow\Hom_\mcB(?,B)_{|\mcP}$.
\end{cor}


\subsection{Localization of categories}
\label{subsec:localization-func}

Next we recall basic facts about a key concept in this paper---localization of categories.
A functor $F\dd\mcC\la\mcC'$ is a \emph{localization functor} at a class of morphism $\mcS$ of $\mcC$ if for any category $\mcE$, the precompostion functor
\[ F^*\dd [\mcC',\mcE] \la [\mcC,\mcE] \]
between the categories of functors is fully faithful and the essential image consists of those functors $G\dd\mcC\la\mcE$ which send all morphisms in $\mcS$ to isomorphisms in $\mcE$.

\begin{rem}\label{rem.set-theory-and-localization}
Of course, having written that, we need to explain how to interpret this statement in the context of the usual set-theoretic foundation of mathematics. We have three possibilities:

\begin{enumerate}
\item Assume that all our categories are small. In that situation, no problems arise as for any category $\mcC$ and any set of morphisms $\mcS$, the corresponding localization functor between small categories always exists and is essentially unique by \cite[\S1.1]{GZ}.

\item If the categories in question are not small---a situation which we encounter in this paper---we can assume that we can enlarge the universe and apply the results in the larger universe, whence making our categories efficiently small and reducing to case~(1). The conclusions are then valid in the original universe as well, up to one aspect where one has to be cautious: Localizations of locally small categories still exist by~\cite{GZ} and do not enlarge the class of objects, but a localization of a locally small category may possess pairs of objects which admit a proper class of morphisms among them. As long as we can prove in some way that this problem does not arise for the categories which we work with (one usually uses Lemma~\ref{lem.detect localization} below), we can apply the results of this section even for categories which are not small. This is our preferred variant since it provides a good trade off between clarity and rigor.

\item Many arguments which may seem dubious from the set-theoretic point of view at a first glance can be actually salvaged with some effort because they are completely constructive. We will not follow this path, however, because this additional effort often comes at the cost of clarity of exposition.
\end{enumerate}
\end{rem}

The following lemma provides a practical method of detecting localization functors, see~\cite[1.3 Proposition]{GZ}.

\begin{lem} \label{lem.detect localization}
Suppose that $F\dd\mcC\la\mcC'$ is a functor which admits a left or right adjoint $G\dd\mcC'\la\mcC$. Then $F$ is a localization functor (at the class of all morphisms $f$ such that $F(f)$ is invertible) if and only if $G$ is fully faithful. 
\end{lem}

In general, it is not obvious whether a composition of two localization functors is a localization functor again. For functors with adjoints (on any side), the situation is, however, easy.
We provide the lemma with a (completely elementary) proof.

\begin{lem} \label{lem.composition and cancellation of localizations}
	Let $F\dd\mcC\la\mcC'$ and $G\dd\mcC'\la\mcC''$ be functors. Then the following hold:
	\begin{enumerate}
		\item If $F$ and $G$ are localization functors and $F$ has a left or right adjoint, then $G\circ F\dd \mcC\la\mcC''$ is a localization functor,
		\item If $F$ and $G\circ F$ are localization functors, so is $G\dd\mcC'\la\mcC''$.
	\end{enumerate} 
\end{lem}

\begin{proof}
	In both statements, $F$ is assumed to be a localization functor and we pick a class $\mcS$ of morphisms of $\mcC$ such that $F$ is a localization at $\mcS$. 
	
	(1) We denote by $\iota\dd \mcC'\la\mcC$ the (left or right) fully faithful adjoint to $F$.
	Suppose that $G$ is a localization at a class $\mcS'$ of morphisms of $\mcC'$. It is clear that the functor $(GF)^*=F^*G^*\dd [\mcC'',\mcE]\la[\mcC,\mcE]$ is fully faithful for each category $\mcE$. Since each morphism $f\in\Mor(\mcC')$ is isomorphic to $F(\iota(f))$ by \cite[Proposition 1.3]{GZ}, one directly identifies the essential image of $(GF)^*$. It consists precisely of those functors which send the morphisms in $\mcS\cup\iota(\mcS')$ to isomorphisms.
	
	(2) Suppose that $GF$ is a localization at $\mcS''\subseteq\Mor(\mcC)$ and, without loss of generality, $\mcS''\supseteq\mcS$. Since both $(GF)^*$ and $F^*$ are fully faithful for any category $\mcE$, so must be the functor $G^*\dd[\mcC'',\mcE]\la[\mcC',\mcE]$. One again checks in a straightforward manner that the essential image of $G^*$ consists of the functors which send $F(\mcS'')$ to isomorphisms.
\end{proof}

If $\mcA$ and $\mcB$ are abelian categories and $F\dd\mcA\la\mcB$ is an exact localization functor, it is called a \emph{Serre quotient functor}. In this case, a morphism $F(f)$ is an isomorphism if and only if $F(\Ker f) = 0 = F(\Coker f)$. The full subcategory
\[ \Ker F=\{ X\in\mathcal{A} \mid F(X) = 0 \} \]
is closed under subobjects, factor-objects and extensions. A subcategory of an abelian category with these properties is called a \emph{Serre subcategory}. Serre quotient functors originating in $\mcA$ are (up to equivalence) precisely classified by Serre subcategories of $\mcA$.

Inspired by the results in~\cite[Chapitre III]{G} and the Gabriel-Popescu theorem (e.g.~\cite[\S X.4]{St}), we call a Serre quotient functor $F\dd\mcA\la\mcB$ with a (fully faithful) right adjoint functor $\iota\dd\mcB\la\mcA$ a \emph{Gabriel localization functor}. The right adjoint $\iota$ is then called a \emph{section functor}. When in addition $\mcA$ is AB3 (i.e. has set-indexed coproducts), then $F$ preserves coproducts and $\Ker F$ is closed under subobjects, factor-objects, extensions and arbitrary coproducts (see \S\ref{subsec.Gabriel-Popescu} for a more detailed discussion of this situation).

We conclude the subsection with a technical but rather useful statement which says that under certain conditions, an exact functor with a fully faithful \emph{left} adjoint is a Gabriel localization functor.

\begin{prop} \label{prop.left adjoint-implies-right adjoint}
Let $F\dd\mcA\longrightarrow\mcB$ be an exact functor between abelian categories, where $\mcA$ is complete AB5 and has an injective cogenerator. If $F$ has a fully faithful left adjoint, then it also has a fully faithful right adjoint. 
In this case, $F$ is a Gabriel localization functor and $\mcB$ is also AB5 with an injective cogenerator. Moreover, if $\mcA$ is a Grothendieck category, so is $\mcB$.
\end{prop}

\begin{proof}
By exactness and Lemma \ref{lem.detect localization} we know that $F$ is  a Serre quotient functor and $\mcT=\Ker(F)$ is the corresponding Serre subcategory. If $F$ has a left adjoint, it preserves products, and consequently $\mcT$ is closed under products in $\mcA$. However, the exactness of direct limits implies that the canonical map $\coprod_{i\in I}A_i\longrightarrow\prod_{i\in I}A_i$ is a monomorphism, for each family of objects $(A_i)_{i\in I}$ in $\mcA$,
as it is a direct limit of the split monomorphisms $\coprod_{i\in J}A_i=\prod_{i\in J}A_i\longrightarrow\prod_{i\in I}A_i$, where $J$ ranges over finite subsets of $I$ (cf.\ \cite[Exercise 1, p.~133]{St}).
Therefore, $\mcT$ is closed under taking coproducts and, hence, also under arbitrary colimits. It follows that each object $A\in\mcA$ has a unique maximal subobject in $\mcT$, which is simply the direct union of all subobjects of $A$ which belong to $\mcT$. The fact that $F$ has a fully faithful right adjoint then follows from~\cite[Corollaire III.3.1]{G}, since any object of $\mcA$ has an injective envelope by~\cite[Proposition V.2.5]{St}. Consequently, $F$ preserves all limits and colimits and, as it is also essentially surjective, it takes (co)generators to (co)generators as well (cf.\ \cite[Lemme III.2.4]{G}). Finally, $\mcB$ has an injective cogenerator by \cite[Corollaire III.3.2]{G}.
\end{proof}

\subsection{A generalized Gabriel-Popescu Theorem}
\label{subsec.Gabriel-Popescu}

When $\mcG$ a Grothendieck category, an object $X$ is called \emph{finitely presented} if the functor $\Hom_\mcG(X,?)\dd\mcG\la\Ab$ preserves direct limits. We denote by  $\fp(\mcG)$ the subcategory of finitely presented objects.
We say that $\mcG$ is \emph{locally finitely presented} when it has a set $\mcS$ of finitely presented  generators. This is equivalent to saying that  $\fp(\mcG)$ is skeletally small and each object of $\mcG$ is a direct limit of objects in $\fp(\mcG)$ (see \cite{CB} and \cite{Pr}). Indeed $\fp(\mcG)$ then consists precisely of those objects $X\in\mcG$ which admit an exact sequence $\coprod_{i=1}^mS_i\la\coprod_{j=1}^nS'_j\la X\la 0$, for some finite families $(S_i)$ and $(S'_j)$ of objects of $\mcS$. We say that $\mcG$ is \emph{locally coherent} when it is locally finitely presented and $\fp(\mcG)$ is an abelian exact subcategory or, equivalently, closed under taking kernels in $\mcG$.

Suppose that $\mcG$ is a Grothendieck category in the rest of this subsection.  A \emph{torsion pair} in $\mcG$ is a pair $\tau =(\mcT,\mcF)$ of subcategories such that $\mcF=\mcT^\perp$ and $\mcT={^\perp\mcF}$. In such case $\mcT$ is called the \emph{torsion class} and $\mcF$ the \emph{torsionfree class}.
 Such a torsion pair (or the torsion class $\mcT$) is called \emph{hereditary} when $\mcT$ is closed under taking subobjects in $\mcG$. The pair $\tau$ is called a \emph{torsion pair of finite type} when $\mcF$ is closed under taking direct limits in $\mcG$. 
 
When $\mcG$ is a Grothendieck category and $\mcT$ is a hereditary torsion class,  the localization  $\mcG/\mcT :=\mcG[\Sigma_\mcT^{-1}]$ with respect to the class $\Sigma_\mcT$ of morphisms $s\dd X\la X'$ in $\mcG$ such that $\Ker(s),\Coker(s)\in\mcT$ has Hom sets. We call $\mcG/\mcT$ the \emph{quotient category} of $\mcG$ by $\mcT$ and the corresponding localization functor $q\dd \mcG\la\mcG/\mcT$ is a Gabriel localization functor in the sense of~\S\ref{subsec:localization-func}.
It is well known (see~\cite{G,St}) that $\mcG/\mcT$ is again a Grothendieck category and that $\Ker(q)=\mcT$. If $\iota\dd\mcG/\mcT\la\mcG$ is the (fully faithful) right adjoint to $q$, then we call $\mathcal{Y}:=\text{Im}(\iota)$ the associated \emph{Giraud subcategory}. It consists of the objects $Y\in\mcG$ such that $\Hom_\mcG(T,Y)=0=\Ext_\mcG^1(T,Y)$, for all $T\in\mcT$.
 
A prototypical example of Grothendieck category is the one given as follows. Take any (skeletally) small pre-additive category $\mcA$. A \emph{(right)  $\mcA$-module} is any additive functor $M\dd\mcA^\op\la\Ab$ . The category with the $\mcA$-modules as objects and the natural transformations between them as morphisms, will be denoted by $\Mod\mcA$. Any category equivalent to $\Mod\mcA$, for some small pre-additive category $\mcA$,  will be called a \emph{module category}. The \emph{Yoneda functor} $\y\dd\mcA\la\Mod\mcA$ takes $a\rightsquigarrow\y(a)=\mcA(?,a)$ and is fully faithful. It is well known that $\Mod\mcA$ is a Grothendieck category, where $\text{Im}(\y)=\{\y(a) \mid a\in\Ob(\mcA)\}$ is a set of finitely generated projective (whence finitely presented) generators (see, e.g., \cite[Theorem 3.1]{Mit} and \cite[Theorem 3.4.2]{Pop}). 
We will put $\modf\mcA:=\fp(\Mod\mcA)$ to denote the subcategory of finitely presented $\mcA$-modules. It consists of the $\mcA$-modules $M$ that admit an exact sequence $\coprod_{i=1}^m\y(a_i)\la\coprod_{j=1}^n\y(b_j)\la M\la 0$, for some finite families $(a_i)$ and $(b_j)$ of objects of $\mcA$, so the terminology is consistent with~\S\ref{subsec:abel-enough-proj}. 
 
The following generalized version of Gabriel-Popescu theorem (see, e.g., \cite{Mit2} or \cite[Theorems 1.1 and 1.2]{Low}) tells us that all Grothendieck categories appear as localizations of module categories:
 
\begin{prop}[Gabriel-Popescu Theorem] \label{prop.Gabriel-Popescu}
Let $\mcG$ be any category. The following assertions are equivalent:
 
\begin{enumerate}
\item $\mcG$ is a Grothendieck category.
\item There is a small pre-additive category $\mcA$ and a hereditary torsion class $\mcT$ in $\Mod\mcA$ such that $\mcG$ is equivalent to $(\Mod\mcA)/ \mcT$. 
\item $\mcG$ is abelian and there is a fully faithful functor $\iota\dd\mcG\la\Mod\mcA$, for some small pre-additive category $\mcA$, such that $\iota$ has an exact left adjoint. 
\end{enumerate}
 
In the situation of assertion~(3) the exact left adjoint $q$ induces an equivalence of categories $(\Mod\mcA)/\mcT\stackrel{\cong}{\la}\mcG$, where $\mcT=\Ker(q)$. 
\end{prop}

For our purposes in this paper, it will be useful to have sufficient conditions for  $(\Mod\mcA)/ \mcT$  to be locally finitely presented. The following result gives such conditions, even in a more general situation.

\begin{prop} \label{prop.locally-fp-quotient categories}
Let $\mcH$ be a locally finitely presented Grothendieck category and fix any set $\mcS$ of finitely presented generators. Let $\tau =(\mcT,\mcF)$ be a hereditary torsion pair in $\mcH$,  $q\dd\mcH\la\mcH/\mcT$ be the corresponding Gabriel localization functor and let $\mcG$ be the associated Giraud subcategory of $\mcH$. The following assertions are equivalent:

\begin{enumerate}
\item $\mcG$ is closed under taking direct limits in $\mcH$.
\item The section functor $\iota\dd\mcH/\mcT\la\mcH$ preserves direct limits.
\item The functor $q$ preserves finitely presented objects.
\item $q(\mcS)$ consists of finitely presented objects in $\mcH/\mcT$
\end{enumerate}

When these equivalent conditions hold, the torsion pair $\tau$ is of finite type and the category $\mcH/ \mcT$ is locally finitely presented, with  $\fp(\mcH/ \mcT)=\add(q(\fp(\mcH)))$. 
\end{prop}

\begin{proof}
Without loss of generality, we assume that $q\dd\mcH\la\mcG$ is a functor with $\mcG$ as codomain whose right adjoint $\iota\dd\mcG\la\mcH$ is the inclusion functor.

$(1)\Longleftrightarrow (2)$ This is clear.

$(3)\Longleftrightarrow (4)$ This follows immediately since the objects in $\fp(\mcH)$ are just cokernels of morphisms in $\add(\mcS)$.

$(2)\Longleftrightarrow (4)$ This is an instance of a general fact that a left adjoint originating in a locally finitely presented category preserves finite presentation if and only if the corresponding right adjoint preserves direct limits.

Indeed, consider $X\in\mcS$ and a direct system $(G_i)_{i\in I}$ in $\mcG$. Assertion (4) precisely says that the canonical morphism
\[ \varinjlim\Hom_{\mcG}(q(X),G_i) \la \Hom_{\mcG}(q(X),\varinjlim G_i) \]
is an isomorphism for every choice of $X$ and $(G_i)_{i\in I}$. Here, the direct limit on the right hand side is computed in $\mcG$. Taking the adjoint form, we obtain morphisms
\[ \varinjlim\Hom_{\mcH}(X,\iota(G_i)) \la \Hom_\mcH(X,\iota(\varinjlim G_i)). \]
Since $X$ is finitely presented in $\mcH$, the latter morphism is further bijective if and only if the canonical map
\begin{equation} \label{eq.locally-fp-quotient-2}
\Hom_{\mcH}(X,\varinjlim\iota(G_i)) \la \Hom_{\mcH}(X,\iota(\varinjlim G_i)).
\end{equation}
is an isomorphism. Now, since $X$ runs over a generating set, the morphisms~\eqref{eq.locally-fp-quotient-2} are bijective, for all $X\in\mcS$ and  all direct systems $(G_i)_{i\in I}$ in $\mcG$ if, and only if
\[ \varinjlim\iota(G_i) \la \iota(\varinjlim G_i) \]
is bijective for every $(G_i)_{i\in I}$, which is precisely assertion~(2).

Suppose now that the equivalent assertions (1)--(4) hold. Since each direct system $(F_i)_{i\in I}$ in $\mcF$ gives a direct system of short exact sequences
\[ (0\la F_i\stackrel{\eta_{F_i}}{\la}(\iota\circ q)(F_i)\la T_i\la 0)_{i\in I}, \]
it  follows that $\varinjlim F_i$ is a subobject of $\varinjlim (\iota\circ q)(F_i)$, and this one is an object in $\mcG$ by assertion~(1). Therefore $\varinjlim F_i\in\mcF$, so that $\tau$ is a torsion pair of finite type.

On the other hand $q(\mcS)$ is a set of finitely presented generators of $\mcH/\mcT$, thus showing that this latter category is locally finitely presented. Moreover if $Y\in\fp(\mcH/\mcT)$ and we express $\iota (Y)$ as a direct limit $\iota (Y)=\varinjlim X_\lambda$, for some direct system $(X_\lambda )_{\lambda\in\Lambda}$ in $\fp(\mcH)$, we get that $Y\cong (q\circ\iota )(Y)\cong\varinjlim q(X_\lambda)$. Since $Y$ is finitely presented, it is isomorphic to a direct summand of $q(X_\lambda)$, for some $\lambda\in\Lambda$. This gives the inclusion $\fp(\mcH/\mcT)\subseteq\add(q(\fp(\mcH)))$, the reverse inclusion being clear by assertion~(3). 
\end{proof}

At the end of Section~\ref{sec:representability}, we will also use a higher-cardinal analogue of finite presentability. Given a regular cardinal $\kappa$, we say that $X\in\mcG$ is \emph{$<\kappa$-presented} if $\Hom_\mcG(X,?)\dd\mcG\la\Ab$ preserves \emph{$\kappa$-direct limits}, i.e.\ colimits indexed by partially ordered sets whose each collection of $<\kappa$ elements has an upper bound.
It is a well-known consequence of the Gabriel-Popescu Theorem that every $X\in\mcG$ is $<\kappa$-presented for some regular cardinal $\kappa$ and that $\mcG$ is locally $<\kappa$-presented for some regular cardinal $\kappa$. The latter means that $\mcG$ has a set $\mcS$ of $<\kappa$-presented generators and, as in the finite case, the condition is equivalent to saying that the full subcategory $\kappa\pres(\mcG)$ of $<\kappa$-presented objects is skeletally small and each object of $\mcG$ is a $\kappa$-direct limit of objects in $\kappa\pres(\mcG)$.

If $\mcG$ is locally $<\kappa$-presented and $\lambda$ is any regular cardinal, then the class of $<\lambda$-presented objects is always closed under cokernels by~\cite[Proposition 1.16]{AR}. On the other hand, the full subcategory $\lambda\pres(\mcG)$ of $<\lambda$-presented objects is also closed under kernels and extensions and so it is an exact abelian subcategory of $\mcG$ for arbitrarily large cardinals $\lambda$. Concretely, this is true if $\kappa$ is sharply smaller than $\lambda$ in the sense of \cite[Definition 2.12]{AR} and if a skeleton of $\kappa\pres(\mcG)$ has $<\lambda$ morphisms (there are arbitrarily large such cardinals by~\cite[Example 2.13(6)]{AR}).
To see this, we remind the reader that the condition of being sharply smaller means that given any $\kappa$-directed poset $I$, each subset $J\subseteq I$ of cardinality $<\lambda$ is contained in a $\kappa$-directed subset $\hat{J}$ of cardinality $<\lambda$.
In this situation, an object $X$ of $\mcG$ is $<\lambda$-presented if, and only if, it is a direct summand of a direct limit $\varinjlim_I C_i$, where the $C_i$ are $<\kappa$-presented and $I$ is $\kappa$-directed set with $\lvert I\rvert<\lambda$ (see \cite[Remark 2.15]{AR}). 
Now, by the proof of~\cite[Theorem 1.46]{AR},
the generalized Yoneda functor $\y\dd\mcG\la[\kappa\pres(\mcG),\Set]$, $X\longmapsto\Hom(?,X)|_{\kappa\pres(\mcG)}$ is fully faithful and the essential image is closed under $\kappa$-direct limits in the target functor category.
Thus, thanks to~\cite[Example 1.31]{AR} and  the description of $<\lambda$-presented objects given above, an object $X\in\mcG$ is $<\lambda$-presented if, and only if, $\y X$ is $<\lambda$-presented in $[\kappa\pres(\mcG),\Set]$
if, and only if, the sum of the cardinalities of $\Hom_\mcG(C,X)$, where $C$ runs over the objects of a skeleton of $\kappa\pres(\mcG)$, is $<\lambda$.
The closure of $\lambda\pres(\mcG)$ under extensions and kernels in $\mcG$ then follows immediately.



\subsection{Triangulated categories---general notions}
\label{subsec:triangulated}

We refer the reader to \cite{N} for the precise definition of
\emph{triangulated category} and the basic facts about them (many of these, albeit with different terminology, can be found also in~\cite{HPS}). Here, we will denote the suspension functor by $?[1]\dd\mcD\la\mcD$. We will then put $?[0]=1_\mcD$ and $?[k]$
will denote the $k$-th power of $?[1]$, for each integer $k$.
(Distinguished) triangles in $\mcD$ will be denoted
$X\stackrel{u}{\la} Y\stackrel{v}{\la} Z\stackrel{w}{\la}X[1]$ or by $X\stackrel{u}{\la} Y\stackrel{v}{\la} Z\laplus$. It is well known that any morphism in the triangle determines the other vertex up to non-unique isomorphism. We will call $Z$ the \emph{cone} of $u$, written $\cone(u)$, and $X$ the \emph{cocone} of $v$, written $\text{cocone}(v)$. 

  A \emph{triangulated functor} between
triangulated categories is one that preserves triangles. The definition is in fact a little subtle in that the datum of a triangulated functor consists not only of a functor $F\dd\mcD\la\mcD'$, but also of a natural equivalence $F(?[1])\cong F(?)[1]$. The latter is, however, usually obvious from the context.

All through the rest of Section~\ref{sec:prelim},  $\mcD$ will be a triangulated category. When $I\subseteq\mathbb{Z}$ is a subset and $\mcS\subseteq\mcD$ is a subcategory, we will denote by $\mcS^{\perp_I}$ (resp. $_{}^{\perp_I}\mcS$) the subcategory of $\mcD$ consisting of the objects $Y$ such that $\Hom_\mcD(S,Y[k])=0$ (resp. $\Hom_\mcD(Y,S[k])=0$), for all $S\in\mcS$ and all integers $k\in I$. In this vein we have subcategories  $\mcS^{\perp_{>n}}$,  $\mcS^{\perp_{\geq n}}$, $\mcS^{\perp_\mathbb{Z}}$  and their symmetric counterparts.

Unlike the
terminology used for abelian categories, a class (resp. set) 
$\mcS\subseteq\Ob(\mcD)$ is called a \emph{class (resp. set) of
generators of $\mcD$} when $\mcS^{\perp_\mathbb{Z}}=0$.  In case $\mcD$  has
coproducts, an object $X\in\mcD$ is called \emph{compact} when the functor $\Hom_\mcD(X,?)\dd\mcD\la\Ab$ preserves coproducts. We denote by $\mcD^c$ the subcategory of compact objects. 
We will say that $\mcD$ is \emph{compactly 
generated} when it has a set of compact generators, in which case the subcategory $\mcD^c$ is  skeletally small. 

Recall that if $\mcD$ and $\mcA$ are a triangulated
and an abelian category, respectively, then an additive  functor
$H\dd\mcD\la\mcA$ is a \emph{cohomological
functor} when, given any triangle $X\la Y\la
Z\laplus$, one gets an induced long exact
sequence in $\mcA$:
\[
\cdots \la H^{n-1}(Z)\la H^n(X)\la
H^n(Y)\la H^n(Z)\la
H^{n+1}(X)\la \cdots,
\]
where $H^n:=H\circ (?[n])$, for each $n\in\mathbb{Z}$.
Such functors are also often called \emph{homological functors} and in that case one requires that triangles yield long exact sequences
\[
\cdots \la H_{n+1}(Z)\la H_n(X)\la
H_n(Y)\la H_n(Z)\la
H_{n-1}(X)\la \cdots,
\]
where $H_n:=H\circ(?[-n])$. We will use both  variants, depending on what will appear more natural or customary in the given context. Obviously, one has the identification $H_{-n} = H^n$.

Each representable functor $\Hom_\mcD(?,X)\dd\mcD^\op\la\Ab$
is cohomological. We will say that $\mcD$ \emph{satisfies Brown representability theorem} when $\mcD$ has coproducts and each cohomological functor $H\dd\mcD^\op\la\Ab$ that preserves products (i.e. that,  as a contravariant functor $\mcD\la\Ab$,  takes coproducts to products) is representable. 
Each compactly generated triangulated category satisfies  Brown representability theorem  (\cite[Theorem 8.3.3]{N}).

Given a triangulated category $\mcD$, a subcategory $\mathcal{E}$ will be called a \emph{suspended subcategory} when it is closed under taking extensions and $\mathcal{E}[1]\subseteq\mathcal{E}$, and \emph{cosuspended} when it is closed under taking extensions and $\mathcal{E}[-1]\subseteq\mathcal{E}$. If, in addition, we have $\mathcal{E}=\mathcal{E}[1]$, we will say that $\mathcal{E}$ is a \emph{triangulated subcategory}. A triangulated subcategory closed under taking direct summands is called a \emph{thick subcategory}. When the ambient triangulated category $\mcD$ has coproducts, a triangulated subcategory closed under taking arbitrary coproducts is called a \emph{localizing subcategory}. Note that such a subcategory is always thick  (see \cite[Lemma 1.4.9]{HPS} or the proof of \cite[Proposition 1.6.8]{N}, which also shows that idempotents split in any triangulated category with coproducts). 
In such case, given any class $\mcS$ of objects of $\mcD$, we will denote by  $\Loc_\mcD(\mcS)$  the smallest localizing subcategory containing $\mcS$. 

Recall that when $\mcE$ is a triangulated subcategory of the triangulated category $\mcD$, the localization of $\mcD$ with respect to the class of morphism $s$ in $\mcD$ with $\cone(s)\in\mcE$ (see~\S\ref{subsec:localization-func}) is called the \emph{Verdier quotient $\mcD/\mcE$} and the associated localization functor $q\dd\mcD\la\mcD/ \mcE$ is the \emph{Verdier quotient functor}. 
The category $\mcD/\mcE$ has a natural triangulated structure and $q$ is naturally a triangulated functor.


\subsection{\texorpdfstring{$t$-structures}{t-structures} in triangulated categories} \label{ssec.t-structures}

A \emph{$t$-structure} in $\mcD$ (see \cite[Section 1]{BBD}) is a pair $\mathbf{t} =(\mcU,\mcV)$ of full subcategories which satisfy the following properties:

\begin{enumerate}
\item[(i)] $\Hom_\mcD(U,V[-1])=0$, for all
$U\in\mcU$ and $V\in\mcV$;
\item[(ii)] $\mcU[1]\subseteq\mcU$ (or $\mcV[-1]\subseteq\mcV$);
\item[(iii)] For each $X\in\Ob(\mcD)$, there is a triangle $U\la X\la
W\laplus$ in $\mcD$, where
$U\in\mcU$ and $W\in\mcV[-1]$.
\end{enumerate}
It is easy to see, using basic properties of triangulated categories, that the objects $U$ and $W$ in the above
triangle are uniquely determined by $X$, up to a unique isomorphism, and thus
define functors
$\tau_{\mathbf{t}}^{\leq 0}\dd\mcD\la\mcU$ and
$\tau_{\mathbf{t}}^{>0}\dd\mcD\la\mcV[-1]$
which are right and left adjoints to the respective inclusion
functors. We call them the \emph{left} and \emph{right truncation functors} with respect to the given $t$-structure.
It immediately follows that $\mcV=\mcU^\perp
[1]$ and $\mcU={^\perp (\mcV[-1])}={^\perp
(\mcU^\perp)}$, that $\mcU$ is a suspended subcategory and $\mcV$ is cosuspended,
and that $\mcU,\mcV$ are both closed under summands in $\mcD$. We will call $\mcU$
and $\mcV$ the \emph{aisle} and the \emph{co-aisle} of
the $t$-structure. Note that, for each $n\in\mathbb{Z}$, the pair $(\mcU[n],\mcV[n])$ is also a $t$-structure, and the corresponding left and right truncation functors are denoted by $\tau_\mathbf{t}^{\leq -n}$ and $\tau_{\mathbf{t}}^{>-n}=:\tau_{\mathbf{t}}^{\geq -n+1}$.   If $\mcD'$ is a triangulated subcategory of $\mcD$, we will say that the $t$-structure $\mathbf{t}$ \emph{restricts to $\mcD'$} when $\mathbf{t}'=(\mcU\cap\mcD',\mcV\cap\mcD')$ is a $t$-structure in $\mcD'$. This is equivalent to say that $\tau_\mathbf{t}^{\leq 0}X$ (or $\tau_\mathbf{t}^{>0}X$) is in $\mcD'$, for all $X\in\mcD'$.

The full subcategory
$\mcH=\mcU\cap\mcV$ is called the \emph{heart} of the $t$-structure and it is an
abelian category, where the short exact sequences `are' the
triangles in $\mcD$ with the first three terms in $\mcH$.
Moreover, with the obvious abuse of notation,  the assignments
$X\rightsquigarrow (\tau_{\mathbf{t}}^{\leq 0}\circ\tau_{\mathbf{t}}^{\geq 0})(X)$ and $X\rightsquigarrow (\tau_{\mathbf{t}}^{\geq 0}\circ\tau_{\mathbf{t}}^{\leq 0})(X)$ define   naturally isomorphic
functors $\mcD\la\mcH$ which are
cohomological (see \cite{BBD}). We fix all through the paper a  functor $\H\dd\mcD\la\mcH$ naturally isomorphic to those two functors. The $t$-structure $\mathbf{t} =(\mcU,\mcV)$ will be called \emph{left (resp. right) non-degenerate} when $\bigcap_{k\in\mathbb{Z}}\mcU[k]=0$ (resp. $\bigcap_{k\in\mathbb{Z}}\mcV [k]=0$). It will be called \emph{non-degenerate} when it is left and right non-degenerate. A $t$-structure $\mathbf{t}=(\mcU,\mcV)$ such that $\mcU[1]=\mcU$, or equivalently $\mcV=\mcV[-1]$, will be called a \emph{semiorthogonal decomposition}. 

Suppose now that $\mcD$ has coproducts. If the co-aisle $\mcV$ is closed under taking coproducts, which is equivalent to say that the truncation functor $\tau_{\mathbf{t}}^{\leq 0}\dd\mcD\la\mcU$ preserves coproducts, then $\mathbf{t}$ is called a \emph{smashing $t$-structure}.
If $\mcS\subseteq\mcU$ is any class of objects, we shall say that the $t$-structure $\mathbf{t}$ is \emph{generated by $\mcS$} or that \emph{$\mcS$ is a class of generators} of $\mathbf{t}$ when $\mcV =\mcS^{\perp_{< 0}}$. We shall say that $\mathbf{t}$ is \emph{compactly generated} when it is generated by a set (i.e.\ \emph{not} a proper class) of compact objects. Note that such a $t$-structure is always smashing.

We now shortly discuss  the question of when a suspended subcategory is an aisle. Fix a suspended subcategory $\mcS$ of $\mcD$. By~\cite[\S1]{KV-aisles}, $\mcS$ is the aisle of a $t$-structure in $\mcD$ if, and only if, the inclusion functor $\mcS\la\mcD$ has a right adjoint. If $\mcD$ has coproducts and satisfies the Brown representability theorem (e.g.\ if $\mcD$ is well generated in the sense of~\S\ref{subsec:well-gen} below) and $\mcS$ is closed under coproducts, Neeman~\cite{N-t-str} has recently provided the following sufficient condition for the existence of the adjunction.
For any $X,Y\in\mcD$, we consider the slice category $X/\mcS/Y$ whose objects are pairs of composable morphisms $(X\overset{f}\to S\overset{g}\to Y)$ with $S\in\mcS$ and a morphism from $(X\overset{f}\to S\overset{g}\to Y)$ to $(X\overset{f'}\to S'\overset{g'}\to Y)$ is given by $h\colon S\la S'$ such that $f'=hf$ and $g=g'h$. If we denote by $H_\mcS(X,Y)$ the class of connected components of $X/\mcS/Y$ (which are the smallest subclasses of objects pairwise connected by zigzags of morphisms; a similar construction also appeared in~\cite[\S2.1]{BP10} in a different context), then it is proved in\cite[Proposition 1.15 and Discussion 1.16]{N-t-str} that $\mcS\la\mcD$ has a right adjoint if, and only if, $H_\mcS(X,Y)$ is a set (and not a proper class) for each pair $X,Y\in\mcD$. In particular, we can easily derive the following criterion which we later use in the proof of Proposition~\ref{prop.t-structure from pi}:

\begin{prop}\label{prop:aisle-by-Neeman}
Let $\mcD$ be a triangulated category with coproducts satisfying Brown representability theorem and let $\mcS\subseteq\mcD$ be a suspended subcategory closed under coproducts. 
Suppose that for each $X\in\mcD$, there is set $\mcS_X\subseteq\mcS$ such that each morphism $X\la S$, with $S\in\mcS$, factors through an object of $\mcS_X$. Then $\mcS$ is an aisle in $\mcD$.
\end{prop}

\begin{proof}
The cardinality of each $H_\mcS(X,Y)$ is clearly bounded by the
sum of the cardinalities of $\Hom_\mcD(X,S)\times\Hom_\mcD(S,Y)$, where
$S$ runs over $\mcS_X$.
\end{proof}


\subsection{Standard well generated triangulated categories}
\label{subsec:well-gen}

Next we recall some known generalizations of compactly generated triangulated categories and define a new convenient one. Let $\mcD$ be triangulated with coproducts. A \emph{perfect class of objects} in $\mcD$ is a class $\mcS$ such that, for any family $(f_i\colon X_i\la Y_i)_{i\in I}$ of morphisms,
\[ \Hom_\mcD(S,\coprod_{i\in I} f_i)\colon\Hom_\mcD(S,\coprod_{i\in I}X_i)\la\Hom_\mcD(S,\coprod_{i\in I}X_i) \]
is an epimorphism for all $S\in\mcS$ whenever the morphisms
\[ \Hom_\mcD(S,f_i)\colon\Hom_\mcD(S,X_i)\la\Hom_\mcD(S,X_i) \]
are such, for all $i\in I$ and $S\in\mcS$. An object X is \emph{perfect} when $\{X\}$ is a perfect set of objects.
We say that $\mcD$ is \emph{perfectly generated by $\mcS$} when $\mcS$ is a perfect set of generators.
On the other hand, given a regular cardinal $\kappa$, we say that an object $X$ is \emph{$\kappa$-small} if any morphism in $\mcD$ of the form $X\la\coprod_{i\in I}Y_i$ factors through a subcoproduct $\coprod_{i\in J}Y_i$ for some subset $J\subseteq I$ of cardinality $<\kappa$.

For any triangulated category $\mcD$ with coproducts, there exists a largest perfect class of $\kappa$-small objects which can be obtained as the union of all such classes.
We denote the full subcategory of $\mcD$ given by this class of objects by $\mcD^\kappa$ and call the objects contained in it \emph{$\kappa$-compact objects}. Observe that $\mcD^{\aleph_0}=\mcD^c$, and also that for any family of objects $(X_i)_{i\in I}$ in $\mcD^\kappa$ such that $I$ is of cardinality $<\kappa$, we also have $\coprod_{i\in I}X_i\in\mcD^\kappa$. Thus, our definition agrees with that in~\cite{K-wellgen} thanks to~\cite[Lemma 4]{K-wellgen}. This leads to the following important definition:

\begin{opr}[\cite{N}] \label{def.well generated}
A triangulated category $\mcD$ with coproducts is \emph{$\kappa$-well generated}, where $\kappa$ is a regular cardinal, when it is perfectly generated by a set of $\kappa$-small objects.
The category $\mcD$ is called  \emph{well generated} when it is $\kappa$-well generated, for some regular cardinal $\kappa$.
\end{opr}

In a well generated triangulated category $\mcD$, the subcategory $\mcD^\kappa$ is essentially small for each $\kappa$ and $\mcD=\bigcup_\kappa\mcD^\kappa$, where $\kappa$ runs through the class of regular cardinals (see~\cite[Lemma~5 and Corollary]{K-wellgen}).
 Furthermore, each well generated triangulated category satisfies Brown representability theorem  (\cite[Theorem 8.3.3]{N}).

Several results in this paper will be, however, stated for a hypothetically narrower class of triangulated categories:

\begin{opr} \label{def.std well generated}
A triangulated category $\mcD$ is called \emph{standard well generated} if it is equivalent to the Verdier quotient $\mcC/\Loc_\mcC(\mcS)$, where $\mcC$ is compactly generated triangulated and $\mcS\subseteq\Ob(\mcC)$ is a set of objects.
\end{opr}

As the terminology suggests, all standard well generated triangulated categories are well generated (see \cite[Theorem 1.14 and Remark 1.16]{N}), and no example of a well generated triangulated category which is not standard well generated is currently known.
This class of triangulated categories should be seen as a suitable triangulated analogue of locally presentable categories \cite{AR} in ordinary category theory on one hand and of locally presentable stable $\infty$-categories~\cite{Lur-HA} in higher category theory on the other hand.

Note that every compactly generated category is standard well generated, as is the unbounded derived category of any Grothendieck categories (cf.~\cite{AJS}).
Much more generally, any well generated algebraic~\cite[\S7.5]{K-Chicago} or topological~\cite{Schwede} triangulated category $\mcD$ is automatically standard well generated thanks to the main results of~\cite{Porta,Heider}.


\subsection{Purity and Milnor colimits in triangulated categories} \label{ssect.purity}
\label{subsec:Mcolim}

When $\mcD$ is a triangulated category with coproducts,  we will use the term \emph{Milnor colimit} of a sequence of morphisms $X_0\stackrel{x_1}{\la}X_1\stackrel{x_2}{\la}\cdots\stackrel{x_n}{\la}X_n\stackrel{x_{n+1}}{\la}\cdots$ for what in \cite{N} is called homotopy colimit. It will be denoted $\Mcolim(X_n)$, without reference to the $x_n$, and it is defined as the third term in the triangle
\begin{equation} \label{eq.Mcolim}
\coprod_{n\ge 0}X_n \overset{1-x}\la \coprod_{n\ge 0}X_n \la \Mcolim(X_n) \laplus.
\end{equation}
The components $f_i\dd X_i\la\Mcolim(X_n)$ of the second map in the triangle define a cocone in $\mcD$, 
\begin{equation} \label{eq.Mcolim cocone}
\vcenter{
\xymatrix{ X_0 \ar[r]^-{x_1} \ar@/_/@<-.5ex>[rrrrd]_-{f_0} & X_1 \ar[r]^{x_2} \ar@/_/[rrrd]|-\hole|-{f_1} & X_2 \ar@/_/@<.5ex>[rrd]|-\hole|-{f_2} \ar[r]^{x_3} & \dots \\
&&&& \Mcolim(X_n), }
}
\end{equation}
which is a weak colimit of the sequence by~\cite[Proposition 2.2.4]{HPS} (i.e.\/ for any other cocone $(g_i\dd X_i\la Y)_{i\ge0}$ there is a not necessarily unique morphism $g\dd\Mcolim(X_n)\la Y$ such that $g_i=gf_i$ for each $i\ge 0$).

In Section \ref{sect.universal copr-pres} we will outline a more general purity theory, valid on all standard well generated triangulated categories. But, for the moment, we remind the reader of the classical theory initiated in \cite{K}. A \emph{pure triangle} in a compactly generated triangulated category $\mcD$ is a triangle $X\stackrel{u}{\la}Y\stackrel{v}{\la}Z\stackrel{w}{\la}X[1]$ that satisfies any of the following equivalent conditions 

\begin{enumerate}
\item $u_*:=\Hom_\mcD(C,u)\dd\Hom_\mcD(C,X)\la\Hom_\mcD(C,Y)$ is a monomorphism, for all $C\in\mcD^c$, where $\mcD^c$ is the subcategory of compact objects;
\item $v_*:=\Hom_\mcD(C,v)\dd\Hom_\mcD(C,Y)\la\Hom_\mcD(C,Z)$ is an epimorphism, for all $C\in\mcD^c$;
\item $w_*:=\Hom_\mcD(C,w)\dd\Hom_\mcD(C,Z)\la\Hom_\mcD(C,X[1])$ is the zero map,  for all $C\in\mcD^c$.
\end{enumerate}

Any morphism $u$ (resp. $v$) appearing in such a triangle is called a \emph{pure monomorphism} (resp. \emph{pure epimorphism}).  A \emph{pure-injective object} of $\mcD$ is an object $Y$ such that the functor $\Hom_\mcD(?,Y)\dd\mcD\la\Ab$ takes pure monomorphisms to epimorphisms or, equivalently, pure epimorphisms to monomorphisms. 

A typical example of pure triangles appears when $X_0\stackrel{x_1}{\la}X_1\stackrel{x_2}{\la}\cdots\stackrel{x_n}{\la}X_n\stackrel{x_{n+1}}{\la}\cdots$ is a sequence of morphisms in $\mcD$.
Then the triangle~\eqref{eq.Mcolim} which defines $\Mcolim(X_n)$ is pure. A useful immediate consequence of the fact is that if $C\in\mcD^c$, then $\Hom_\mcD\big(C,\Mcolim(X_n)\big)\cong \varinjlim\Hom_\mcD(C,X_n)$.


\section{\texorpdfstring{$t$-structures}{t-structures} and localization of categories}
\label{sec.t-structures and localization}

In this section we establish basic general facts about the interaction of $t$-struct\-ures, Serre quotients of abelian categories and Verdier quotients of triangulated categories.
We in particular discuss methods how to turn degenerate $t$-structures to non-degenerate ones.
For the entire section, we denote by $\mcD$ a triangulated category with a $t$-structure $\mathbf{t}=(\mcU,\mcV)$, whose heart we denote by $\mcH$.
We start with an easy observation.

\begin{lem} \label{lem.H^0 is a localization}
The homological functor $\H\dd\mcD\la\mcH$ associated with the $t$-structure $\mathbf{t}$ is a localization functor.
\end{lem}

\begin{proof}
The functor $\H$ is obtained as the composition
\[ \mcD \stackrel{\ai0}\la \mcU \stackrel{{\coa0}_{|\mcU}}\la \mcH, \]
where the first functor has a fully faithful left adjoint $\mcU\subseteq\mcD$ and the second functor a fully faithful right adjoint, so both are localization functors. Thanks to Lemma~\ref{lem.composition and cancellation of localizations}, $\H$ is a localization functor as well.
\end{proof}

Recall that if $\mcP$ is an additive category, we use the notation $\widehat{\mcP}:=\modf\mcP$. As all of $\mcH$, $\mcU$ and $\mcD$ have weak kernels, the corresponding categories $\widehat{\mcH}$, $\widehat{\mcU}$ and $\widehat{\mcD}$ are abelian by Lemma~\ref{lem.coherent functors are abelian}
(in the case of $\mcU$, we construct a weak kernel of $f\dd U\la U'$ by completing it to triangle $Z\stackrel{u}{\la}U\stackrel{f}{\la}U'\laplus$ in $\mcD$ and composing $u$ with the truncation map $\tau_\mathbf{t}^{\leq 0}Z\la Z$).
Moreover, it is also well known that the Yoneda functor
\[ \y_\mcD\dd \mcD \la \widehat{\mcD} \]
is a universal homological functor in the sense that any other homological functor $H\dd \mcD\la\mcA$ with $\mcA$ abelian uniquely lifts to an exact functor $\widehat{H}\dd\widehat{\mcD}\la\mcA$, \cite[Lemma 2.1]{K}.
We may apply this in particular to the homological functor $\H\dd\mcD\la\mcH$ to obtain a commutative diagram
\begin{equation} \label{eq.H tilde}
\vcenter{
\xymatrix{
\mcD \ar[rr]^-{\y_\mcD} \ar[drr]_-{\H} &&
\widehat{\mcD} \ar[d]^-{\widehat{\H}} \\
&& \mcH.
}
}
\end{equation}

We will focus on the exact functor $\widehat{\H}$ now. For the context, we record the following straightforward observation which will be illuminating also later.

\begin{lem} \label{lem.factorization of exact functors}
Every exact functor $F\dd\mcA\la\mcB$ of abelian categories factors as $F=J\circ Q$, where $Q\dd\mcA\la\mcB'$ is a Serre quotient functor and $J\dd\mcB'\la\mcB$ is an exact faithful functor. This factorization is essentially unique in the sense that any other such factorization $F=J'\circ Q'$ induces a commutative diagram
\[
\xymatrix@R=1em{
& \mcB' \ar[dr]^-{J} \ar[dd]_-{\simeq} \\
\mcA \ar[ur]^-{Q} \ar[dr]_-{Q'} && \mcB, \\
& \mcB'' \ar[ur]^-{J'} \\
}
\]
where the vertical arrow is an equivalence.
\end{lem}

\begin{proof}
Regarding the existence, we simply put $\mcB'=\mcA/\Ker(F)$ and denote by $Q$ the localization functor and by $J\dd\mcB'\la\mcB$ the induced exact functor. Any morphism $f\dd X\la Y$ in $\mcB'$ is represented by a morphism $f'\dd X'\la Y'$ in $\mcA$ such that $X'\subset X$ is a subobject with $X/X'\in\Ker(F)$ and $Y'$ is a factor of $Y$ modulo a subobject in $\Ker(F)$. If $J(f)$ vanishes, so does clearly $JQ(f')=F(f')$. Since $F$ is exact, this implies that $\Img(f')\in\Ker(F)$ and that $Q(f')=0$. Since $f$ and $Q(f')$ are isomorphic in $\mcB'$, we infer that $f=0$ and $J$ is faithful.
	
Finally, observe that if $F=J\circ Q$ is any factorization with $Q$ a Serre quotient and $J$ faithful, we must have $\Ker(Q)=\Ker(F)$. The uniqueness of the factorization then follows from the universal property of the Serre quotient.
\end{proof}

The point with $\widehat{\H}$ is that the second part in the factorization from Lemma~\ref{lem.factorization of exact functors} is trivial---$\widehat{\H}$ itself is a localization functor.

\begin{prop} \label{prop.abelianized homology is Serre quotient}
The exact functor $\widehat{\H}\dd\widehat{\mcD}\la\mcH$ is a Serre quotient functor.
\end{prop}

\begin{proof}
We factorize $\widehat{\H}\dd\widehat{\mcD}\la\mcH$ into a composition of three localization functors with fully faithful adjoints as follows:	
\begin{equation} \label{eq.decomposition of widehat H}
\widehat{\mcD} \la \widehat{\mcU} \la \widehat{\mcH} \stackrel{C}\la \mcH.
\end{equation}
The fact that the composition is a localization functor follows by Lemma~\ref{lem.composition and cancellation of localizations}, and since $\widehat{\H}$ is exact, it is a Serre quotient functor.

Let us explain what functors we compose. The first two are obtained from $\ai0\dd\mcD\la\mcU$ and ${\coa0}_{|_\mcU}\dd\mcU\la\mcH$, respectively, using Lemma~\ref{lem.coherent functors are abelian}. The corresponding inclusions $\mcH\subseteq\mcU\subseteq\mcD$ lift to fully faithful functors, which we will by abuse of notation consider as inclusions $\widehat{\mcH}\subseteq\widehat{\mcU}\subseteq\widehat{\mcD}$. As we can also lift natural transformations and, in particular, the adjunction units and counits, the inclusions will be the correspoding adjoints of the first two functors in~\eqref{eq.decomposition of widehat H}.

Finally, the functor $C\dd\widehat{\mcH}\la\mcH$ is left adjoint to the fully faithful Yoneda embedding $\y_\mcH\dd\mcH\la\widehat{\mcH}$. Given any $f\dd X\la Y$ in $\mcH$, $C$ sends the cokernel of
\[ \y(f)\dd \y(X) \la \y(Y) \]
in $\widehat{\mcH}$ to $\Coker f\in\mcH$ (see also~\cite[\S3]{A-coherent},
$C$ is known to be exact and $C=\widehat{1_\mcH}$ in the notation of Lemma~\ref{lem.coherent functors are abelian}).
\end{proof}

The latter proposition has a drawback, however---$\widehat{\H}$ need not be an adjoint functor and thus is out of the scope of Lemma~\ref{lem.detect localization}. This can be often remedied if we focus our attention only on the aisle or the co-aisle. 

\begin{prop} \label{prop.left Kan extension for the aisle}
The unique extension $\widetilde{\H}\dd\widehat{\mcU}\la\mcH$ of $\H_{|\mcU}\dd\mcU\la\mcH$, in the sense of Lemma~\ref{lem.coherent functors are abelian}, is a Serre quotient functor left adjoint to the restriction ${\y_\mcU}_{|\mcH}\dd\mcH\la\widehat{\mcU}$ of the Yoneda functor $\y_\mcU$. 
In particular, the following square commutes up to natural equivalence for both the left and the right adjoints:
\[
\xymatrix@R=3em{
\mcU \ar[d]_-{\y_\mcU} \ar@/^/[rr]^-{\H_{|\mcU}} \ar@{}[rr]|\perp && \mcH \ar@{=}[d] \ar@/^/[ll]^-{\operatorname{inc}} \\
\widehat{\mcU} \ar@/^/[rr]^-{\widetilde{\H}} \ar@{}[rr]|\perp && \mcH. \ar@/^/[ll]^-{{\y_\mcU}_{|\mcH}}
}
\]
\end{prop}

\begin{proof}
Consider the adjunctions
\[
\xymatrix{ 
\widehat{\mcU} \ar@/^/[r]^-{\widehat{\H}} \ar@{}[r]|\perp &
\widehat{\mcH} \ar@/^/[r]^-{C} \ar@{}[r]|\perp \ar@/^/[l]^-{\widehat{\operatorname{inc}}} &
\mcH. \ar@/^/[l]^-{\y_\mcH}
}
\]
studied in the proof of Proposition~\ref{prop.abelianized homology is Serre quotient}. The right adjoints are both fully faithful and clearly compose to ${\y_\mcU}_{|\mcH}$. Just by unraveling the definitions, one also checks that $C\circ\widehat{\H}\circ\y_\mcU\cong\H$. Since $C\circ\widehat{\H}$ is also right exact, it follows that coincides with the essentially unique functor $\widetilde{\H}$ given by Lemma~\ref{lem.coherent functors are abelian} and it is a localization functor by Lemma~\ref{lem.detect localization}.

It remains to prove that $\widetilde{\H}$ is exact. Suppose that $M\in\widehat{\mcU}$ and $g\dd U\la U''$ is a map in $\mcU$ such that $\Hom_\mcU(?,U) \la \Hom_\mcU(?,U'') \la M \la 0$ is exact. We may complete $g$ to a triangle
\[ X \stackrel{f}\la U \stackrel{g}\la U'' \laplus \]
and consider the truncation morphism $\varepsilon\dd\ai{0}X\la X$. Then
\[ \Hom_\mcU(?,\ai{0}X) \stackrel{f_*\varepsilon_*}\la \Hom_\mcU(?,U) \stackrel{g_*}\la \Hom_\mcU(?,U'') \la M \la 0 \]
is a projective presentation of $M$ in $\widehat{\mcU}$ and if we apply $\widetilde{\H}$, we obtain the sequence
\[ \H(\ai{0}X) \stackrel{\H(f\varepsilon)}\la \H(U) \stackrel{\H(g)}\la \H(U'') \la \widetilde{\H}(M) \la 0 \]
in $\mcH$, which is exact since $\H$ is homological and $\H(\varepsilon)$ is an isomorphism. The exactness of $\widetilde{\H}$ then follows by the next lemma.
\end{proof}

\begin{lem} \label{lem.rightexact-exact}
Let $F\dd\mcA\la\mcB$ be a right exact functor between abelian categories and suppose that $\mcA$ has enough projective objects. The following assertions are equivalent:
	
\begin{enumerate}
\item $F$ is exact.
\item For each exact sequence $P'\stackrel{f}{\la}P\stackrel{g}{\la}P''$ in $\mcA$ whose all terms are projective, the sequence  $F(P')\stackrel{F(f)}{\la}F(P)\stackrel{F(g)}{\la}F(P'')$ is exact in $\mathcal{B}$. 
\item Each object $A\in\mcA$ admits a projective presentation $P'\stackrel{f}{\la}P\stackrel{g}{\la}P''\stackrel{\pi}{\la}A\rightarrow 0$ such that the sequence  $F(P')\stackrel{F(f)}{\la}F(P)\stackrel{F(g)}{\la}F(P'')$ is exact in $\mcB$. 
\end{enumerate}
\end{lem}
\begin{proof}
Condition (2) (resp.\ (3)) holds if, and only if, the  first left derived functor $\mathbb{L}_1F$ vanishes, which is tantamount to say that $F$ is exact. 
\end{proof}

Finally, we discuss another natural question, which is important later. We can ask to which extent the homological functor $\H\dd\mcD\la\mcH$ determines the $t$-structure $\mathbf{t}=(\mcU,\mcV)$. In general, there may be several $t$-structures with the same homological functor (e.g. any semiorthogonal decomposition of $\mcD$ has the same and trivial homological functor). However, the $t$-structure is clearly determined by $\H$ if it is non-degenerate as then

\begin{align*}
\mcU &= \{ U \in \mcD \mid H_{\mathbf{t}}^i(U) = 0 \textrm{ for all } i>0 \}, \\
\mcV &= \{ V \in \mcD \mid H_{\mathbf{t}}^i(V) = 0 \textrm{ for all } i<0 \}. \\
\end{align*}

Indeed, clearly $\mcU\subseteq\{ U \in \mcD \mid H_{\mathbf{t}}^i(U) = 0 \textrm{ for all } i>0 \}$ and, on the other hand, if $H_{\mathbf{t}}^i(U) = 0$ for all $i>0$, then $\tau_\mathbf{t}^{>0}U\in\bigcap_{k\in\mathbb{Z}}\mcV[k]=0$ by~\cite[Lemma 3.3]{NSZ}, so $U\cong\tau_\mathbf{t}^{\le0}U\in\mcU$.
The other equality is dual.

Here we will show how to reduce a $t$-structure to a non-degenerate one. We call the full subcategory
\[ \mcN_\mathbf{t} = \{ X\in\mcD \mid H^i_{\mathbf{t}}(X) = 0 \textrm{ for all } i\in\mathbb{Z} \} \]
the \emph{degeneracy class} of $\mathbf{t}$. Clearly $\mcN_\mathbf{t}$ is a thick subcategory of $\mcD$. Moreover, the homological functor $\H$ lifts as
\[
\vcenter{
\xymatrix{
\mcD \ar[rr]^-{q} \ar[drr]_-{\H} &&
\mcD/\mcN_{\mathbf{t}} \ar[d]^-{\overline{\H}} \\
&& \mcH.
}
}
\]

\noindent
We will show that actually $\big(q(\mcU),q(\mcV)\big)$ is a (non-degenerate) $t$-structure in $\mcD/\mcN_\mathbf{t}$ and the functor $\overline{\H}$ is the corresponding homological functor.

\begin{lem} \label{lem.so decomposition of the degeneracy}
We have equalities $\mcU\cap\mcN_\mathbf{t}=\bigcap_{n\in\mathbb{Z}}\mcU[n]$, 
$\mcV\cap\mcN_\mathbf{t}=\bigcap_{n\in\mathbb{Z}}\mcV[n]$ and the pair
$(\mcU\cap\mcN_\mathbf{t},\mcV\cap\mcN_\mathbf{t})$ is a semiorthogonal decomposition of $\mcN_\mathbf{t}$.
\end{lem}

\begin{proof}
The equalities follow from \cite[Lemma 3.3]{NSZ}. For the last statement, note that for any $X\in\mcN_\mathbf{t}$, we have $\ai0(X),\coa1(X)\in\mcN_\mathbf{t}$. Hence, $\mathbf{t}$ restricts to a $t$-structure in $\mcN_\mathbf{t}$. Since both $\mcU\cap\mcN_\mathbf{t}$ and $\mcV\cap\mcN_\mathbf{t}$ are thick subcategories by the first part, the restricted $t$-structure is in fact a semiorthogonal decomposition.
\end{proof}

\begin{rem}
The objects in $\mcU\cap\mcN_\mathbf{t}=\bigcap_{n\in\mathbb{Z}}\mcU[n]$ are called \emph{$\infty$-connective} in \cite[Definition C.1.2.12]{Lur-SAG}.
\end{rem}

Now we can prove an even more general version of the degeneracy reduction result for $t$-structures.

\begin{prop} \label{lem.localization of t-structure}
Let $\mcD$ be a triangulated category with a $t$-structure $\mathbf{t}=(\mcU,\mcV)$ and
\[
\mcN_\mathbf{t}=\{ X\in\mcD \mid H^i_{\mathbf{t}}(X) = 0 \textrm{ for all } i\in\mathbb{Z} \}.
\]
If $\mcN'\subseteq\mcN_\mathbf{t}$ is a triangulated subcategory such that $\mathbf{t}$ restricts to a semiorthogonal decomposition of $\mcN'$ (this in particular applies to $\mcN'$ chosen as one of $\mcN_\mathbf{t}$, $\bigcap_{n\in\mathbb{Z}}\mcU[n]$ or $\bigcap_{n\in\mathbb{Z}}\mcV[n]$) and if we denote by $q\dd\mcD\la\mcD/\mcN'$ the Verdier quotient functor, then
$\overline{\mathbf{t}}=\big(q(\mcU),q(\mcV)\big)$ is a $t$-structure in $\mcD/\mcN'$ whose homological functor is, up to postcomposition with an equivalence, the unique one which fits into the commutative diagram
\[
\vcenter{
\xymatrix{
\mcD \ar[rr]^-{q} \ar[drr]_-{\H} &&
\mcD/\mcN' \ar[d]^-{H^0_{\overline{\mathbf{t}}}} \\
&& \mcH.
}
}
\]
\end{prop}

\begin{proof}
To prove that $\overline{\mathbf{t}}=\big(q(\mcU),q(\mcV)\big)$ is a $t$-structure, we only need to show that $\Hom_{\mcD/\mcN'}(U,V[-1]) = 0$ for each $U\in\mcU$ and $V\in\mcV$. The closure properties of $q(\mcU)$ and $q(\mcV)$ and the truncation triangles are inherited from $\mathbf{t}$ in $\mcD$.

To this end, suppose that $fs^{-1}\dd U\la V[-1]$ is a fraction representing a morphism in $\mcD/\mcN'$ as in~\cite[\S2.1]{N}, where $s\dd X\la U$ is a map in $\mcD$ whose cocone $N$ belongs to $\mcN'$. Then we truncate $N$ using the semiorthogonal decomposition of $\mcN'$ induced by $\mathbf{t}$ and, by the octahedral axiom, we obtain a commutative diagram in $\mcD$
\[
\xymatrix{
U_\infty \ar@{=}[r] \ar[d] & U_\infty \ar[d] \\
N \ar[r] \ar[d] & X \ar[r]^-s \ar[d]_-{g} & U \ar[r] \ar@{=}[d] & N[1] \ar[d] \\
V_\infty \ar[r] \ar[d] & Y \ar[r]_-{s'} \ar[d] & U \ar[r] & V_\infty[1] \\
U_\infty[1] \ar@{=}[r] & U_\infty[1] \\
}
\]
with triangles in rows and columns, $U_\infty\in\bigcap_{n\in\mathbb{Z}}\mcU[n]$ and
$V_\infty\in\bigcap_{n\in\mathbb{Z}}\mcV[n]$. As $\Hom_\mcD(U_\infty,V[-1])=0$, the morphism $f\dd X\la V[-1]$ factors through $g$ and $fs^{-1}=f'(s')^{-1}$ in $\mcD/\mcN'$ for some morphism $f'\dd Y\la V[-1]$ in $\mcD$. On the other hand, we have $\Hom_\mcD(U,V_\infty)=0$, so $s'$ splits and and if $t\dd U\la Y$ is a section, then $f'(s')^{-1}=f't(s't)^{-1}=f't$. However, the latter is a morphism from $U$ to $V[-1]$ in $\mcD$ and it vanishes since $(\mcU,\mcV)$ is a $t$-structure in $\mcD$.

Let us denote the heart of $\overline{\mathbf{t}}$ by $\overline{\mcH} := q(\mcU)\cap q(\mcV)$. The above argument also shows that $q_{|\mcH}\dd\mcH\la\overline{\mcH}$ is a full functor. If $f\dd H_1\la H_2$ is a morphism in $\mcH$ such that $q(f)=0$, then $f$ factors through some $N\in\mcN'$ and, since $H_1\in\mcU$, also through $\ai0(N)\in\bigcap_{n\in\mathbb{Z}}\mcU[n]$. Since $\Hom_\mcD(\mcU[1],H_2)=0$, it follows that $f$ vanishes already in $\mcH$ and that $q_{|\mcH}$ is faithful. Finally, since the truncation triangles for $\overline{\mathbf{t}}$ coincide with those for $\mathbf{t}$ in $\mcD$, we have $\H(X)\cong X$ in $\mcD/\mcN'$ for each $X\in\overline{\mcH}$. Thus, $q_{|\mcH}\dd\mcH\la\overline{\mcH}$ is essentially surjective as well and the last diagram from the statement commutes.
\end{proof}

\begin{rem} \label{rem.Luries left completion}
A different method of getting rid of the degeneracy of a $t$-structure was developed by Lurie, but he needed to work in the context of stable $\infty$-categories (in particular, he needed a full model for the triangulated category $\mcD$).

If $\mathbf{t}=(\mcU,\mcV)$ is a $t$-structure, he takes instead of $q\dd\mcD\la\mcD/\bigcap\mcU[n]$ the so called left completion $\lambda\dd\mcD\la\mcD'$ of $\mcD$ at $\mcU$. There is an induced $t$-structure $\mathbf{t'}=(\mcU',\mcV')$ in $\mcD'$ and $\lambda$ induces an equivalence $\mcV\simeq\mcV'$. The advantages over the Verdier quotient are that

\begin{enumerate}
\item $\mcD'$ is always locally small provided that $\mcD$ is such (for the Verdier quotient extra assumptions seem necessary, cf.~\cite[Proposition C.3.6.1]{Lur-SAG}),
\item $\mcD'$ can be recovered from the triangulated subcategory $\mcD^+\subseteq\mcD$ of objects which are left bounded with respect to $\mathbf{t}$.
\end{enumerate}

\noindent
Similarly, one can perform a right completion. We refer to~\cite[\S1.2.1]{Lur-HA}.
\end{rem}


\section{Homological functors from \texorpdfstring{$t$-generating}{t-generating} classes}
\label{sec:t-generating}

In the last section we studied the interaction of a $t$-structure $\mathbf{t}=(\mcU,\mcV)$ with the Yoneda functor $\y_\mcU\dd \mcU \la \widehat{\mcU}$ (Proposition\ref{prop.left Kan extension for the aisle}). In the sequel, it will be much more efficient to study homological functors of the form $h_\mcP\dd\mcU\la\widehat{\mcP}$ obtained by composing $\y_\mcU$ with the restriction to a suitable full subcategory $\mcP\subseteq\mcU$. A similar approach played a prominent role in the study of localization theory for triangulated categories~\cite{K,N,K-triang}, but it is in fact also an important technique in representation theory of finite dimensional algebras. Here, we establish basic facts about the interaction of restricted Yoneda functors with $t$-structures.

First of all, however, we note a basic lemma which is of use throughout the rest of the paper. It among others illustrates why precovering classes were called contravariantly finite in~\cite{ASm}.

\begin{lem} \label{lem.restricted Yoneda fundamental}
Let $\mcD$ be an additive category with weak kernels and $\mcP\subseteq\mcD$ a precovering full subcategory. Then
\begin{align*}
\res\dd \widehat{\mcD} &\la \widehat{\mcP}, \\
(F\dd\mcD\to\Ab) &\rightsquigarrow F_{|\mcP}
\end{align*}
is a well-defined Serre quotient functor and it has a fully faithful left adjoint.
\end{lem}

\begin{proof}
This is essentially~\cite[Theorem 3.4]{K-functors-lfp}.
First of all, $\mcP$ has weak kernels---if $f\dd P_1\la P_0$ is a map in $\mcP$, we can take a weak kernel $k\dd K\la P_1$ is $\mcD$ and precompose it with a $\mcP$-precover. Hence, both $\widehat{\mcD}$ and $\widehat{\mcP}$ are abelian by Lemma~\ref{lem.coherent functors are abelian}.
%
%
Secondly, if $M\dd\mcD^\op\la\Ab$ is a finitely presented functor, $M_{|\mcP}\dd\mcP^\op\la\Ab$ is as well by~\cite[Lemma 3.2]{K-functors-lfp}.
Finally, thanks to Proposition~\ref{prop.reconstructing} (see also~\cite[Lemma 2.6(1)]{K-functors-lfp}), the inclusion $\mcP\subseteq\mcD$ induces a fully faithful functor
\[ \iota\dd \widehat{\mcP}\simeq\stpMor(\mcP)\la\stpMor(\mcD)\simeq\widehat{\mcD} \]
and that it is left adjoint to $\res$ is shown by the following computation for each map $f\dd P_1\la P_0$ in $\mcP$ and each $M\in\widehat{\mcD}$:
\begin{equation} \label{eq.adjoint to restriction}
\begin{split}
\big[\iota(\Coker \Hom_\mcP(?,f)),M \big]
&\cong \big[\Coker\Hom_\mcD(?,f),M \big] \\
&\cong \Ker\big[\Hom_\mcD(?,f), M\big] \\
&\cong \Ker M(f) \\
&\cong \Ker\big[\Hom_\mcP(?,f), \res(M)\big] \\
&\cong \big[\Coker\Hom_\mcP(?,f), \res(M)\big].
\end{split}
\end{equation}
Here, the square brackets denote Hom-functors in $\widehat{\mcD}$ and $\widehat{\mcP}$. Hence $\res$ is a localization functor by Lemma~\ref{lem.detect localization} and, since it is clearly exact, it is even a Serre quotient functor.
\end{proof}

Next we define the class of full subcategories which satisfy appropriate compatibility condition with aisles or co-aisles of $t$-structures.

\begin{opr} \label{def.t-generating}
Let $\mcU$ be a suspended subcategory of a triangulated category $\mcD$. Then a full subcategory $\mcP\subseteq\mcU$
is called \emph{$t$-generating} in $\mcU$ if it is precovering and each $U\in\mcU$ admits a triangle of the form
\begin{equation} \label{eq:t-generating tria}
U' \la P \stackrel{p}\la U \laplus,
\end{equation}
with $U'\in\mcU$ and $P\in\mcP$.

If $\mcV\subseteq\mcD$ is a cosuspended subcategory, \emph{$t$-cogenerating} subcategories of $\mcV$ are defined dually.
\end{opr}

The terminology is motivated by~\cite[Definition C.2.1.1]{Lur-SAG}, where a notion of generator is defined in the context of prestable $\infty$-categories. By~\cite[Proposition C.1.2.9]{Lur-SAG}, prestable $\infty$-categories with finite limits are precisely enhancements of aisles of $t$-structures in the world of $\infty$-categories, and the reader may use the following lemma (see also~\cite[Proposition 1.2.3(6)]{Bo}) to match our Definition~\ref{def.t-generating} with the one of Lurie.

\begin{lem} \label{lem:t-generating tria}
Let $\mathbf{t}=(\mcU,\mcV)$ be a $t$-structure in a triangulated category $\mcD$ and
\[ U_2 \la U_1 \stackrel{p}\la U_0 \laplus, \]
a triangle in $\mcD$ with $U_0,U_1\in\mcU$. Then $U_2\in\mcU$ if and only if $\H(p)\dd\H(U_1)\la\H(U_0)$ is an epimorphism in the heart of $\mathbf{t}$.
\end{lem}

\begin{proof}
We always have $U_2[1] \in \mcU$, since $\mcU$ is closed under taking mapping cones, and also the following exact sequence in the heart
\[ \H(U_1) \stackrel{\H(p)}\la \H(U_0) \la \H(U_2[1]) \la \H(U_1[1])=0. \]
Now clearly $U_2\in\mcU$ if and only if $\H(U_2[1])\cong \tau_\mathbf{t}^{\geq 1}(U_2)[1]=0$ if and only if $\H(p)$ is an epimorphism in the heart.
\end{proof}

The latter lemma also has a more direct consequence which relates the two conditions imposed on $\mcP$ in Definition~\ref{def.t-generating} (i.e.\ the existence of precovers and the existence of triangles~\eqref{eq:t-generating tria}).

\begin{lem} \label{lem:t-generating precovers}
Let $\mcP\subseteq\mcU$ be a $t$-generating subcategory and suppose that $U' \la P \stackrel{p}\la U \laplus$ is a triangle in the ambient triangulated category such that $U\in \mcU$ and $p$ is a $\mcP$-precover. Then $U'\in\mcU$ (so the triangle is as in~\eqref{eq:t-generating tria}).
\end{lem}

\begin{proof}
Since $\mcP$ is $t$-generating in $\mcU$, there exists for the chosen $U$ some triangle
\[ U'' \la P' \stackrel{p'}\la U \laplus \]
with $U''\in\mcU$ and $P'\in\mcP$ (but $p'$ may not be a $\mcP$-precover). Since $p$ is a $\mcP$-precover, we have a factorization $p'=p\circ f$ for some $f\dd P'\la P$, and hence also $\H(p')=\H(p)\circ\H(f)$. Now $\H(p')$ is an epimorphism in the heart by Lemma~\ref{lem:t-generating tria} and so must be $\H(p)$ by the factorization. It remains to apply Lemma~\ref{lem:t-generating tria} again.
\end{proof}

The main result of the section is the following extension of Proposition~\ref{prop.left Kan extension for the aisle}. The added degree of freedom---the possibility to choose the class $\mcP$---is very important as we shall see later. It often happens that $\widehat{\mcP}$ for suitable $\mcP$ is a much smaller and a more tractable category than $\widehat{\mcU}$.

\begin{thm} \label{thm.t-generating left Kan extensions}
Let $\mathbf{t}=(\mcU,\mcV)$ be a $t$-structure in the triangulated category $\mcD$, let $\mcP\subseteq\mcU$ be a  precovering additive  subcategory and denote by $\y_\mcP$ the restricted Yoneda functor 
 \begin{align*}
\y_\mcP\dd \mcU &\la \widehat{\mcP}, \\
U &\rightsquigarrow \Hom_\mcU(?,U)_{|\mcP}.
\end{align*} 
The following assertions are equivalent:

\begin{enumerate}
\item The functor $\H\dd\mcU\longrightarrow\mcH$ factors as a composition $\mcU\stackrel{\y_\mcP}{\la}\widehat{\mcP}\stackrel{F}{\la}\mcH$, for some right exact functor $F$.
\item $\mcP$ is a $t$-generating subcategory of $\mcU$
\end{enumerate}

In such case $F$ is a Serre quotient functor and $G={\y_\mcP}_{|\mcH}\dd\mcH\la\widehat{\mcP}$ is its fully faithful right adjoint. In other words, we have the following square which commutes up to natural equivalence for both the left and the right adjoints:
\[
\xymatrix@R=3em{
\mcU \ar[d]_-{\y_\mcP} \ar@/^/[rr]^-{\H} \ar@{}[rr]|\perp && \mcH \ar@{=}[d] \ar@/^/[ll]^-{\operatorname{inc}} \\
\widehat{\mcP} \ar@/^/[rr]^-{F} \ar@{}[rr]|\perp && \mcH. \ar@/^/[ll]^-{G}
}
\]
\end{thm}

\begin{proof}
Note that $\y_\mcP$ can be factored as the composition $\mcU\stackrel{\y}{\la}\widehat{\mcU}\stackrel{\res}{\la}\widehat{\mcP}$. 
Let $\widetilde{\H}\dd\widehat{\mcU}\la\mcH$ be the Serre quotient functor given by Proposition~\ref{prop.left Kan extension for the aisle}.

If we have a factorization as in assertion~(1), then $\widetilde{\H}$ is naturally isomorphic to $F\circ\res$ by Lemma~\ref{lem.coherent functors are abelian}. Hence, condition~(1) is equivalent to saying that $\widetilde{\H}$ factors through $\res\dd\widehat{\mcU}\la\widehat{\mcP}$, something that happens exactly when $\Ker(\res)\subseteq\Ker(\widetilde{\H})$. Note that in that case the induced functor $F\dd\widehat{\mcP}\la\mcH$ is a Serre quotient functor. Indeed it is a a localization functor by Lemma~\ref{lem.composition and cancellation of localizations}(2) since $\res$ and $\widetilde{\H}$ are such. Moreover, since $\res$ is a Serre quotient functor with a fully faithful left adjoint, any exact sequence $\varepsilon\dd\; 0\la L\la M\la N\la 0$ in $\widehat{\mcP}$ lifts to an exact sequence $\varepsilon'\dd\; 0\la L'\la M'\la N'\la 0$, and the exactness of $F(\varepsilon)$ follows from that of $\widetilde{\H}(\varepsilon')$.

Suppose now that $\mcP$ is $t$-generating and take any morphism $f\dd U_1\la U_0$ in $\mcU$ such that $M:=\Coker(\y(f))\in\Ker(\res)$. Recall that any object of $\widehat{\mcU}$ is of the form $\Coker(\y(f))$ for some $f\dd U_1\la U_0$, and note that $M\in\Ker(\res)$ if and only if $(\y U_1)_{|\mcP}=\Hom_\mcU(?,U_1)_{|\mcP}\la \Hom_\mcU(?,U_0)_{| \mcP}=(\y U_0)_{| \mcP}$ is an epimorphism. Hence, any chosen $\mcP$-precover $p\dd P\la U_0$ factors thorough $f$. Consequently, $\H(p)$ factors through $\H(f)$ and, since $\H(p)$ is an epimorphism in the heart of $\mathbf{t}$ by Lemmas~\ref{lem:t-generating tria} and~\ref{lem:t-generating precovers}, so is $\H(f)$. It follows that $\widetilde{\H}(M)=\Coker(\H(f))=0$. This proves that $\Ker(\res)\subseteq\Ker(\widetilde{\H})$ and, by the above discussion, also assertion~(1).

Suppose, conversely, that~(1) holds, or equivalently $\Ker(\res)\subseteq\Ker(\widetilde{\H})$. Let $p\dd P'\la U$ be any $\mcP$-precover, where $U\in\mcU$. We then have that $N := \Coker(\y(p))\in\Ker(\res)\subseteq\Ker(\widetilde{\H})$. That is, we have $0=\widetilde{\H}(N)=\Coker (\H(p))$, so that $\H(p)$ is an epimorphism in $\mcH$. Then $\mcP$ is $t$-generating by Lemma~\ref{lem:t-generating tria}.

It remains to prove the final assertion.
The adjunction $(F,G)\dd \widehat{\mcP} \rightleftarrows \mcH$ simply arises as a composition of the two adjunctions
\[
\xymatrix{
	\widehat{\mcP} \ar@/^/[r]^-{\iota} \ar@{}[r]|\perp &
	\widehat{\mcU} \ar@/^/[r]^-{\widetilde{\H}} \ar@{}[r]|\perp  \ar@/^/[l]^-{\res} &
	\mcH. \ar@/^/[l]^-{{\y_\mcU}_{|\mcH}}
}
\]
given by Lemma~\ref{lem.restricted Yoneda fundamental} and Proposition~\ref{prop.left Kan extension for the aisle}, respectively.
Finally, the fact that $G$ is fully faithful follows by Lemma~\ref{lem.almost-fullyfaithful-Yoneda} below.
\end{proof}

\begin{lem} \label{lem.almost-fullyfaithful-Yoneda}
Let $\mathbf{t}=(\mcU,\mcV)$ be a $t$-structure in a triangulated category $\mcD$ and let $\mcP\subseteq\mcU$ be a  $t$-generating  subcategory. The map 
\[
\eta_{U,X}\dd\Hom_\mcU(U,X)\la\Hom_{\widehat{\mcP}}(\y_\mcP U,\y_\mcP X),
\]
induced by the functor $\y_\mcP\dd\mcU\la\widehat{\mcP}$, is bijective whenever $U\in\mcU$ and $X\in\mcH$.
\end{lem}

\begin{proof}
Let us fix $X\in\mcH$ all through the proof. By Yoneda's lemma $\eta_{P,X}$ is bijective whenever $P\in\mcP$. Let $U\in\mcU$ be arbitrary and, using that $\mcP$ is $t$-generating,  choose a triangle $U'\stackrel{u}{\la}P_0\stackrel{p}{\la}U\laplus$, where $p$ is a $\mcP$-precover and $U'\in\mcU$. Similarly we choose a $\mcP$-precover $P_1\stackrel{q}{\la}U'$ with cone in $\mcU[1]$. Note that then $\H(p)$ and $\H(q)$ are epimorphisms in $\mcH$ while $\y_\mcP(p)$ and $\y_\mcP(q)$ are epimorphisms in $\widehat{\mcP}=\modf\mcP$. Using that $\H\dd\mcD\la\mcH$ and $\y\dd\mcD\la\widehat{\mcD}$ are cohomological and that the restriction functor $\res\dd\widehat{\mcD}\la\widehat{\mcP}$ is exact, we get exact sequences $\H(P_1)\stackrel{\H(uq)}{\la}\H(P_0)\stackrel{\H(p)}{\la}\H(U)\la 0$ and  $\y_\mcP P_1\stackrel{\y_\mcP (uq)}{\la}\y_\mcP P_0\stackrel{\y_\mcP (p)}{\la}\y_\mcP U\la 0$ in $\mcH$ and $\widehat{\mcP}$, respectively.

Applying the functor $\Hom_\mcH(?,X)$ to the first sequence, we get an exact sequence
\[ 0\la\Hom_\mcH(\H(U),X)\stackrel{p^*}{\la}\Hom_\mcH(\H(P_0),X)\stackrel{(uq)^*}{\la}\Hom_\mcH(\H(P_1),X) \]
in $\Ab$. But the adjunction $(\H,\iota)\dd\mcU\rightleftarrows\mcH$, where $\iota\dd\mcH\la\mcU$ is the inclusion, gives a corresponding exact sequence 
\begin{equation} \label{eq:t-generating-seq1}
0\la\Hom_\mcU (U,X)\stackrel{p^*}{\la}\Hom_\mcU (P_0,X)\stackrel{(uq)^*}{\la}\Hom_\mcU (P_1,X).
\end{equation}
On the other hand, applying the functor $\Hom_{\widehat{\mcP}}(?,\y_\mcP X)$ to the second of the exact sequences in the previous paragraph, we get another exact sequence
\begin{equation} \label{eq:t-generating-seq2}
0\la\Hom_{\widehat{\mcP}}(\y_\mcP U,\y_\mcP X)\stackrel{p^*}{\la}\Hom_{\widehat{\mcP}}(\y_\mcP P_0,\y_\mcP X)\stackrel{(uq)^*}{\la}\Hom_{\widehat{\mcP}}(\y_\mcP P_1,\y_\mcP X)
\end{equation}
in $\Ab$. The two exact sequences~\eqref{eq:t-generating-seq1} and~\eqref{eq:t-generating-seq2} can be clearly inserted as rows of a commutative diagram with $\eta_{U,X}$,  $\eta_{P_0,X}$ and  $\eta_{P_1,X}$ as vertical arrows connecting the two rows. Then  $\eta_{U,X}$ is an isomorphism since so are  $\eta_{P_0,X}$ and $\eta_{P_1,X}$. 
\end{proof}

We conclude the section by extracting a concrete description of the Serre subcategory $\Ker(F)\subseteq\widehat{\mcP}$ from Theorem~\ref{thm.t-generating left Kan extensions}, which will be of use later.

\begin{lem} \label{lem.identification-of-Serre-subcategory}
In the situation of Theorem~\ref{thm.t-generating left Kan extensions}, we have that an object $M\in\widehat{\mcP}$ lies in $\Ker(F)$ if, and only if, there exists a triangle $U'\stackrel{f}{\la}U\stackrel{g}{\la}U''\stackrel{h}{\la}U'[1]$ in $\mcD$, with the first three terms in $\mcU$, such that $M$ is isomorphic to $\Coker\y_\mcP(g)$. 
\end{lem}

\begin{proof}
The `if' part:  If the mentioned triangle exists, the exactness of $F$ gives an exact sequence
\[ F(\y_\mcP U)\stackrel{F(\y_\mcP(g))}{\la}F(\y_\mcP U'')\la F(M)\la 0 \]
in $\mcH$. Thanks to the natural isomorphism $F\circ\y_\mcP\cong (\H)_{| \mcU}$, this last sequence is isomorphic 
\[\H(U)\stackrel{\H(g)}{\la}\H(U'')\la F(M)\la 0. \]
Since $\H\dd\mcD\la\mcH$ is cohomological, we also have an exact sequence 
\[ \H(U)\stackrel{\H(g)}{\la}\H(U'')\stackrel{\H(h)}{\la}\H(U'[1])=0. 
\]
It follows that $F(M)=0$.
	
The `only if' part: Suppose now that $F(M)=0$ and choose a morphism $g\dd U\la U''$ in $\mcU$ (even in $\mcP$, if we want) such that $M\cong\Coker\y_\mcP(g)$. It then follows that $F(\y_\mcP(g))$ is an epimorphism since $F$ is exact and $F(M)=0$, and then in turn $\H(g)$ is an epimorphism since $F\circ \y_\mcP\cong (\H)_{| \mcU}$.
If we now complete $g$ to a triangle
\[ U'\stackrel{f}{\la}U\stackrel{g}{\la}U''\la U'[1], \]
it follows from Lemma~\ref{lem:t-generating tria} that $U'\in\mcU$.
\end{proof}


\section{Pure-injective objects and exact direct limits}
\label{sec.pure injectives}

In the previous section we have constructed, for an aisle $\mcU$ in a triangulated category $\mcD$ and a nice enough subcategory $\mcP\subseteq\mcU$, a Serre quotient functor $F\dd\widehat{\mcP}\la\mcH$ onto the heart of the $t$-structure whose aisle is $\mcU$. The construction dualizes easily and we also obtain a similar Serre quotient functor $F'\dd\widecheck{\mcQ}\la\mcH$ for a nice enough full subcategory $\mcQ\subseteq\mcV$ of a co-aisle, where
\[ \widecheck{\mcQ} := (\widehat{\mcQ^\op})^\op = (\modf(\mcQ^\op))^\op. \]

One of our main concerns is when $\mcH$ is AB5 or a Grothendieck category, and we will address the question via first checking whether $\widehat{\mcP}$ or $\widecheck{\mcQ}$ is AB5 or a Grothendieck category. In other words, we wish to obtain practical criteria on $\mcP$ and $\mcQ$ ensuring that $\widehat{\mcP}$ and $\widecheck{\mcQ}$ possess the required exactness properties, respectively.

In the first case, we restrict ourselves to the case of module categories, i.e.\ to the situation where $\mcP$ has coproducts and there exists a set $\mcS\subseteq\mcP$ such that $\mcP=\Add(\mcS)$ and each $S\in\mcS$ is small in $\mcP$ (in the sense that $\Hom_{\mcP}(S,?)\dd\mcP\la\Ab$ preserves coproducts). It is well known that then we have an equivalence
\begin{align*}
\widehat{\mcP} &\stackrel{\simeq}\la \Mod\mcS, \\
(F\dd\mcP^\op\to\Ab) &\rightsquigarrow F_{|\mcS^\op}.
\end{align*}
Although there exist Grothendieck categories with enough projective objects which are not module categories (see~\cite{BHPST}), they seem to be quite difficult to construct and we do not use them here.

Here we focus more on the dual question when $\widecheck{\mcQ}$ is AB5 or a Grothendieck category. A main argument, which we extend and apply here, was given in \cite{PoSt}. The key notion is that of pure-injectivity, which is defined in the spirit of~\cite{CS} and which coincides with the classical one when $\mcA$ is either
\begin{itemize}
\item a locally finitely presented additive category with products (see \cite{CB}; beware that following~\cite{AR}, one would dub such categories finitely accessible with products) or
\item a compactly generated triangulated category (see \cite[Theorem 1.8]{K}).
\end{itemize}

\begin{opr} \label{def.pure-injective}
Let $\mcA$ be any additive category with (set-indexed) products.
	
\begin{enumerate}
\item An object $Y$ of $\mcA$ will be called \emph{pure-injective} if, for each set $I$, there is a morphism $f\dd Y^I\la Y$ such that $f\circ\lambda_i=1_Y$, where $\lambda_i\dd Y\la Y^I$ is the canonical section, for each $i\in I$.
\item A pure-injective object $Y\in\mcA$ is \emph{accessible} if the category $\Prod_
\mcA(Y)$ has a generator (that is, there is $Y'\in\Prod_\mcA(Y)$ such that the functor $\Hom(Y',?)\dd\Prod_\mcA(Y)\la\Ab$ is faithful).
\end{enumerate}
\end{opr}

Let us collect first some easy consequences of the definition.

\begin{lem} \label{lem.closure properties of pi}
Any product of pure-injective objects in $\mcA$ is pure-injective. A summand of a pure-injective object is pure-injective.
\end{lem}

\begin{proof}
Suppose that $(Y_j)_{j\in J}$ is a collection of pure-injective objects, $I$ is a set and $f_j\dd Y_j^I \la Y_j$ is a map as in Definition~\ref{def.pure-injective}. Then $\prod_{j\in J}f_j\dd (\prod_{j\in J} Y_j)^I \la \prod_{j\in J} Y_j$ yields the identity when composed with any canonical section of the product. Hence $\prod_{j\in J} Y_j$ is pure-injective.
	
Similarly, if $Y = Y'\oplus Y''$ is pure-injective, $I$ is a set and we have a map $f\dd Y^I \la Y$ as in the definition, then the composition
\[ (Y')^I \rightarrowtail Y^I \stackrel{f}\la Y \twoheadrightarrow Y' \]
with the section of the splitting of $Y^I$ and the retraction of the splitting of $Y$ gives the desired map for $Y'$.
\end{proof}

\begin{lem} \label{lem.pi and products}
Let $\mcA$ be an additive category with products. Then $Y$ is pure-injective (resp.\ accessible pure-injective) in $\mcA$ if and only if $Y$ is such in $\Prod_\mcA(Y)$.
If $\mcB$ is another additive category with products, $F\dd\mcA\la\mcB$ is a product-preserving functor and $Y$ a pure-injective object of $\mcA$, then $F(Y)$ is pure-injective in $\mcB$.
\end{lem}

\begin{proof}
The first claim is obvious from the definition.
Regarding the second claim, let $Y\in\mcA$ be pure-injective and $I$ be a set. Fix a morphism $f\dd Y^I\la Y$ in $\mcA$ such that $f\circ\lambda_i=1_Y$, where $\lambda_i\dd Q\la Q^I$ is the canonical $i$-th section, for all $i\in I$. Then $F(f)\dd F(Y^I)\la F(Y)$ is a morphism from the product of $I$ copies of $F(Y)$ in $\mcB$ such that $F(f)\circ F(\lambda_i)=1_{F(Y)}$, for all $i\in I$. Therefore $F(Y)$ is pure-injective in $\mcB$. 
\end{proof}

In the context of Lemma~\ref{lem.closure properties of pi}, one is often interested not in individual pure-injective objects $Q\in\mcA$, but rather in the classes of the form $\Prod(Y)$. This leads to the following definition.

\begin{opr} \label{def.product equivalent}
We call two pure-injective objects $Y,Y'\in\mcA$ \emph{product-equivalent} if $\Prod(Y)=\Prod(Y')$ in $\mcA$.
\end{opr}

Note that, in the situation of Definition~\ref{def.pure-injective}, even when in addition $\mcA$ is abelian with coproducts, an injective object of $\mcA$ need not be pure-injective. The reason for this is that the canonical morphism $Y^{(I)}\la Y^I$ need not be a monomorphism, e.g.\ when $\mcA=\Ab^\op$ and $Y=\mathbb{Z}$. In fact the following extension of the dual of \cite[Theorem 3.3]{PoSt} is the main result of the subsection. Note that AB3* abelian categories with an injective cogenerator  are automatically AB3 by the adjoint functor theorem \cite[Proposition 6.4]{Faith}, and hence satisfy AB4.

\begin{prop} \label{prop.Positsetski-Stovicek}
Let $\mcA$ be an AB3* abelian category with an injective cogenerator $E$ (i.e.~$\mcA\simeq\widecheck{\mcQ}$ for $\mcQ=\Prod_\mcA(E)$ by the dual of Proposition~\ref{prop.reconstructing}). Then the following assertions are equivalent:
\begin{enumerate}
\item $\mcA$ is AB5.
\item $\mcA$ has an injective cogenerator which is pure-injective.
\item All injective objects of $\mcA$ are pure-injective. 
\end{enumerate}
Moreover, if the equivalent conditions above hold, then $\mcA$ is a Grothendieck category if and only if some (or any) injective cogenerator of $\mcA$ is accessible pure-injective.
\end{prop}
\begin{proof}
As mentioned, the first part is formally dual to \cite[Theorem 3.3]{PoSt}.

Regarding the moreover part, let use denote by $\mcQ\subseteq\mcA$ the class of injective objects and suppose first that $\mcA$ is a Grothendieck category with a generator $G$. Consider $j\dd G\rightarrowtail E$ an embedding of $G$ into an injective object. Then the $j$ induces a surjective natural transformation
\[ j^*\dd \Hom_{\mcQ}(E,?) \longrightarrow \Hom_\mcA(G,?)_{|\mcQ}. \]
Since $\Hom_\mcA(G,?)$ is faithful, so is $\Hom_{\mcQ}(E,?)$ and, hence, $E$ is a generator of $\mcQ$.
	
Suppose conversely that $\mcA$ is complete and AB5 and $E$ is a generator for $\mcQ$. We first observe that the canonical map $f\dd E^{(I)}\la F$, where $I=\Hom_\mcA(E,F)$, is an epimorphism in $\mcA$ for any $F\in\mcQ$. Indeed, if it were not, we  could consider a composition
\[ g\dd F \la \Coker(f) \rightarrowtail F', \]
where the second map is an inclusion into an injective object $F'$. Then $g$ is non-zero, but the composition $g\circ f'$ vanishes for any $f'\in\Hom_\mcA(E,F)$ by the choice of $g$. This contradicts the fact that $E$ is a generator of $\mcQ$.
	
Now we claim that the set $\mcS$ of all subquotients of finite direct sums of copies of $E$ generates $\mcA$. Indeed, given any $X\in\mcA$, we first embed it into an injective object $F$ and then we again consider the canonical map $f\dd E^{(I)}\la F$, where $I=\Hom_\mcA(E,F)$. This map is an epimorphism in $\mcA$ by the previous paragraph. If we denote for any finite subset $J\subseteq I$ by $Z_J$ the image of the composition $E^{(J)} \rightarrowtail E^{(I)} \stackrel{f}\la F$, then clearly $F$ is the direct union of the subobjects $Z_J$. By the AB5 condition, we have equalities
\[ X = F\cap X = \Big(\bigcup_{J\textrm{ finite}} Z_J\Big)\cap X = \bigcup_{J\textrm{ finite}} (Z_J\cap X) \]
in the lattice of subobjects of $F$, and hence we obtain an epimorphism
\[
\coprod_{J\textrm{ finite}} Z_J\cap X \la X
\]
in $\mcA$. This proves the claim and the proposition.
\end{proof}

Finally, we touch the question of accessibility of pure-injective objects. It is a purely technical condition, which is very often satisfied for categories arising in practice. For our purposes, we record the following lemma.

\begin{lem} \label{lem.pi in std well gen}
Let $Q$ be a pure-injective object in a standard well generated triangulated category $\mcD$. Then $Q$ is accessible pure-injective.
\end{lem}

\begin{proof}
Assume first that $\mcD$ is compactly generated. By \cite[Theorem 1.8]{K}, we know that $\y Q$ is an injective object of $\Mod\mcD^c$, where $\y\dd\mcD\la\Mod\mcD^c$ is the generalized Yoneda functor that takes $D\rightsquigarrow\y D=\Hom_\mcD(?,D)_{|\mcD^c}$. Note that $\y$ preserves products and induces an equivalence of categories $\Prod_\mcD(Q)\stackrel{\simeq}{\la}\Prod_{\Mod\mcD^c}(\y Q)$. If  now $\mcT$ denotes the hereditary torsion class in $\Mod\mcD^c$ consisting of the $\mcD^c$-modules $T$ such that $\Hom_{\Mod\mcD^c}(T,\y Q)=0$, we have that the quotient functor $q\dd\Mod\mcD^c\la (\Mod\mcD^c)/\mcT=:\mcG$ induces an equivalence $\Prod_{\Mod\mcD^c}(\y Q)\stackrel{\simeq}{\la}\Inj\mcG$. As $\mcG$ is a Grothendieck category and we have proved that $\Prod_\mcD(Q)\simeq\Inj\mcG$, the conclusion follows by Proposition~\ref{prop.Positsetski-Stovicek}.
	
Suppose now that $\mcD=\mcC/\Loc_\mcC(\mcS)$ is general, where $\mcC$ is compactly generated and $\mcS$ is a set of objects of $\mcC$. Then the localization functor $q\dd\mcC\la\mcD$ has a fully faithful right adjoint $\iota\dd\mcD\la\mcC$ by~\cite[Proposition 1.21 and Lemma 9.1.7]{N}. If $Q\in\mcD$ is pure injective, so is $\iota(Q)\in\mcC$ by Lemma~\ref{lem.pi and products}. Moreover, $\iota$ induces an equivalence of categories $\Prod_\mcD(Q) \simeq \Prod_\mcC\big(\iota(Q)\big)$. As the latter category has a generator  by the previous paragraph (see Definition \ref{def.pure-injective}), the same is true for the former category and the lemma follows.
\end{proof}


\section{Representability for coproduct-preserving homological functors}
\label{sec:representability}

In several treatments of compactly or well generated triangulated categories (see~\cite{K,N,K-triang}), coproduct-preserving homological functors played an important role. Here we wish to explain how such functors can be in great generality represented by objects of the triangulated category.

Throughout, we will denote by $\mcD$ a triangulated category with coproducts, and consider homological functors $H\dd\mcD\la\mcA$ to an AB3* abelian category $\mcA$ that has an injective cogenerator. As said before,  such categories are also AB4.

In fact, we will study homological functors as above only up to a certain equivalence. The rationale is that given a homological functor $H\dd\mcD\la\mcA$, one is for a large part only interested in the long exact sequences from triangles and whether terms or maps in these sequences vanish. If we compose $H$ with a faithful and exact functor $F\dd\mcA\la\mcA'$ of abelian categories, these properties do not change and computations with $H$ using only these properties could be equally performed with $F\circ H\dd\mcD\la\mcA'$. If $H$ preserves coproducts,  we typically wish that $F$ preserves coproducts as well.
This leads us to the following definition, where we consider an even more restrictive condition on functors $F\dd\mcA\la\mcA'$ (which is, however, equivalent if $\mcA$ is a Grothendieck category).

\begin{opr} \label{def.computationally equivalent}
Let $\mcD$ be a triangulated category with coproducts and $H\dd\mcD\la\mcA$ and $H'\dd\mcD\la\mcA'$ be coproduct-preserving homological functors, where $\mcA$, $\mcA'$ are AB3* abelian categories with  injective cogenerators.

We say that $H'$ is a \emph{faithfully exact reduction} of $H$ is there exists a faithful exact left adjoint functor $F\dd\mcA\la\mcA'$ such that $H'\cong F\circ H$.

We call $H$ and $H'$ \emph{computationally equivalent} if they are related by a finite zig-zag of faithfully exact reductions. In other words, computational equivalence is the smallest equivalence relation extending the relation 'being a faithfully exact reduction'.
\end{opr}	

The next theorem among others says that for nice enough triangulated categories $\mcD$, computational equivalence classes of coproduct-preserving functors from $\mcD$ to an AB3* abelian category with an injective cogenerator are in bijection with product-equivalence classes of objects in $\mcD$. The theorem in fact gives more precise information---it says that each computation equivalence class of homological functors contains one such functor which is initial (this can be viewed as an analogue of~Lemma~\ref{lem.factorization of exact functors} for homological functors). To state that precisely, we will use a very small piece of 2-category theory.

We define the 2-category $\HFun(\mcD)$ of coproduct-preserving homological functors originating in $\mcD$ as follows. The objects will be all coproduct-preserving homological functors $H\dd\mcD\la\mcA$, where $\mcA$ is an AB3* abelian category with an injective cogenerator. The morphisms between $H\dd\mcD\la\mcA$ and $H'\dd\mcD\la\mcA'$ will be the faithful exact left adjoint functors $F\dd\mcA\la\mcA'$ making the triangle
\[
\xymatrix{
& \mcD \ar[dl]_-{H} \ar[dr]^-{H'} \\
\mcA \ar[rr]_-{F} && \mcA' \\
}
\]
strictly commutative. The collection of natural transformations between the morphisms $F,F'\dd H\la H'$ consists all natural transformation between the underlying functors $\mcA\la\mcA'$ in the usual sense.

In that language, the computational equivalence classes precisely correspond to the connected components of $\HFun(\mcD)$ (i.e. the smallest subclasses of objects which are pairwise connected by zigzags of morphisms). This is because any two naturally isomorphic homological functors  $H_0,H_1\dd\mcD\la\mcA$ which are objects of $\HFun(\mcD)$ are in the same connected component. Indeed, the full subcategory $\Iso (\mcA)$ of $\Mor (\mcA)$ given by the isomorphisms is equivalent to $\mcA$ via equivalences of categories $\pi_0,\pi_1\dd\Iso(\mcA)\la\mcA$ that take any isomorphism $A_0\la A_1$ to $A_0$  and $A_1$, respectively. A quasi-inverse for each $\pi_i$ is the functor $\mcA\la\Iso (\mcA)$ that takes any object to its identity morphism. If  $\alpha\colon H_0\overset{\cong}\la H_1$ is natural isomorphism, then we have an obvious homological functor  $\overline{H}\colon\mcD\la\Iso(\mcA)$ which sends $X\in\mcD$ to $\alpha_X\in\Iso(\mcA)$.  Since we clearly have that $\pi_i\circ\overline{H}=H_i$, for $i=0,1$, we conclude that all of $H_0$ ,$\overline{H}$ and $H_1$ are in the same connected component of $\HFun(\mcD)$.


We will consider each component of $\HFun(\mcD)$ as a full sub-2-category; then $\HFun(\mcD)$ is a disjoint union of these. Finally, we define an initial object in a 2-category $\mcC$ as an object $X\in\mcC$ such that each $Y\in\mcC$ admits a unique morphism from $X$ up to natural equivalence. Such an $X$ is necessarily unique in $\mcC$ up to equivalence (and, in fact, is the initial object of the ordinary category which we obtain from $\mcC$ when we identify naturally equivalent morphisms, so that equivalences in $\mcC$ are turned to isomorphisms).

\begin{thm} \label{thm.classify homological functors}
Let $\mcD$ be a triangulated category which has arbitrary (set-indexed) coproducts and satisfies Brown representability theorem. Then there is a bijective correspondence between
\begin{enumerate}
\item the connected components of $\HFun(\mcD)$ and
\item product-equivalence classes of objects in $\mcD$.
\end{enumerate}
Moreover, each connected component of $\HFun(\mcD)$ has an (up to equivalence) unique initial object $H\dd\mcD\la\mcA$, which is characterized by the fact that it induces an equivalence $H_{|\Prod(Q)}\dd\Prod(Q)\stackrel{\simeq}\la\Inj\mcA$, where $Q$ is an object representing the product-equivalence class as in~(2) corresponding to $H$.
\end{thm}

In order to prove the theorem, we first establish the following characterization of exact and faithful left adjoints.

\begin{lem} \label{lem.faithful exact left adjoints}
Let $(F,G)\dd\mcA\rightleftarrows\mcB$ be an adjoint pair of functors between abelian categories and suppose that $\mcB$ is AB3* with an injective cogenerator $E$. Then
\begin{enumerate}
\item[(i)] $F$ is exact if and only if $G(E)$ is injective in $\mcA$, and
\item[(ii)] $F$ is faithful if and only if $G(E)$ is a cogenerator in $\mcA$.
\end{enumerate}
\end{lem}

\begin{proof}
(i) Let $f\dd X\rightarrowtail Y$ be a monomorphism in $\mcA$. Then $F(f)$ is a monomorphism if and only if $\Hom_\mcB(F(f),E)$ is surjective if and only if $\Hom_\mcA(f,G(E))$ is surjective. Thus, $G(E)$ is injective in $\mcA$ if and only if $F$ preserves monomorphisms. Since $F$ is right exact, the conclusion follows.

(ii) Let $f\dd X\la Y$ be any morphism in $\mcA$. Then $F(f)$ vanishes if and only if $\Hom_\mcA(f,G(E))$ vanishes. Hence $G(E)$ is a cogenerator if and only if $F$ is faithful.
\end{proof}

\begin{proof}[Proof of Theorem~\ref{thm.classify homological functors}]
Note that $\mcD$ has products, which we obtain by applying Brown representability to products of functors $\prod_{i\in I}\Hom_\mcD(-,D_i)\dd \mcD^\op\la\Ab$.

Let us describe the correspondence between (1) and (2). First fix an object $(H\dd\mcD\la\mcA)$ of $\HFun(\mcD)$. Given an injective object $E\in\mcA$, we choose $G(E)\in\mcD$ representing $\Hom_\mcA(H(?),E)\dd\mcD^\op\la\Ab$. By the Yoneda lemma, we in fact obtain a product-preserving functor $G\dd\Inj\mcA\la\mcD$ and a natural isomorphism
\begin{equation} \label{eq.partial adjoint to H}
\Hom_\mcA\big(H(?),?\big) \cong \Hom_\mcD\big(?,G(?)\big)\dd \mcD^\op \times \Inj\mcA \la \Ab.
\end{equation}

We assign to $H$ the object $G(E)\in\mcD$, where $E\in\mcA$ is an injective cogenerator. In order to see that this is well-defined, first note that any two injective cogenerators are product-equivalent, and so are their images under $G$. Furthermore, if $F\dd\mcA\la\mcA'$ is a faithful exact left adjoint functor and $G'\dd\mcA'\la\mcA$ is the corresponding right adjoint, then
\[ \Hom_{\mcA'}\big(F(H(?)),?\big) \cong \Hom_\mcA\big(H(?),G'(?)\big) \cong \Hom_\mcD\big(?,G(G'(?))\big). \]
If $E'$ is an injective cogenerator of $\mcA'$, then $G'(E)$ is an injective cogenerator of $\mcA$ thanks to Lemma~\ref{lem.faithful exact left adjoints}. This implies that both $H$ and $F\circ H$ are assigned to the product-equivalence class of the object $G(G'(E))\in\mcD$.

Conversely, let us start with the class $\Prod_\mcD(Q)$ obtained from $Q\in\mcD$. Then $\mcA_Q := \Cont(\Prod(Q),\Ab)^\op = \widecheck{\Prod(Q)}$ is an AB3* abelian category and the functor $E_Q := \Hom_{\Prod(Q)}(Q,?)$ is its injective cogenerator by Lemma~\ref{lem.Cont(P)}. We assign the product-equivalence class of $Q$ to the restricted Yoneda functor
\begin{equation} \label{eq.homological functor for Q}
\begin{split}
H_Q\dd \mcD &\la \mcA_Q, \\
D &\rightsquigarrow \Hom_\mcD(D,?)_{|\Prod(Q)}.
\end{split}
\end{equation}
This is obviously a homological functor and it preserves coproducts since
\[ \Hom_\mcD\Big(\coprod_i D_i,?\Big)_{|\Prod(Q)} \cong \prod_i \Hom_\mcD(D_i,?)_{|\Prod(Q)}\dd \Prod_\mcD(Q) \la \Ab \]
is a product in $\Cont(\Prod(Q),\Ab)$ and, thus, a coproduct in $\mcA_Q$.

Now we prove that the assignments provide mutually inverse bijections. If we start with $Q\in\mcD$, we have for any $D\in\mcD$ that
\begin{multline*}
\Hom_{\mcA_Q}\big(H_Q(D),E_Q\big) \cong \\
\Hom_{\Cont(\Prod(Q),\Ab)}\big( \Hom_{\Prod(Q)}(Q,?), \Hom_\mcD(D,?)_{|\Prod(Q)} \big) \cong \\
\Hom_\mcD(D,Q)
\end{multline*}
by the Yoneda lemma. Comparing this with~\eqref{eq.partial adjoint to H}, we see that the corresponding functor $G_Q\dd\Inj{\mcA_Q}\la\mcD$ sends $E_Q$ to $Q$. It follows that the assignment $(1)\la(2)$ recovers $Q$ back from $\mcA_Q$.

Let us conversely start with $(H\dd\mcD\la\mcA)\in\HFun(\mcD)$ and consider the functor $G\dd\Inj\mcA\la\mcD$ defined by~\eqref{eq.partial adjoint to H}, an injective cogenerator $E\in\mcA$ and the object $Q=G(E)\in\mcD$.
To see that $H$ and $H_Q$ as in~\eqref{eq.homological functor for Q} are computationally equivalent, it suffices to prove that there is a faithful exact left adjoint functor $F\dd\mcA_Q\la\mcA$ such that $F\circ H_Q\cong H$. To this end, observe that the precomposition with $G\dd\Inj\mcA\la\Prod_\mcD(Q)$ induces a functor
\[ F\dd \mcA_Q = \Cont\big(\Prod_\mcD(Q),\Ab\big)^\op \la \Cont\big(\Inj\mcA,\Ab\big)^\op \simeq \mcA. \]
The functor $F\circ H_Q:\mcD\la\mcA$ takes $D$ to $\Hom_\mcD(D,G(?))_{|\Inj\mcA}$. Since $\Inj\mcA=\Prod (E)$ we get by~\eqref{eq.partial adjoint to H} that $(F\circ H_Q)(D)\cong\Hom_\mcA(H(D),?)_{| \Inj\mcA}$, which is precisely the object that corresponds to $H(D)$ by the canonical equivalence $\mcA\cong\Cont(\Inj\mcA,\Ab)^\op$. That is, we have an isomorphism $(F\circ H_Q)(D)\cong H(D)$, for all $D\in\mcD$,  which easily leads to a natural isomorphism $F\circ H_Q\cong H$.

There is also a natural functor in the opposite direction. Namely, the dual version of Lemma~\ref{prop.reconstructing} provides and equivalence
\[ \stiMor\big(\Inj\mcA\big) \la \mcA, \quad (f\dd Q^0 \la Q^1) \rightsquigarrow \Ker(f), \]
where $\stiMor(\Inj\mcA)$ is the quotient of $\Mor(\Inj\mcA)$ by the ideal of all maps factoring through a split monomorphism. Similarly, we have $\stiMor(\Prod_\mcD(Q))\simeq\mcA_Q$ given by $f'\rightsquigarrow\Coker\Hom_{\Prod(Q)}(f',?)$ (the cokernel is taken in $\Cont(\Prod_\mcD(Q))=\mcA_Q^\op$). Now $G\dd\Inj\mcA\la\Prod_\mcD(Q)$ induces a functor $\Mor(\Inj\mcA)\la\Mor(\Prod_\mcD(Q))$ and hence also a functor
\[ G'\dd \mcA \simeq \stiMor(\Inj\mcA) \la \stiMor\big(\Prod_\mcD(Q)\big) \simeq \mcA_Q. \]
That $F$ is left adjoint to $G$ follows by a computation analogous to~\eqref{eq.adjoint to restriction} in the proof of Lemma~\ref{lem.restricted Yoneda fundamental}. Finally, $F$ is faithful and exact since $G$ sends by construction the injective cogenerator $E\in\mcA$ to the injective cogenerator $\Hom_{\Prod(Q)}(Q,?)\in\mcA_Q$.

To prove the moreover part, note that given $Q\in\mcD$, $H_Q$ as in~\eqref{eq.homological functor for Q} induces an equivalence $(H_Q)_{|\Prod(Q)}\dd\Prod(Q)\stackrel{\simeq}\la\Inj{\mcA_Q}$ by the Yoneda lemma. Furthermore, we have just proved that any computational equivalent homological functor $H\dd\mcD\la\mcA$ admits a morphism $H_Q\la H$ in $\HFun(\mcD)$.
On the other hand, if $H'\dd\mcD\la\mcA'$ in $\HFun(\mcD)$ induces an equivalence
\[ H'_{|\Prod(Q)}\dd\Prod(Q)\stackrel{\simeq}\la\Inj{\mcA'} \]
and $F\dd \mcA'\la\mcA''$ is an exact functor to an abelian category $\mcA''$, then $F\circ H'$ determines $F$ up to natural isomorphism. Indeed, $F\circ H'$ determines $F_{|\Inj{\mcA'}}$ and, since $\mcA'$ has enough injectives and $F$ is left exact, $F_{|\Inj{\mcA'}}$ determines $F$. It follows that $H'$ is a (necessarily unique up to equivalence) initial object of the connected component of $\HFun(\mcD)$ in which it is contained.
\end{proof}

Once Theorem~\ref{thm.classify homological functors} is at hand, we have a clean criterion to determine when two functors as in its statement are computationally equivalent.

\begin{cor} \label{cor.clean-definition-comp.equivalent}
Let $\mcD$ be as in  Theorem~\ref{thm.classify homological functors} and let $H,H'\in\HFun(\mcD)$.  The following assertions are equivalent:
\begin{enumerate}
\item $H$ and $H'$ are computationally equivalent.
\item A morphism $s\in\Mor(\mcD)$ is in $\Ker H$ if, and only if, it is in $\Ker H'$.
\end{enumerate}
\end{cor}
\begin{proof}
$(1)\Longrightarrow (2)$ We have factorizations \[
H\dd\mcD\stackrel{H_Q}{\la}\mcA_Q\stackrel{F}{\la}\mcA
\qquad \textrm{and} \qquad  H'\dd\mcD\stackrel{H_Q}{\la}\mcA_Q\stackrel{F'}{\la}\mcA',
\]
where $H_Q$ is the initial object in the connected component of $\HFun (\mcD)$ to which $H$ and $H'$ belong and $F$ and $F'$ are faithful exact functors. Then, for a given $s\in\Mor(\mcD)$, one has that $H(s)=0$ if and only if $H_Q(s)=0$, if and only if $H'(s)=0$.

$(2)\Longrightarrow (1)$ Let $Q$ and $Q'$ be  objects of $\mcD$  representing the initial objects of the connected component of $H$ and $H'$ in $\HFun(\mcD)$.
By the previous paragraph we have that, given an $s\in\Mor(\mcD)$, $H_Q(s)=\Hom_\mcD(s,?)_{|\Prod(Q)}=0$ if, and only if, $H_{Q'}(s)=\Hom_\mcD(s,?)_{|\Prod(Q')}=0$. That is, $\Hom_\mcD(s,Q)=0$ if, and only if, $\Hom_\mcD(s,Q')=0$.
We now consider the canonical map $v\dd Q'\la Q^{\Hom_\mcD(Q',Q)}$ and complete it to a triangle $K\stackrel{u}{\la}Q'\stackrel{v}{\la}Q^{\Hom_\mcD(Q',Q)}\laplus$. We have that $\Hom_\mcD(u,Q)=0$, and hence  $\Hom_\mcD(u,Q')\dd\Hom_\mcD(Q',Q')\la\Hom_\mcD(K,Q')$ is also the zero map. This gives that $u=0$ and so $v$ is a section. This proves that $\Prod(Q')\subseteq\Prod(Q)$ and the reverse inclusion follows by exchanging the roles of $Q$ and $Q'$ in the argument.
\end{proof}

A conceptual explanation of the criterion in Corollary~\ref{cor.clean-definition-comp.equivalent} is given by the following observation, which we will use in the next section.

\begin{cor} \label{cor.initial homological functor as localization}
Let $\mcD$ be as in Theorem~\ref{thm.classify homological functors} and $Q\in\mcD$. If $H_Q\dd\mcD\la\mcA_Q$ is initial in the connected component of $\HFun(\mcD)$ corresponding to the product equivalence class of $Q$, then the induced functor $\widehat{H}_Q\dd\widehat{\mcD}\la\mcA_Q$ (given by Lemma~\ref{lem.coherent functors are abelian}) is a Gabriel localization functor and
\[
\Ker\widehat{H}_Q = \{ \Img\y_\mcD(s) \mid s\in\Ker H_Q \}.
\]

\end{cor}
\begin{proof}
By (the proof of) Theorem~\ref{thm.classify homological functors}, we may identify $H_Q$ with the generalized Yoneda functor
\begin{align*}
\mcD &\la \Cont\big(\Prod_\mcD(Q),\Ab\big)^\op = \widecheck{\Prod_\mcD(Q)} = \mcA_Q, \\
D &\rightsquigarrow \Hom_\mcD(D,?)_{|\Prod_\mcD(Q)}.
\end{align*}
On the other hand, it is rather well known that $\widehat{\mcD}$ is an abelian category with enough injective objects and these coincide with the projective objects. We also know by~\cite[Lemma 2.1]{K} that $\y_\mcD\dd\mcD\la\widehat{\mcD}$ is a universal homological functor and, by the last sentence, the functor
\begin{align*}
\y'_\mcD := (\y_{\mcD^\op})^\op\dd \mcD &\la \widecheck{\mcD} = (\widehat{\mcD^\op})^\op \\
D &\rightsquigarrow \Hom_\mcD(D,?).
\end{align*}
has the same property. By comparing the universal properties, this implies that we can canonically identify $\widehat{\mcD}$ and $\widecheck{\mcD}$ in a way compatible with the Yoneda embeddings.

The functor $\widehat{H}_Q$ then identifies with the opposite of the restriction functor along the inclusion $\Prod_\mcD(Q)\subseteq\mcD$, since the following diagram commutes:
\[
\xymatrix{
\mcD \ar[rr]^-{(\y_{\mcD^\op})^\op} \ar[rrd]_-{H_Q} && \widecheck{\mcD} \ar[d]^-{\res^\op} \\
&& \widecheck{\Prod_\mcD(Q)}.
}
\]
Now we just apply the dual version of Lemma~\ref{lem.restricted Yoneda fundamental} to see that $\res^\op$ is a Gabriel localization functor.

In order to compute $\Ker\widehat{H}_Q = \Ker(\res^\op) = \Ker\Hom_{\widecheck{\mcD}}(-,\y'_\mcD(Q))$, note first that any $M\in\widecheck{\mcD}$ is of the form $\Img\y'_\mcD(s)$ for a map $s\dd X\la Y$ in $\mcD$. This holds since $\widecheck{\mcD}\;(\simeq\widehat{\mcD})$ has enough projective and enough injective objects and both are precisely the representable functors. Since also $\y'_\mcD(Q)$ is injective in $\widecheck{\mcD}$, we have $\Hom_{\widecheck{\mcD}}\big(\Img\y'_\mcD(s),\y'_\mcD(Q)\big) = \Img\Hom_{\widecheck{\mcD}}\big(\y'_\mcD(s),\y'_\mcD(Q)\big)\cong\Img\Hom_\mcD(s,Q)$. Hence $\Img\y'_\mcD(s)\in\Ker\widehat{H}_Q$ if and only if $\Hom_\mcD(s,Q)=0$ if and only if $H_Q(s)=0$.
\end{proof}

Theorem~\ref{thm.classify homological functors} shows that there are too many coproduct-preserving homological functors from $\mcD$, almost as many as objects of $\mcD$. In order to make the theorem practical, we restrict the class of functors of interest to those with an AB5 target. To then end, note that AB5 descends along faithful exact left adjoints.

\begin{lem} \label{lem.AB5 descent}
Let $F\dd\mcA\la\mcB$ be a faithful exact left adjoint functor between AB3* abelian categories. If $\mcB$ is AB5 and with an injective cogenerator, then $\mcA$ has the same properties.
\end{lem}

\begin{proof}
Let $E\in\mcB$ be an injective cogenerator, which is pure-injective by Proposition~\ref{prop.Positsetski-Stovicek}. If $G$ is right adjoint to $F$, then $G(E)\in\mcA$ is a pure-injective injective cogenerator by Lemmas~\ref{lem.pi and products} and~\ref{lem.faithful exact left adjoints}. Finally, $\mcA$ is AB5 by Proposition~\ref{prop.Positsetski-Stovicek}.
\end{proof}

Then we obtain the following corollaries of Theorem~\ref{thm.classify homological functors}. The good news is that the class of pure-injective objects is generally considered much more tractable than that of all objects. For instance, any compactly generated triangulated category admits a pure-injective object $Q$ such that $\Prod(Q)$ exhausts all pure-injectives. In the following section we will prove the same for standard well generated triagulated categories. On the other hand, it rarely happens that there is $Q\in\mcD$ such that $\Prod(Q)$ exhausts all objects of $\mcD$.

\begin{cor} \label{cor.classify homological functors}
The bijection from Theorem~\ref{thm.classify homological functors} restricts to a bijection between
\begin{enumerate}
\item the computational equivalence classes of coproduct-preserving homological functors $H\dd\mcD\la\mcA$, where $\mcA$ is complete AB5 abelian with an injective cogenerator,
\item product-equivalence classes of pure-injective objects in $\mcD$.
\end{enumerate}
\end{cor}

\begin{proof}
If $(H\dd\mcD\la\mcA)\in\HFun(\mcD)$ is such that $\mcA$ is AB5, so is the initial object of the connected component of $H$ by Lemma~\ref{lem.AB5 descent}.
\end{proof}

\begin{exs}
\begin{enumerate}
\item If $C\in\mcD$ is compact and $H=\Hom_\mcD(C,?)$, then $H$ corresponds to the product-equivalence class of the object $C^*$ representing the functor $\Hom_{\mathbb{Z}}\big(\Hom_\mcD(C,?),\mathbb{R}/\mathbb{Z}\big)\dd\mcD^\op\la\Ab$.
\item If $\mcD$ is compactly generated and $\y\dd\mcD\la\Mod\mcD^c$ is the standard restricted Yoneda functor, then $\y$ corresponds to the product-equivalence class of $\prod_{C\in\mcD^c}C^*$. In fact $\Prod\{C^*\mid C\in\mcD^c\}$ is the class of all pure injective objects of $\mcD$.
\item If $\mcD=\mcS\mcH$ is the stable homotopy category of spectra and $E\in\mcS\mcH$, then we have the homological theory with coefficients in $E$ given by $E_* := \pi_0(E\wedge?)\dd\mcS\mcH\la\Ab$. It corresponds to the pure-injective spectrum $E'$ which represents the functor $\Hom_{\mathbb{Z}}\big(E_*(?),\mathbb{R}/\mathbb{Z}\big)\dd\mcS\mcH^\op\la\Ab$. This construction was considered by Brown and Comenetz~\cite{BC-dual}.
\item If we specifically choose the Eilenberg-Maclane spectrum $E=H\mathbb{Z}$ in~(3), then $E_*$ is the ordinary homology with coefficients in $\mathbb{Z}$. In that case $E'=HG$, where $G=\mathbb{R}/\mathbb{Q}$ (considered as a discrete group). This follows from the universal coefficient theorem for cohomology.
\end{enumerate}
\end{exs}

We conclude the section with a general existence result for $t$-structures (and so also semi-orthogonal decompositions) cogenerated by a pure-injective object,
generalizing a recent result of Laking and Vit\'{o}ria, \cite[Corollary 5.11]{LV}.

\begin{prop} \label{prop.t-structure from pi}
If $\mcD$ is a standard well generated triangulated category and $Q$ a pure-injective object, then $\mcD$ admits a $t$-structure $(\mcU_Q,\mcV_Q):=(_{}^{\perp_{<0}}Q, (_{}^{\perp_{\leq 0}}Q)^\perp)$.
\end{prop}

\begin{proof}
Recall that $\mcD$ has products and satisfies the Brown representability theorem. If we put $Q':=\prod_{i<0}Q[i]$, then $Q'$ is also pure-injective and we have the usual coproduct-preserving homological functor $H_{Q'}\dd\mcD\la\mcA_{Q'}$ such that $\mcA_{Q'}$ is a Grothendieck category
(recall Proposition~\ref{prop.Positsetski-Stovicek} and Lemma~\ref{lem.pi in std well gen}). Moreover, $\Ker(H_{Q'})={}^{\perp_{<0}}Q$. Note that this subcategory is clearly  suspended and closed under coproducts in $\mcD$. So according to Proposition~\ref{prop:aisle-by-Neeman},
it remains to show that, for any object $X\in\mcD$, there is a set $\mcS_X\subset {}^{\perp_{<0}}Q$ such that any morphism $f\colon X\la U$, with $U\in {}^{\perp_{<0}}Q$, factors through an object of~$\mcS_X$.

Here we rely on results on well generated triangulated categories from~\cite{N,K-triang} and model the argument on the proof of~\cite[Theorem 7.5.1]{K-triang}. First of all, $\mcD$ is $\lambda$-well generated for some regular cardinal $\lambda$ and there exist arbitrarily large regular cardinals $\mu\ge\lambda$ such that
\begin{enumerate}
\item $H_{Q'}(C)$ is $<\mu$-presented in $\mcA_{Q'}$ for each $C\in\mcD^\lambda$ (here $\mcD^\lambda$ stands for the essentially small subcategory of $\lambda$-compact objects as in~\S\ref{subsec:well-gen}) and
\item the class of $<\mu$-presented objects in $\mcA_{Q'}$ forms and exact abelian subcategory (see~\S\ref{subsec.Gabriel-Popescu}).
\end{enumerate}
Since $\mcD^\mu$ coincides by~\cite[Lemma~5]{K-wellgen} with the smallest triangulated subcategory containing $\mcD^\lambda$ and closed under coproducts with $<\mu$ terms, and since $H_{Q'}$ is homological and preserves these coproducts, it follows that $H_{Q'}(C)$ is $<\mu$-presented even for each $C\in\mcD^\mu$.

Now let $f\dd D_0=X\la U$ be a morphism in $\mcD$ with $U\in {}^{\perp_{<0}}Q$ and we fix an uncountable regular cardinal $\mu\ge\lambda$ such that $D_0$ is $\mu$-compact and $\mu$ satisfies conditions $(1)$ and $(2)$ above.
If we choose a skeleton $\widetilde{\mcD}^\mu$ of $\mcD^\mu$ (and we without loss of generality assume that $D_0$ is contained in $\widetilde{\mcD}^\mu$), we shall see that $\mcS_X=\widetilde{\mcD}^\mu\cap {}^{\perp_{<0}}Q$ will satisfy the required condition from Proposition~\ref{prop:aisle-by-Neeman}.

To see this, 
note that $H_{Q'}(U)=\varinjlim_{(g\dd C\to U)} H_{Q'}(C)$ in $\mcA_{Q'}$ by~\cite[Theorem 6.9.1]{K-triang}, where the colimit runs over all morphisms $g\dd C\la U$ with $C\in\widetilde{\mcD}^{\mu}$.
Although this is not a $\mu$-direct limit, it is a so-called $\mu$-filtered colimit by~\cite[Lemma 6.5.1]{K-triang} and $\mu$-filtered colimits are very close to $\mu$-direct limits (see~\cite[Theorem 1.5 and Remark 1.21]{AR} for a precise relation). In particular $\Hom_{\mcA_{Q'}}(H_{Q'}(D_0),?)$ commutes with $\mu$-filtered colimits and, since $H_{Q'}(U)=0$, we have 
\[
\varinjlim_{g\dd C\to U} \Hom_{\mcA_{Q'}}(H_{Q'}(D_0), H_{Q'}(C))\cong
\Hom_{\mcA_{Q'}}(H_{Q'}(D_0), H_{Q'}(U))=
0
\]
and so we can find a factorization $D_0\overset{g_1}\la D_1\overset{f_1}\la U$ of $f$ with $D_1\in\widetilde{\mcD}^\mu$ and such that $H_{Q'}(g_1)=0$. However, we can obtain a similar factorization of $f_1\dd D_1\la U$ for the same reason and repeating this procedure again and again, we construct by induction a cocone 
\[
\vcenter{
\xymatrix{ D_0 \ar[r]^-{g_1} \ar@/_/@<-.5ex>[rrrrd]_-{f} & D_1 \ar[r]^{g_2} \ar@/_/[rrrd]|-\hole|-{f_1} & D_2 \ar@/_/@<.5ex>[rrd]|-\hole|-{f_2} \ar[r]^{g_3} & \dots \\
&&&& U, }
}
\]
with all the $D_i$ in $\widetilde{\mcD}^\mu$ and such that $H_{Q'}(g_i)=0$ for all $i>0$.
Finally, note that $f$ factors through the Milnor colimit $\Mcolim(D_i)$ of the sequence in the upper row (see \S\ref{subsec:Mcolim}) and, since $\mu$ was chosen uncountable, we have $\Mcolim(D_i)\in\widetilde{\mcD}^\mu$ up to isomorphism. However,   putting $x=g$ in  the triangle~\eqref{eq.Mcolim} of the Milnor colimit, we have that $H_{Q'}(g)=0$ and so $1-g\dd\coprod_{i\geq 0}D_i\la\coprod_{i\geq 0}D_i$   is sent to the identity $\text{id}_{H_{Q'}(\coprod_{i\geq 0}D_i)}$ by $H_{Q'}$. Since the triangle is sent by $H_{Q'}$ to a short exact sequence (see Definition \ref{def.pure-triangle in standard well-gen. categs} and Example \ref{ex.Milnor triangle} below), we conclude that  $H_{Q'}(\Mcolim(D_i)) = 0$, which implies that $\Mcolim(D_i)\in\widetilde{\mcD}^\mu\cap{}^{\perp_{<0}}Q$.
\end{proof}

\section{Universal coproduct-preserving homological functors} \label{sect.universal copr-pres}

A starting point for this section is a result by Krause~\cite[Corollary 2.4]{K} saying that, for a compactly generated triangulated category $\mcD$,  the generalized Yoneda functor
\begin{equation} \label{eq.Freyd to Yoneda}
\begin{split}
h_\pure\dd \mcD &\la \Mod\mcD^c \quad \big(\simeq \widehat{\Add_\mcD(\mcD^c)\big)}, \\
X &\rightsquigarrow \Hom_\mcD(?,X)_{|\mcD^c},
\end{split}
\end{equation}
is a universal coproduct-preserving homological functor with an AB5 target in the following sense: any other coproduct-preserving homological functor $H\dd\mcD\la\mcA$, where $\mcA$ is an AB5 abelian category, factors essentially uniquely as $H\cong F\circ h_\pure$, where the functor $F\dd\Mod\mcD^c\la\mcA$ is exact and coproduct-preserving (or equivalently, $F$ is an exact left adjoint).

In fact, the proof of~\cite[Proposition 2.3]{K} shows more: Each natural transformation $\alpha\dd H\la H'$ between coproduct-preserving homological functors $\mcD\la\mcA$ uniquely extends to a natural transformation $\varphi\dd F\la F'$ between the corresponding exact coproduct-preserving functors $\Mod\mcD^c\la\mcA$. Thus, the precomposition with $h_\pure$ induces an equivalence between the corresponding functor categories
\begin{equation}\label{eq.univ-prop-hpure-Krause}
h_\pure^*\dd [\Mod\mcD^c,\mcA]_{\textsf{ex},\amalg} \stackrel{\simeq}\la [\mcD,\mcA]_{\textsf{h},\amalg}.
\end{equation}
We remind the reader of Remark~\ref{rem.set-theory-and-localization} at this point---analogous considerations about the interaction with set theory apply here as well.

This result has been further generalized to homological functors with only exact $\kappa$-directed colimits for some cardinal $\kappa$ (see~\cite{N,K-triang}), but here we pursue another direction.
As we are interested in methods involving purity and pure-injectivity, it appears crucial to insist that the targets of our coproduct-preserving homological functors are AB5. The next proposition says that such a universal functor $h_\pure\dd\mcD\la\mcA_\pure(\mcD)$ exists at least  for any standard (in particular, for any algebraic or topological) well generated triangulated category $\mcD$.

\begin{prop} \label{prop.universal AB5 homology}

Let $\mcD$ be a standard well generated triangulated category. Then there exists a coproduct-preserving homological functor $h_\pure\dd\mcD\la\mcA_\pure(\mcD)$ to a Groth\-en\-dieck category $\mcA_\pure(\mcD)$ with the following universal property:
The precomposition with $h_\pure$ induces and equivalence between the categories of coproduct-preserving homological functors $H\dd\mcD\la\mcA$ and all natural transformations on one hand, and exact coproduct-preserving functors $F\dd\mcA_\pure(\mcD)\la\mcA$ and all natural transformations on the other hand,
\[ h_\pure^*\dd [\mcA_\pure(\mcD),\mcA]_{\textsf{ex},\amalg} \stackrel{\simeq}\la [\mcD,\mcA]_{\textsf{h},\amalg}. \]

Moreover, there is up to isomorphism a unique functor $\operatorname{res}'\dd\widehat{\mcD}\la\mcA_\pure(\mcD)$ which is exact, has a fully faithful right adjoint (so it is a Serre quotient) and makes the following triangle commutative:
\[
\xymatrix{
\mcD \ar[rr]^-{\y_\mcD} \ar[rrd]_-{h_\pure} && \widehat{\mcD} \ar[d]^-{\operatorname{res}'} \\
&& \mcA_\pure(\mcD).
}
\]

\end{prop}

The proof requires some preparation and will be given later in the section.
This result allows us to extend the definition of pure triangles 
and we will see later that the pure-injective objects in the sense of Definition~\ref{def.pure-injective} become injective with respect to them.

\begin{opr} \label{def.pure-triangle in standard well-gen. categs}
Let $\mcD$ be a standard well generated triangulated category.  A triangle 
\[ X\stackrel{u}{\la}Y\stackrel{v}{\la}Z\stackrel{w}{\la}X[1] \] 
in $\mcD$ is called \emph{pure} if the functor $h_\pure\dd\mcD\la\mcA_\pure(\mcD)$ induces an exact sequence
\[ 0 \la h_\pure(X) \stackrel{h_\pure(u)}\la h_\pure(Y) \stackrel{h_\pure(v)}\la h_\pure(Z) \la 0 \]
(and this happens if and only if any coproduct-preserving homological functor with AB5 target takes the triangle to a short exact sequence).
\end{opr}

If $\mcD$ is compactly generated, we define $h_\pure$ as in~\eqref{eq.Freyd to Yoneda} and put $\mcA_\pure(\mcD) := \Mod\mcD^c$.
In general, we can express $\mcD$ as $\mcD=\mcC/\Loc_\mcC(\mcS)$, where $\mcC$ is compactly generated triangulated and $\mcS\subseteq\Ob(\mcC)$ is a set of objects. Let $\mcT\subseteq\Mod\mcC^c$ be the smallest hereditary torsion class containing $h_\pure\big(\bigcup_{n\in\mathbb{Z}}\mcS[n]\big)$, put $\mcA_\pure(\mcD)=\Mod\mcC^c/\mcT$ and define $h_\mcD=h_\pure\dd\mcD\la\mcA_\pure(\mcD)$ as the unique functor fitting into the following commutative diagram (we will abuse the notation and denote both horizontal arrows by $h_\pure$):
\begin{equation} \label{eq.definition of h_pure}
\vcenter{
\xymatrix{
\mcC \ar[r]^-{h_\pure} \ar[d]_-{q} & \Mod\mcC^c \ar[d]^-{q'} \\
\mcD \ar@{.>}[r]_-{h_\pure} & \mcA_\pure(\mcD).
}
}
\end{equation}

\begin{ex} \label{ex.Milnor triangle}
In any standard well generated triangulated category, the triangle~\eqref{eq.Mcolim} that defines the Milnor colimit is pure. This is because if $\mcD=\mcC /\text{Loc}_\mcC(\mcS)$,  with $\mcC$ and $\mcS$ as above, then the triangle is the image under  $q\dd\mcC\la\mcD$ of the triangle associated to the same sequence,   when viewed as a sequence in $\mcC$ using the fully faithful right adjoint $\iota\dd\mcD\la\mcC$. That the functor $h_\pure\dd\mcD\la \mcA_\pure(\mcD)$ maps the triangle to a short exact sequence is then a consequence of the purity of the triangle in $\mcC$, the commutativity of the last diagram and the exactness of $q'$.
\end{ex}

\begin{lem}\label{lem.pi for universal homological functor}
With the notation above and, without loss of generality, assume that $\mcS=\mcS[n]$ for all $n\in\mathbb{Z}$.  Then the following hold:
\begin{enumerate}
\item $Y\in\mcD$ is pure-injective if and only if $Y\cong q(Y_0)$ for $Y_0\in\mcS^{\perp}\subseteq\mcC$ which is pure-injective. Such $Y_0$ is, moreover, unique up to isomorphism.
\item If $X,Y\in\mcD$ and $Y$ is pure injective, then $h_\pure$ induces an isomorphism $\Hom_\mcD(X,Y)\cong\Hom_{\mcA_\pure(\mcD)}\big(h_\pure(X),h_\pure (Y)\big)$.
\item $h_\pure$ restricts to an equivalence $\PInj\mcD\simeq\Inj{\mcA_\pure(\mcD)}$, where $\PInj\mcD$ stands for the full subcategory of pure-injective objects.
\end{enumerate}
\end{lem}

\begin{proof}
(1) The functor $q$ has a fully faithful right adjoint $\iota\dd\mcD\la\mcC$ which induces an exact equivalence $\mcD\la\mcS^\perp$ (see~\cite[Remark 1.16, Proposition 1.21 and Lemma 9.1.7]{N}). Hence $Y$ is pure-injective in $\mcD$ if and only if $\iota(Y)$ is pure-injective in $\mcS^\perp$ if and only if $\iota(Y)$ is pure-injective in $\mcC$. It remains to note that if we put $Y_0=\iota(Y)$, then $q(Y_0)\cong Y$. The last sentence follows by the fact that $q$ restricts to an equivalence $\mcS^\perp\simeq\mcD$. 
	
(2) Since $q$ is essentially surjective on objects, we can take $X_0,Y_0\in\mcC$ such that $q(X_0)\cong X$ and $q(Y_0)\cong Y$. Moreover, we can take $Y_0=\iota(Y)\in \mcS^\perp$ which is pure-injective in $\mcC$ by the previous part. Then $q$ induces an isomorphism $\Hom_\mcC(X_0,Y_0)\cong\Hom_\mcD(X,Y)$. On the other hand, \cite[Lemma 1.7 and Theorem 1.8]{K} say that $h_\pure$ induces an isomorphism $\Hom_\mcC(X_0,Y_0)\cong\Hom_{\Mod\mcC^c}\big(h_\pure(X_0),h_\pure(Y_0)\big)$. Finally, since $h_\pure(Y_0)$ is injective in $\Mod\mcC^c$ by~\cite[Theorem 1.8]{K} and it belongs to $\mcT^\perp$, the functor $q'$ induces
\begin{align*}
\Hom_{\Mod\mcC^c}\big(h_\pure(X_0),h_\pure(Y_0)\big) &\cong
\Hom_{\mcA_\pure(\mcD)}\big(q'\circ h_\pure(X_0),q'\circ h_\pure(Y_0)\big) \\ &\cong
\Hom_{\mcA_\pure(\mcD)}\big(h_\pure(X),h_\pure(Y)\big)
\end{align*}

(3) Note that for any pure-injective object  $Y_0\in\mcC$, we have $Y_0\in\mcS^\perp$ if and only if $h_\pure(Y_0)\in h_\pure(\mcS)^\perp=\mcT^\perp$. If we combine this with part~(1) and a classical fact that $E\in\mcA_\pure(\mcD)$ is injective if and only if $E\cong q'(E_0)$ for an up to isomorphism unique injective module $E_0\in\mcT^\perp$, we deduce that $X\in\mcD$ is pure-injective if and only if $h_\pure(X)$ is injective in $\mcA_\pure(\mcD)$.
\end{proof}

\begin{proof}[Proof of Proposition~\ref{prop.universal AB5 homology}]
Let us keep the notation of~\eqref{eq.definition of h_pure} and suppose that we are given an AB5 abelian category $\mcA$. For simplicity and without loss of generality, assume that $\mcS =\mcS [n]$, for all $n\in\mathbb{Z}$.
Then the equivalence
\[
h_\pure^*\dd [\Mod\mcC^c,\mcA]_{\textsf{ex},\amalg} \stackrel{\simeq}\la [\mcC,\mcA]_{\textsf{h},\amalg}
\]
as in~\eqref{eq.univ-prop-hpure-Krause} restricts to an equivalence between the full subcategories of
\begin{enumerate}
\item[(i)] exact coproduct-preserving functors $\Mod\mcC^c\la\mcA$ vanishing on $h_\pure(\mcS)$ and
\item[(ii)] coproduct-preserving homological functors $\mcC\la\mcA$ vanishing on $\mcS$.
\end{enumerate}
Now $q'$ from~\eqref{eq.definition of h_pure} is a coproduct-preserving Serre quotient functor, so $(q')^*$ is fully faithful and, in fact, restricts to a fully faithtul functor
\[
(q')^*\dd [\mcA_\pure(\mcD),\mcA]_{\textsf{ex},\amalg} \la [\Mod\mcC^c,\mcA]_{\textsf{ex},\amalg}
\]
whose essential image is precisely the class of functors described in (i) above. Similarly, $q$ from~\eqref{eq.definition of h_pure} is also a localization functor (recall~\cite[\S2.1]{N}) and $q^*$ restricts to a fully faithful functor
\[
q^*\dd [\mcD,\mcA]_{\textsf{h},\amalg} \la [\mcC,\mcA]_{\textsf{h},\amalg}
\]
whose essential image is precisely the class of functors as in (ii) above. This implies that $h_\pure^*\dd [\mcA_\pure(\mcD),\mcA]_{\textsf{ex},\amalg} \la [\mcD,\mcA]_{\textsf{h},\amalg}$ is an equivalence, as desired.


To prove the moreover part, note that $h_\pure\dd \mcD\la\mcA_\pure(\mcD)$ restricts to an equivalence $\PInj\mcD\simeq\Inj{\mcA_\pure(\mcD)}$ by Lemma~\ref{lem.pi for universal homological functor}.
Since $\mcA_\pure(\mcD)$ has an injective cogenerator, there exists $Q\in\PInj\mcD$ such that $\PInj\mcD=\Prod_\mcD(Q)$. It follows from Theorem~\ref{thm.classify homological functors} that $h_\pure$ is an initial object in the connected component of $\HFun(\mcD)$ corresponding to the product equivalence class of $Q$ and the conclusion follows from Corollary~\ref{cor.initial homological functor as localization}.
%
%
\end{proof}

\begin{rem}
If $\mcD$ is compactly generated, then it has enough pure-projective objects (we call $X\in\mcD$ pure-projective if $\Hom_\mcD(X,?)$ sends pure triangles to short exact sequences) and one can dualize the final part of the last proof to show that the functor $\res'\dd\widehat{\mcD}\la\mcA_\pure(\mcD)=\Mod\mcD^c$ in Proposition~\ref{prop.universal AB5 homology} also has a fully faithful left adjoint (see~\cite[Proposition 6.7.1]{K-triang}).
\end{rem}

An immediate consequence of the arguments is the following observation. We shall call $u$ as below the \emph{pure-injective envelope} of~$D$ in $\mcD$.

\begin{cor} \label{cor.pure-injective envelopes in standard well-gen. triangcats}
Let $\mcD$ be a standard well generated triangulated category. For each object $D\in\mcD$ there is a pure monomorphism $u\dd D\la Q_D$, uniquely determined up to isomorphism,  such that $Q_D$ is pure-injective and $h_\pure (u)\dd h_\pure (D)\rightarrowtail h_\pure(Q_D)$ is an injective envelope in $\mcA_\pure (\mcD)$.
\end{cor}

\begin{rem} \label{rem.zero-pinj-envelope}
A note of caution is apropos concerning last corollary. 
When $\mcD$ is standard well generated, one immediately gets from assertions~(2) and~(3) of Lemma~\ref{lem.pi for universal homological functor} that an object $D$ is in the kernel of $h_\pure\dd\mcD\la\mcA_\pure (\mcD)$ if and only if $\Hom_\mcD(D,Q)=0$, for all  $Q\in\text{PInj}(\mcD)$. In such case the pure-injective envelope of $D$ is just the morphism $D\la 0$. This pathology will be  possible only  when $\text{PInj}(\mcD)$ does not cogenerate $\mcD$.  It  then excludes the cases when $\mcD$ is compactly generated or when $\mcD=\mcD(\mcG)$ is the derived category of a Grothendieck category. By contrast, if $R$ is any (associative unital) ring and $\mcS\subseteq\mcD(R)$ is a set of objects such that $\mcS^{\perp_\mathbb{Z}}$ is not of the form ${^\perp\mcQ}$, for a class $\mcQ$ of pure-injective objects, then the standard well generated triangulated category $\mcD(R)/\Loc_{\mcD(R)}(\mcS)$ shows the pathology above.
Although it seems very likely that such examples do exist, we unfortunately do not know any actual instance.
\end{rem}

If we combine Theorem~\ref{thm.classify homological functors} with Proposition~\ref{prop.universal AB5 homology}, we obtain the following structure result for coproduct-preserving homological functors to AB5 abelian categories.

\begin{cor}\label{cor.factorization for homological AB5 functors}
Let $H\dd\mcD\la\mcA$ be a coproduct-preserving homological functor with $\mcD$ standard well generated and $\mcA$ satisfying AB5. Then $H$ essentially uniquely factorizes as
\[
\mcD \stackrel{h_\pure}\la \mcA_\pure(\mcD) \stackrel{\tilde{q}}\la \mcA_{\mathrm{init}} \stackrel{F}\la \mcA,
\]
where $\tilde{q}$ is a Gabriel localization functor  and $F$ is a faithful exact left adjoint functor.
\end{cor}

\begin{proof}
We first use the universal property of $h_\pure\dd\mcD\la\mcA_\pure(\mcD)$ and then factorize the resulting functor $\mcA_\pure(\mcD) \la \mcA$ according to Lemma~\ref{lem.factorization of exact functors}. This ensures the uniqueness.
Since $\tilde{q}$ is an exact coproduct-preserving localization functor, it is a Serre quotient functor and the corresponding Serre subcategory is closed under coproducts. Since further $\mcA_\pure(\mcD)$ is a Grothendieck category, $\tilde{q}$ is actually a Gabriel localization functor, and hence $\mcA_{\mathrm{init}}$ is also a Grothendieck category.

It is clear that the exact faithful functor $F$ preserves coproducts, and hence all colimits. This together with the Grothendieck condition of $\mcA_{\mathrm{init}}$ imply that $F$ has a right adjoint functor, due to Freyd's Adjoint Theorem~\cite[Corollary 5.52]{Faith}. 
\end{proof}

The situation of primary interest in this paper is the one where $\H\dd\mcD\la\mcH$ is a coproduct-preserving homological functor which is associated with a $t$-structure $\mathbf{t}=(\mcU,\mcV)$ whose heart is AB5. In analogy with Propotision~\ref{prop.abelianized homology is Serre quotient}, we prove that the last step in its factorization according to Corollary~\ref{cor.factorization for homological AB5 functors} is trivial. As a consequence, we obtain a counterpart of~\cite[Theorem C]{SSV} for abstract triangulated categories without using their models.

\begin{cor} \label{cor.AB5 implies Grothendieck in well-gen. triangcats}
Let $\mcD$ be a standard well generated triangulated category with the universal coproduct-preserving homological functor $h_\pure\dd\mcD\la\mcA_\pure(\mcD)$ and let $\mathbf{t}=(\mcU,\mcV)$ be a $t$-structure such that the heart $\mcH$ is AB5 and $\H\dd\mcD\la\mcH$ preserves coproducts.
	
Then $\mcH$ is a Grothendieck category and the induced exact coproduct-preserving functor $\tilde{q}\dd\mcA_\pure(\mcD)\la\mcH$ satisfying $\H=\tilde{q}\circ h_\pure$ is a Gabriel localization functor.
\end{cor}

\begin{proof}
Using Proposition~\ref{prop.universal AB5 homology} and the universal property of $\y_\mcD\dd\mcD\la\widehat{\mcD}$, we obtain a factorization of $\H$ of the form
\[  \mcD\stackrel{\y_\mcD}\la\widehat{\mcD}\stackrel{\res'}\la\mcA_\pure(\mcD)\stackrel{\tilde{q}}\la\mcH, 
\]
Now $\res'$ is a Serre quotient functor by Proposition~\ref{prop.universal AB5 homology} and $\tilde{q}\circ\res'$ is a Serre quotient by Proposition~\ref{prop.abelianized homology is Serre quotient}. Hence $\tilde{q}$ is a localization functor by Lemma~\ref{lem.composition and cancellation of localizations}. Since $\tilde{q}$ is also exact and coproduct-preserving by the universal property of $h_\pure=\res'\circ\y_\mcD$, it is a Gabriel localization functor and $\mcH$ is a Grothendieck category by the discussion in~\S\ref{subsec.Gabriel-Popescu}.
\end{proof}


\section{\texorpdfstring{$t$-structures}{t-structures} with Grothendieck hearts}
\label{sec.t-str with Grothendieck heart}

\subsection{The AB5 condition for hearts of \texorpdfstring{$t$-structures}{t-structures} via injective cogenerators}

In this subsection we study $t$-structures whose homological functors are coproduct-preserving and whose hearts are AB3* with an injective cogenerator. In particular, we analyze the objects which represent these homological functors in view of Theorem~\ref{thm.classify homological functors}.

We first characterize the situation where the cohomological functor associated with a $t$-structure preserves coproducts.

\begin{lem} \label{lem.H preserves coproducts}
Let $\mcD$ be a triangulated category with coproducts  and $\mathbf{t}=(\mcU,\mcV)$ be a $t$-structure in $\mcD$, with heart $\mcH$. The following assertions are equivalent:

\begin{enumerate}
\item The cohomological functor $\H\dd\mcD\la\mcH$ preserves coproducts.
\item  For each family $(V_i)_{i\in I}$ of objects in $\mcV$, one has that $\tau_\mathbf{t}^{\leq 0}(\coprod_{i\in I}V_i[-1])\in\mcU[1]$.
\item For each family $(V_i)_{i\in I}$ of objects in $\mcV$, one has that $\tau_\mathbf{t}^{\leq 0}(\coprod_{i\in I}V_i[-1])\in\bigcap_{n\in\mathbb{Z}}\mcU[n]$.
\end{enumerate}
\end{lem}

\begin{rem} \label{rem.H preserves coproducts}
Note that if $\mathbf{t}$ is left non-degenerate, condition~(3) above precisely means that $\mathbf{t}$ is a smashing $t$-structure.
\end{rem}

\begin{proof}[Proof of Lemma~\ref{lem.H preserves coproducts}]
$(1)\Longrightarrow (3)$ If $(V_i)_{i\in I}$ is a family of objects of $\mcV$, then $V_i[-1+j]\in\mcV[-1]$, for all $j\leq 0$. We then have $H_\mathbf{t}^j(\coprod_{i\in I}V_i[-1])\cong\coprod_{i\in I}H_\mathbf{t}^j(V_i[-1])=\coprod_{i\in I}\H(V_i[-1+j])=0$, for all $j\leq 0$. Now apply \cite[Lemma 3.3]{NSZ} to complete the proof of the implication. 

$(3)\Longrightarrow (2)$ is clear.

$(2)\Longrightarrow (1)$ Let $(D_i)_{i\in I}$ be a family of objects in $\mcD$. For each $i\in I$, we have an induced triangle $\tau_\mathbf{t}^{>0}D_i[-1]\la\tau_\mathbf{t}^{\leq 0}D_i\la D_i\la\tau_\mathbf{t}^{>0}D_i$, where $\tau_\mathbf{t}^{>0}D_i$ and $\tau_\mathbf{t}^{>0}D_i[-1]$ are in $\mcV[-1]$. By applying the cohomological functor $\H$, we get an isomorphism $\H(\tau_\mathbf{t}^{\leq 0}D_i)\cong\H(D_i)$, for all $i\in I$, and hence an isomorphism $\coprod_{i\in I}\H(\tau_\mathbf{t}^{\leq 0}D_i)\stackrel{\cong}{\la}\coprod_{i\in I}\H(D_i)$. However, the restriction $\H_{|\mcU}\dd\mcU\la\mcH$ preserves coproducts (see \cite[Lemma 3.1 and Proposition 3.2]{PS1}), and so the canonical morphism $\coprod_{i\in I}\H(\tau_\mathbf{t}^{\leq 0}D_i)\la\H(\coprod_{i\in I}\tau_\mathbf{t}^{\leq 0}D_i)$ is an isomorphism. 

On the other hand, coproducts of triangles are triangles (see the dual of \cite[Proposition 1.2.1]{N}), so that we have another triangle
\[\coprod_{i\in I}\tau_\mathbf{t}^{>0}D_i[-1]\la\coprod_{i\in I}\tau_\mathbf{t}^{\leq 0}D_i\la\coprod_{i\in I}D_i\la\coprod_{i}\tau_\mathbf{t}^{>0}D_i\]
in $\mcD$ and, by hypothesis, we know that $\H\cong\tau_\mathbf{t}^{\geq 0}\circ\tau_\mathbf{t}^{\leq 0}$ vanishes on $\coprod_{i\in I}\tau_\mathbf{t}^{>0}D_i[-1]$ and $\coprod_{i}\tau_\mathbf{t}^{>0}D_i$. We then get the following commutative diagram, where the vertical arrows are the canonical morphisms:

\[
\xymatrix{
  \coprod_{i\in I} H^0_\mathbf{t}(\tau_\mathbf{t}^{\leq 0} D_i) \ar[r]^-{\cong} \ar[d]_{\cong} &
  \coprod_{i\in I} H^0_\mathbf{t}(D_i) \ar[d]
  \\
  H^0_\mathbf{t}(\coprod_{i\in I} \tau_\mathbf{t}^{\leq 0} D_i) \ar[r]^-{\cong} &
  H^0_\mathbf{t}(\coprod_{i\in I} D_i).
}
\]

By the above comments, the two horizontal and the left vertical arrows of this diagram are isomorphisms. Then also the right vertical arrow is an isomorphism. 
\end{proof}

The main subject of our study will be the $\Ext$-injective objects in the co-aisle $\mcV$ of a $t$-structure. We will call the collection of all such objects the \emph{right co-heart} of the $t$-structure (the terminology is explained in~\S\ref{subsec:co-t-structures} below; some authors call these objects simply injective in their contexts, see~\cite[Appendix C.5.7]{Lur-SAG} or \cite{Shaul}). Dually, the \emph{left co-heart} is the class of all $\Ext$-projective objects in the aisle.
In the case of right nondegenerate $t$-structures, parts~(1) and~(2) of the following proposition can be also found in \cite[Proposition C.5.7.3]{Lur-SAG}.

\begin{prop} \label{prop.near-partialsilting-object}
Let $\mcD$ be a triangulated category with products, let $\mathbf{t}=(\mcU,\mcV)$ be a $t$-structure in $\mcD$, with heart $\mcH=\mcU\cap\mcV$,  and let $Q\in\mcV$ be an object such that $\Hom_\mcD(?,Q)$ vanishes on $\mcV[-1]$. The following assertions hold:
\begin{enumerate}
\item $\H(Q)$ is an injective object of $\mcH$ and the assignment $f\rightsquigarrow\H(f)$ gives a natural isomorphism $\Hom_\mcD(?,Q)\stackrel{\cong}{\la}\Hom_\mcH(\H(?),\H(Q))$ of functors $\mcD^\op\la\Ab$. 
\item $\H(Q)$ is a cogenerator of $\mcH$ if and only if $\Hom_\mcD(M,Q)\neq 0$, for all  $0\neq M\in\mcH$.
\item $Q$ is pure-injective (resp.\ accessible pure-injective) in $\mcD$ if and only if $\H(Q)$ is such in $\mcH$. 
\end{enumerate}
\end{prop}
\begin{proof}
Note that, by \cite[Proposition 3.2]{PS1}, the category $\mcH$ has products, so that assertion~(3) makes sense.  Furthermore, the proof of that proposition shows that the restriction $\H_{|\mcV}\dd\mcV\la\mcH$ preserves products.

(1) The pair $\mathbf{t}^\op=(\mcV^\op,\mcU^\op)$ is a $t$-structure in $\mcD^\op$ with heart $\mcH^\op$.  Its left co-heart  $\mathcal{C}^*$ (see \cite[Section 3]{NSZ})  consists of the objects $V\in\mcV^\op$ such that $\Hom_{\mcD^\op}(V,V'[-1])=0$, for all $V'\in\mcV^\op$, since the shift functor of $\mcD^\op$ is $?[-1]$. By \cite[Lemma 3.2(1)]{NSZ},  we know that $H_t^0\dd\mathcal{C}^*\la\mcH^\op$ is a fully faithful functor whose essential image consists of projective objects in $\mcH^\op$. Furthermore, the proof of  \cite[Lemma 3.2(1)]{NSZ} also gives an isomorphism $\Hom_{\mcD^\op}(C,M)\cong\Hom_{\mcH^\op}(\H(C),M)$, functorial on both variables, for all $C\in\mathcal{C}^*$ and $M\in\mcH^\op$. It is actually given by the assignment $f\rightsquigarrow\H(f)$ with the obvious identification $M=\H(M)$. Particularizing to $C=Q$, we get that $\H(Q)$ is injective in $\mcH$ and the mentioned assignment gives an isomorphism, functorial on both variables, $\Hom_\mcD(M,Q)\stackrel{\cong}{\la}\Hom_\mcH(M,\H(Q))$.

Let now $D\in\mcD$ arbitrary. We have a functorial isomorphism $\Hom_\mcD(D,Q)\cong\Hom_\mcD(\tau_\mathbf{t}^{\geq 0}D,Q)$ since $Q\in\mcV$. On the other hand, we have $\H(D)\cong\tau_\mathbf{t}^{\leq 0}\tau_\mathbf{t}^{\geq 0}D$, so that we have a triangle $W[-1]\la\H(D)\la\tau_\mathbf{t}^{\geq 0}D\la W$, where $W=\tau_\mathbf{t}^{>0}(\tau_\mathbf{t}^{\geq 0}D)\in\mcV[-1]$. We then get a functorial isomorphism $\Hom_\mcD(\tau_\mathbf{t}^{\geq 0}D,Q)\stackrel{\cong}{\la}\Hom_\mcD(\H(D),Q)$ since $\Hom_\mcD(?,Q)$ vanishes on $W[-1]$ and $W$, because these two objects are in $\mcV[-1]$. There are then isomorphisms, natural on $D$:
\[
\xymatrix{
  \Hom_\mcD(D,Q) &
  \Hom_\mcD(\tau_\mathbf{t}^{\geq 0}D,Q) \ar[l]_-{\cong} \ar[d]^-{\cong}
  \\ &
  \Hom_\mcD(\H(D),Q) \ar[r]^-{\cong} &
  \Hom_\mcH(\H(D),\H(Q)).
}
\]

It is routine, and left to the reader, to check that the isomorphism from the first to the fourth abelian group in the list is given by the assignment $f\rightsquigarrow\H(f)$. 

(2) From assertion~(1) we get that, for $M\in\mcH$, one has  $\Hom_\mcH(M,\H(Q))\neq 0$ if and only if $\Hom_\mcD(M,Q)\neq 0$. Assertion~(2) then  immediately follows. 

(3) By assertion~(1) and using the fact that $\H_{|\mcV}\dd\mcV\la\mcH$ preserves products, there is an equivalence of categories $\H\dd\Prod_\mcD(Q)\stackrel{\simeq}\la\Prod_\mcH(\H(Q))$. The conclusion follows immediately from the first claim in Lemma~\ref{lem.pi and products} applied to both $Q\in\mcD$ and $\H(Q)\in\mcH$.
\end{proof}

Now we can combine the above observations with the results of~\S\ref{sec:representability}, which we apply to the cohomological functor associated with our $t$-structure.

\begin{thm} \label{thm.AB5 heart}
Let $\mcD$ be a triangulated category with products and coproducts, and let $\mathbf{t}=(\mcU,\mcV)$ be a $t$-structure with heart $\mcH$. Consider the following assertions:

\begin{enumerate}
\item There exists an object $Q\in\mcV$ that satisfies the following conditions:

\begin{enumerate}
\item $\Hom_\mcD(?,Q)$ vanishes on $\mcV[-1]$;
\item $\Hom_\mcD(M,Q)\neq 0$, for all $0\neq M\in\mcH$;
\item $Q$ is pure-injective in $\mcD$.
\end{enumerate}
\item There is a pure-injective object $\hat{Q}\in\mathcal{V}$ such that, for each $V\in\mcV$, there is a triangle  $V\longrightarrow\hat{Q}_V\longrightarrow V'\stackrel{+}{\longrightarrow}$,  with $V'\in\mathcal{V}$ and $\hat{Q}_V\in\Prod(\hat{Q})$.
\item $\mcH$ is an AB5 abelian category with an injective cogenerator and, for each family $(V_i)_{i\in I}$ of objects in $\mcV$, one has that $\tau_\mathbf{t}^{\leq 0}(\coprod_{i\in I}V_i[-1])\in\mcU[1]$.
\item  $\mcH$ is an AB5 abelian category with an injective cogenerator and the cohomological functor $\H\dd\mcD\la\mcH$ preserves coproducts (cf.~Theorem~\ref{thm.classify homological functors}). 
\item $\mcH$ is a Grothendieck category and the cohomological functor $\H\dd\mcD\la\mcH$ preserves coproducts. 
\end{enumerate}

The implications $(1)\Longrightarrow (2)\Longrightarrow (3)\Longleftrightarrow (4)\Longleftarrow (5)$ hold true. When  $\mathcal{D}$ satisfies Brown representability theorem, the implication $(3)\Longrightarrow (1)$ also holds. When $\mathcal{D}$ is standard well generated, all assertions are equivalent. 
\end{thm}
\begin{proof}
Note that, by \cite[Proposition 3.2]{PS1}, the category $\mcH$ is complete.

$(3)\Longleftrightarrow (4)$ is a straightforward consequence of Lemma \ref{lem.H preserves coproducts}.

$(5)\Longrightarrow (4)$ is clear.

$(1)\Longrightarrow (2)$ We put $\hat{Q}=Q$ and will prove that, for each $V\in\mcV$, the canonical morphism $u\dd V\la Q^{\Hom_\mcD(V,Q)}$ has its cone in $\mcV$. 

Thanks to the natural isomorphism $\Hom_\mcH(\H(?),\H(Q))\cong\Hom_\mcD(?,Q)$ and the fact that $\H\dd\mcV\longrightarrow\mcH$ preserves products (see \cite[Lemma 3.1]{PS1}) it immediately follows that the map $\H(u)$ gets identified with the canonical morphism $\H(V)\longrightarrow\H(Q)^{\Hom_\mcH(\H(V),\H(Q))}$, where the product in the codomain is taken in $\mcH$. This last morphism is a monomorphism since $\H(Q)$ is a cogenerator of $\mcH$  by Proposition~\ref{prop.near-partialsilting-object} and the conclusion follows from Lemma~\ref{lem:t-generating tria} applied to the $t$-structure $\mathbf{t}^\op=(\mcV^\op,\mcU^\op)$ in $\mcD^\op$.

$(2)\Longrightarrow (4)$ Since $\mcQ := \Prod_\mcD(\hat{Q})$ is a preenveloping subcategory, it follows from assertion~(2) that $\Prod(\hat{Q})$ is $t$-cogenerating in $\mcV$ (see Definition \ref{def.t-generating}).
Hence we can apply Theorem~\ref{thm.t-generating left Kan extensions} to the $t$-structure $\mathbf{t}^\op=(\mcV^\op,\mcU^\op)$ in $\mcD^\op$ and the $t$-generating subcategory $\mcQ^\op$. Upon taking the opposite categories again and in view of Lemma~\ref{lem.Cont(P)}, we obtain the following diagram, which commutes up to natural equivalence both for the left and for the right adjoints, $F$ is a Serre quotient functor and $G$ its fully faithful left adjoint:
\[
\xymatrix@R=3em{
\mcV \ar[d]_-{H_{\hat{Q}} := (\y_{\mcQ^\op})^\op} \ar@/_/[rr]_-{\H} \ar@{}[rr]|\perp && \mcH \ar@{=}[d] \ar@/_/[ll]_-{\operatorname{inc}} \\
\Cont(\mcQ,\Ab)^\op \ar@/_/[rr]_-{F} \ar@{}[rr]|\perp && \mcH. \ar@/_/[ll]_-{G}
}
\]

Here, $H_{\hat{Q}}$ is the restricted Yoneda functor taking $V\rightsquigarrow\Hom_\mcD(V,?)_{|\mcQ}$ and, in particular, it coincides with the coproduct-preserving initial functor constructed in (the proof of) Theorem~\ref{thm.classify homological functors} for the product equivalence class of $\hat{Q}$. Moreover, $\Cont(\mcQ,\Ab)^\op$ is an AB3* abelian category with enough injectives and its subcategory of injective objects is equivalent to $\Prod(\hat{Q})$ (see Lemma \ref{lem.Cont(P)}). Since $\hat{Q}$ is pure-injective, we know by Proposition~\ref{prop.Positsetski-Stovicek} that it is also AB5.

Finally, we conclude by Proposition~\ref{prop.left adjoint-implies-right adjoint}. It follows that $F$ is a Gabriel localization functor and $\mcH$ is also AB5 with an injective cogenerator. Furthermore, $F$ is a left adjoint, so that $\H\cong F\circ H_{\hat Q}$ preserves coproducts.

$(4)\Longrightarrow (1)$ (when $\mcD$ satisfies Brown representability theorem) Let $E$ be an injective cogenerator of $\mcH$ that, by Proposition~\ref{prop.Positsetski-Stovicek}, is pure-injective in $\mcH$. The contravariant functor $\Hom_\mcD(\H(?),E)\dd\mcD\la\Ab$ takes coproducts to products and, hence, it is representable. Let $Q\in\mcD$ be the object that represents this functor. Since $\H$ vanishes on $\mcV[-1]$ and we have $\H(M)\cong M$, for all $M\in\mcH$, it immediately follows that $Q$ satisfies conditions (1)(a) and (1)(b). On the other hand, $\H$ also vanishes on $\mcU[1]$, which then implies that $\Hom_\mcD(?,Q)$ vanishes in $\mcU[1]$. This is equivalent to say that $Q\in\mcU^\perp[1]=\mcV$. Moreover, we have natural isomorphisms of functors $\mcD^\op\la\Ab$,
\[ \Hom_\mcD(\H(?),E) \cong \Hom_\mcD(?,Q) \cong \Hom_\mcH(\H(?),\H(Q)) \]
by the definition of $Q$ and Proposition~\ref{prop.near-partialsilting-object}. By Yoneda's lemma, we conclude that $\H(Q)\cong E$, and $Q$ is pure-injective again by Proposition~\ref{prop.near-partialsilting-object}.

$(4)\Longrightarrow (5)$ (assuming that $\mcD$ is standard well generated) is a direct consequence of Corollary \ref{cor.AB5 implies Grothendieck in well-gen. triangcats}.
\end{proof}

We conclude with a slightly more general version of Corollary~\ref{cor.AB5 implies Grothendieck in well-gen. triangcats}.

\begin{cor} \label{cor.t-str initial}
Suppose that $\mcD$ is triangulated with coproducts and satisfies Brown representability theorem.
Then the functor $\H\dd\mcD\la\mcH$ in the situation of Theorem~\ref{thm.AB5 heart}(4) is initial in its connected component of $\HFun(\mcD)$ in the sense of Theorem~\ref{thm.classify homological functors}.
\end{cor}

\begin{proof}
This immediately follows from Theorem~\ref{thm.classify homological functors} and Proposition~\ref{prop.near-partialsilting-object}(1).
\end{proof}

\subsection{\texorpdfstring{$t$-structures}{t-structures} with definable co-aisle}
\label{subsec:definable co-aisle}
Now we focus on $t$-structures in compactly generated triangulated categories whose co-aisles are definable. Recall from \cite{K-cohfun} that, when $\mcD$ is such a triangulated category, a functor $F\dd\mcD\la\Ab$ is called a \emph{coherent functor} on $\mcD$ when there is a morphism $\alpha\dd C\la C'$ in $\mcD^c$ and an exact sequence $\Hom_\mcD(C',?)\stackrel{\Hom_\mcD(\alpha ,?)}{\la}\Hom_\mcD(C,?)\la F\la 0$  in $\modf\mcD^\op$. The contents of the following result can be gathered from the papers \cite{K,K-cohfun,Lak}. At the request of the referee  we include a short proof:

\begin{prop} \label{prop.definable subcategory}
Let $\mcD$ be a compactly generated triangulated category.
For a subcategory $\mcY$ of $\mcD$ the following conditions are equivalent:

\begin{enumerate}
\item There is a set  $\mcS\subseteq\Mor(\mcD^c)$ such that $\mcY$ consists of the objects $Y\in\mcD$ such that the map $f^*:=\Hom_\mcD(f,Y)$ is surjective, for all $f\in\mcS$.
\item There is a set $\Sigma\subseteq\Mor(\mcD^c)$ such that  $\mcY$ consists of the objects $Y\in\mcD$ such that  $g^*:=\Hom_\mcD(g,Y)$ is the zero map, for all $g\in\Sigma$
\item There is a set $\mcF$ of coherent functors on  $\mcD$ such that $\mcY$ consists of the objects $Y$ such that $F(Y)=0$, for all $F\in\mcF$.
\end{enumerate}
A subcategory satisfying these equivalent conditions is closed under pure monomorphisms, products, pure epimorphisms and pure-injective envelopes. 
\end{prop}
\begin{proof}
To ease the notation, we denote by $\y:=\y_{\mcD^\op}:\mcD^\op\la\modf\mcD^\op$, $D\rightsquigarrow\mathbf{y}(D)=\Hom_\mcD(D,?)$ the Yoneda embedding. All through the proof we use that $\Hom_\mcD(?,Y)\dd\mcD^\op\la\Ab$ is a cohomological functor for each $Y\in\mcD$.

$(1)\Longleftrightarrow (3)$ Given $\mcS$ as in condition (1), take $\mcF=\{\Coker(\y(f))\mid f\in\mcS\}$. Given $\mcF$ as in condition (3), choose for each $F\in\mcF$ a morphism $g_F\dd C_F\la C'_F$ in $\mcD^c$ such that $F\cong\Coker (\mathbf{y}(g_F))$. Then take $\mcS=\{g_F\mid F\in\mcF\}$.

$(1)\Longleftrightarrow (2)$  Given $\mcS$ as in condition (1), complete each $f\dd C\la C'$ in $\mcS$ to a triangle $C''\stackrel{\alpha_f}{\la}C\stackrel{f}{\la}C'\stackrel{+}{\la}$. Then take $\Sigma=\{\alpha_f\mid f\in\mcS\}$. Given $\Sigma$ as in condition (2), complete each $\sigma\dd C''\la C$ in $\Sigma$ to a triangle $C''\stackrel{\sigma}{\la}C\stackrel{f_\sigma}{\la}C'\stackrel{+}{\la}$. Then take $\mcS=\{f_\sigma\mid\sigma\in\Sigma\}$.

For the final statements, suppose that $X\la Y\la Z\stackrel{+}{\la}$ is a pure triangle in $\mcD$.   Then we have the following commutative diagram of abelian groups with exact rows, for each  morphism $\sigma:C\la C'$ in $\mcD^c$,

\[
\xymatrix{0 \ar[r] & \y(C')(X) \ar[r] \ar[d]^{{\y(\sigma)}_X} & \y(C')(Y) \ar[r] \ar[d]^{{\y(\sigma)}_Y} & \y(C')(Z) \ar[r] \ar[d]^{{\y(\sigma)}_Z}  & 0\\ 0 \ar[r] & \y(C)(X) \ar[r] & \y(C)(Y)\ar[r] & \y(C)(Z) \ar[r] & 0 }
\]

If now $\Sigma\subseteq\Mor(\mcD^c)$ is any set and $\mcY$ is defined as in condition (2), it is clear that if $Y\in\mcY$, i.e. the central vertical arrow of the last diagram is zero for all $\sigma\in\Sigma$,  then the same is true for $X$ and $Z$. That is, the class $\mcY$ is closed under pure monomorphisms and pure epimorphisms. It is clearly closed under products. Finally, by the Fundamental Correspondence Theorem of  [Kra02b], we know that, for $\mcY$ satisfying condition (3), one has that the skeletally small class $\text{PInj}(\mcY)$ of pure-injective objects in $\mcY$ is a closed subset of the Ziegler spectrum of $\mcD$, and $\mcY$ consists of the objects $Y\in\mcD$ such that there is a pure monomorphism $\alpha\dd Y\la\prod_{i\in I}Q_i$, for some family $(Q_i)_{i\in I}$ in $\PInj\mcY$. Suppose now that $Y\in\mcY$, fix such a pure monomorphism $\alpha$ and let $u\dd Y\la Q_Y$ be the pure-injective envelope. If we consider, as in Section~\ref{sect.universal copr-pres}, the generalized Yoneda embedding $h_\pure\dd\mcD\la\Mod\mcD^c$, $D\rightsquigarrow\Hom_\mcD(?,D)_{| \mcD^c}$, we have that $h_\pure(\alpha )$ is a monomorphism into an injective object of $\Mod\mcD^c$ and $h_\pure(u)$ is the injective envelope in $\Mod\mcD^c$. This together with \cite[Theorem 1.8]{K} gives a section $s\dd Q_Y\la\prod_{i\in I}Q_i$ such that $s\circ u=\alpha$. Since $\mcY$ is clearly closed under direct summands, we conclude that $Q_Y\in\mcY$.
\end{proof}

\begin{opr} \label{def.definable subcategory}
A subcategory $\mcY$ of a compactly generated triangulated category is said to be \emph{definable} when it satisfies any of  conditions (1)--(3) of last proposition. 
\end{opr}

\begin{rem} \label{rem:definability and closure properties}
If a compactly generated triangulated category $\mcD$ has an enhancement which allows for a good calculus of homotopy colimits (in the form of a stable derivator as in~\cite{Lak} or a stable $\infty$-category), then some of the closure properties from Proposition~\ref{prop.definable subcategory} in fact characterize definable classes. The following statements are equivalent for $\mcY\subseteq\mcD$ in that case:
\begin{enumerate}
\item $\mcY$ is definable,
\item $\mcY$ is closed under products, pure monomorphisms and  directed homotopy colimits,
\item $\mcY$ is closed under products, pure monomorphism and pure epimorphism.
\end{enumerate}
The equivalence between~(1) and~(2) is a part of \cite[Theorem 3.11]{Lak}. The implication $(1)\Rightarrow(3)$ is an elementary consequence of the definition of a definable subcategory and does not require any enhancement. For $(3)\Rightarrow(2)$, one can follow the strategy from~\cite[Theorem 4.7]{LV} and prove that for any coherent directed diagram $X$, the colimit morphism $\coprod_{i\in I}X_i\la \hocolim X$ is a pure epimorphism in $\mcD$. This follows essentially from~\cite[Definition~5.1 and Proposition 5.4]{SSV}.
\end{rem}

\begin{rem} \label{rem:definable coaisles}
The previous remark immediately implies that if $\mcD$ is compactly generated triangulated and has a suitable enhancement (a stable derivator or a stable $\infty$-category) and if $\mathbf{t}=(\mcU,\mcV)$ is a $t$-structure, then the following are equivalent:
\begin{enumerate}
\item $\mcV$ is closed under pure epimorphisms,
\item $\mcV$ is closed under directed homotopy colimits and pure monomorphisms,
\item $\mcV$ is definable.
\end{enumerate}
The main point is that if $\mcV$ is closed under pure epimorphisms, it is also closed under pure monomorphisms since $\mcV[-1]\subseteq\mcV$ and $\mcV$ is closed under extensions.

Under certain additional assumptions, the conditions above are also known to be equivalent to the apparently weaker condition
\begin{enumerate}
\item[(2')] $\mcV$ is closed under directed homotopy colimits.
\end{enumerate}
This holds by \cite[Theorem 4.6]{Lak} if $\textbf{t}$ is left non-degenerate. It also holds whenever $\mcD$ is an algebraic compactly generated triangulated category. We will only sketch the argument in this case and discuss the details elsewhere later. The point is that $\mcD$ is the derived category $\mcD(\mcA)$ of a small dg category $\mcA$ and $\mcV$ turns out to be a class of dg-$\mcA$-modules which is the right hand side class of a cotorsion pair and closed under direct limits (see e.g.\ \cite[\S2.2 and Lemma 3.4]{StPo}). Such a class, however, must be definable in the category of dg-$\mcA$-modules by~\cite[Theorem 6.1]{Saroch-tools}, and one routinely checks that it is then also definable in $\mcD=\mcD(\mcA)$.

Unfortunately, we do not know whether (2) is equivalent to (2') in general.
\end{rem}

An important feature of $t$-structures with definable co-aisle in compactly generated triangulated categories is that the class of pure-injective objects in the co-aisle is $t$-cogenerating (recall Definition~\ref{def.t-generating}). In fact, we can state that fact more generally. If $\mcV$ is a class of objects in a triangulated category with products, we denote by $\PInj\mcV\subseteq\mcV$ the class of all pure-injective objects in $\mcV$.

\begin{lem} \label{lem.def coaisles are t-cogen}
Let $\mcD$ be a standard well generated triangulated category and $\mathbf{t}=(\mcU,\mcV)$ be a $t$-structure whose co-aisle $\mcV$ is closed under pure epimorphisms and pure-injective envelopes (see Definition \ref{def.pure-triangle in standard well-gen. categs} and Corollary \ref{cor.pure-injective envelopes in standard well-gen. triangcats}). Then:
\begin{enumerate}
\item[(a)] $\PInj\mcV$ is $t$-cogenerating in $\mcV$,
\item[(b)] there exists an object $\hat{Q}\in\mcV$ such that $\PInj\mcV=\Prod_\mcD(\hat{Q})$.
\end{enumerate}
\end{lem}

\begin{proof}
If $V\in\mcV$ and $u\dd V\la Q_V$ is a pure injective envelope, then $Q_V\in\PInj\mcV$ and the third term in a triangle $V \stackrel{u}\la Q_V\stackrel{p}\la V' \laplus$ is again in $\mcV$ since $p$ is a pure epimorphism. This proves assertion~(a).

We consider $\mcT={^\perp h_\pure(\PInj\mcV)} \subseteq \mcA_\pure(\mcD)$. Then $\mcT$ is a hereditary torsion class in the Grothendieck category $\mcA_\pure(\mcD)$ and, as a consequence, we have an injective object $E$ in this latter category such that the associated torsionfree class $\mcF=\mcT^\perp$ consists of the subobjects of objects in $\Prod(E)$. But then $E =h_\pure(\hat{Q})$, for some pure-injective object $\hat{Q}$ which is necessarily in $\mcV$.
Indeed we have a section  $E=h_\pure (\hat{Q})\rightarrowtail \prod_{i\in I}h_\pure (Q_i)$ in $\mcA_\pure (\mcD)$, for some family $(Q_i)_{i\in I}$ in $\PInj\mcV$, that is the image under $h_\pure\dd\mcD\la\mcA_\pure (\mcD)$ of a section $s\dd\hat{Q}\rightarrowtail\prod_{i\in I}Q_i$, by Lemma~\ref{lem.pi for universal homological functor}. This proves assertion~(b).
\end{proof}

\begin{rem} \label{def.t-cogenerated coaisle}
Suppose again that we are in the situation of Remark~\ref{rem:definable coaisles}. The last lemma says that $\PInj\mcV$ is $t$-cogenerating if condition~(2) holds. The subtle point is that the conclusion of Lemma~\ref{lem.def coaisles are t-cogen} can be derived already from the a priori weaker condition (2') (which, however, can be stated only using an enhancement of $\mcD$) and, thus, holds for any homotopically smashing $t$-structure in the language of~\cite{SSV,Lak}.

To see this, one can use essentially the same argument as in~\cite[Proposition 3.7]{Lak}. If $V\in\mcV$, then there is a set $I$ and an ultrafilter $\mcF$ on $I$ such that the coherent ultrapower $V^S/\mcF$ is pure-injective and the diagonal morphism $V\la V^S/\mcF$ is a pure monomorphism. In fact, the triangle $V \la V^S/\mcF \la V' \laplus$ is by definition of the coherent ultrapower~\cite[\S2.2]{Lak} a directed homotopy colimit of split triangles
\[ V \overset{d_J}\la V^J \la V'_J \laplus, \]
where $J\subseteq I$ runs over the elements of $\mcF$ and $d_J$ are the diagonal embeddings. If $\mcV$ is closed under directed homotopy colimits, then $V^S/\mcF\in\PInj\mcV$ and $V'\in\mcV$.
\end{rem}

The main result of this subsection is now an easy consequence of the results in previous (sub)sections.

\begin{thm} \label{thm.Groth.heart-for-definable-coaisle}
Let $\mcD$ be a standard well generated triangulated category and $\mathbf{t}=(\mcU,\mcV)$ be a $t$-structure such that the class $\PInj\mcV$ of all pure-injective objects in $\mcV$ is $t$-cogenerating in $\mcV$---these assumptions are satisfied e.g.\ if $\mcD$ is compactly generated and either
\begin{enumerate}
\item $\mcV$ is definable or
\item $\mcD$ has a suitable enhancement (see Remarks~\ref{rem:definable coaisles} and~\ref{def.t-cogenerated coaisle}) and $\mcV$ is closed under taking directed homotopy colimits.
\end{enumerate}
Then the heart $\mcH=\mcU\cap\mcV$ is a Grothendieck category and $\H\dd\mcD\la\mcH$ preserves coproducts.

Moreover, if we fix a Verdier quotient functor $q\dd\mcC\la\mcD$ such that $\mcC$ is a compactly generated triangulated category and $\Ker(q)$ is the localizing subcategory of $\mcC$ generated by a set of objects, then $\H(q(\mcC^c))$ is a skeletally small class of generators of $\mcH$.
\end{thm}

\begin{proof}
By the same argument as for Lemma~\ref{lem.def coaisles are t-cogen}(b), we find a pure-injective object $\hat{Q}\in\mcV$ such that $\PInj\mcV=\Prod_D(\hat{Q})$. Then such an object $\hat{Q}$ satisfies the assumption of Theorem~\ref{thm.AB5 heart}(2) and $\mcH$ is a Grothendieck category by Theorem~\ref{thm.AB5 heart}(5).

As for the class of generators, note that, by Proposition \ref{prop.universal AB5 homology} and its proof, we know that $\mcA_\pure(\mcD)$ is a Gabriel quotient of $\Mod\mcC^c$ and if $q'\dd\Mod\mcC^c\la\mcA_\pure(\mcD)$ is the corresponding Gabriel localization functor, then $q'\circ\y_\mcC\cong h_\pure\circ q$, where $\y_\mcC\dd\mcC\la\Mod\mcC^c$ is the Yoneda functor. Since $\y(\mcC^c)$ is a skeletally small class of generators of $\Mod\mcC^c$, we conclude that $h_\pure (q(\mcC^c))$ is a skeletally small class of generators of $\mcA_\pure (\mcD)$. Applying now Corollary \ref{cor.AB5 implies Grothendieck in well-gen. triangcats} we conclude that if $\tilde{q}\dd\mcA_\pure(\mcD)\la\mcH$ is the (uniquely determined up to natural isomorphism) Gabriel localization functor such that $\tilde{q}\circ h_\pure\cong\H$, then $\tilde{q}(h_\pure (q(\mcC^c)))=\H(q(\mcC^c))$ is a skeletally small class of generators of $\mcH$. 
\end{proof}

As immediate consequences, we get:
  
\begin{cor} \label{cor.Grothendieck hearts for definable co-aisle}
Let $\mcD$ be a compactly generated triangulated category and $\mathbf{t}=(\mcU,\mcV)$ be a $t$-structure in $\mcD$ with definable co-aisle. The heart $\mcH=\mcU\cap\mcV$ is a Grothendieck category for which $\H(\mcD^c)$ is a skeletally small class of generators. 
\end{cor}
  
\begin{cor} \label{cor.definable-coaisles-Grothcats}
Let $\mcG$ be a Grothendieck category with generator $X$, and let $\mathbf{t}=(\mcU,\mcV)$ be a $t$-structure in the derived category $\mcD(\mcG)$ such that $\mcV$ is closed under taking pure epimorphisms and pure-injective envelopes.  The heart $\mcH=\mcU\cap\mcV$ is a Grothendieck category on which $\H (\thick_{\mcD(\mcG)}(X))$ is a skeletally small class of generators. Here $\thick_{\mcD(\mcG)}(X)$ is the smallest thick subcategory of $\mcD(\mcG)$ containing $X$.
\end{cor}
\begin{proof}
By the usual Gabriel-Popescu's theorem, if $R=\text{End}_\mcG(X)$ we have a Gabriel localization functor $q\dd\Mod R\la\mcG$ that takes $R$ to $X$. The induced triangulated functor $q\dd\mcD(\Mod R)\la\mcD(\mcG)$ satisfies that
\[
q\big(\mcD(\Mod R)^c\big)=q\big(\thick_{\mcD(\Mod R)}(R)\big)\subseteq\thick_{\mcD(\mcG)}(X).
\]
The result then follows since, by Theorem \ref{thm.Groth.heart-for-definable-coaisle}, $\H\big(q(\mcD(\Mod R)^c)\big)$ is a skeletally small class of generators of $\mcH$. 
\end{proof}

\subsection{Suspended ideals of the category of compact objects}

In this subsection, we assume that $\mcD$ is a compactly generated triangulated category we denote by $\y_\mcD\dd\mcD\la\Mod\mcD^c$ the Yoneda functor. In~\cite[Corollary 12.5]{K-cohq}, Krause classified smashing localizations of $\mcD$ in terms of so-called exact ideals of $\mcD^c$. Here we establish an analogous classification of $t$-structures with definable co-aisle in $\mcD$.

\begin{opr} \label{def.suspended ideal}
Let $\mcD_0$ be a triangulated category and $\mcI$ a two-sided ideal of $\mcD_0$. The ideal is called \emph{saturated} (see \cite[Definition 8.3]{K-cohq}) if whenever we have a triangle $X_2\stackrel{u}\la X_1\stackrel{v}\la X_0\la X_2[1]$ and a morphism $f\dd X_1 \la Y$ in $\mcD_0$, then
\[ f\circ u,\; v\in\mcI \implies f\in\mcI. \]

\noindent
The ideal $\mcI$ is called \emph{suspended} if it satisfies the following three conditions:
\begin{enumerate}
\item $\mcI$ is idempotent, i.e.\ $\mcI^2=\mcI$,
\item $\mcI$ is saturated, and
\item $\mcI[1]\subseteq\mcI$.
\end{enumerate}
\end{opr}

\begin{thm} \label{thm.classification definable}
Let $\mcD$ be a compactly generated triangulated subcategory. Then there is a bijective correspondence between
\begin{enumerate}
\item the $t$-structures $\mathbf{t}=(\mcU,\mcV)$ in $\mcD$ with $\mcV$ definable, and
\item suspended ideals $\mcI\subseteq\mcD^c$,
\end{enumerate}
which is given by the assignments
\begin{align*}
\mcV &\rightsquigarrow \mcI = \{ f\dd C\la D \mathrm{~in~} \mcD^c \mid \Hom_\mcD(f,\mcV) = 0 \}
\quad \textrm{and} \\
\mcI &\rightsquigarrow \mcV = \{ V\in\mcD \mid \Hom_\mcD(\mcI,V) = 0 \}.
\end{align*}
\end{thm}

At the first step, we combine existing results from the literature in order to relate saturated ideals to definable subcategories and other already discussed notions.

\begin{prop} \label{prop.definable via ideals}
There are bijective correspondences between
\begin{enumerate}
\item definable subcategories $\mcV\subseteq\mcD$,
\item saturated ideals $\mcI\subseteq\mcD^c$,
\item Serre subcategories $\mcS\subseteq\modf\mcD^c$, and
\item hereditary torsion pairs $(\mcT,\mcF)$ of finite type in $\Mod\mcD^c$.
\end{enumerate}
The bijection between (1) and (2) is given by the same rule as in Theorem~\ref{thm.classification definable}, the map from (2) to (3) is given by
\[ \mcI \rightsquigarrow \mcS = \{ \Img\y_\mcD(f) \mid f\in\mcI \}, \]
and the correspodence between (3) and (4) is given by

\[ \mcS \rightsquigarrow (\varinjlim\mcS,\mcS^\perp) \qquad \textrm{and} \qquad (\mcT,\mcF) \rightsquigarrow \mcT\cap\modf\mcD^c, \]
where $\varinjlim\mcS$ is the class of direct limits of objects from $\mcS$.
\end{prop}

\begin{proof}
The bijections between (1), (2) and (3) are essentially the contents of the Fundamental Correspondence in \cite{K-cohfun}. The only thing we need to add is a reference to~\cite[Lemma 8.4]{K-cohq}, where one proves that the cohomological ideals as in the Fundamental Correspondence are precisely the saturated ideals.
Finally, the correspondence between (3) and (4) has been obtained in \cite[Lemma 2.3]{K-loccoh} or \cite[Theorem 2.8]{Herzog}.
\end{proof}

Next we recall the following result, which was first proved in \cite[Proposition 3.11]{K} and also appears in~\cite[Proposition 4.5]{AMV} with exactly the same terminology as we use.
A closely related result can be also found in~\cite[Theorem 4.3]{AI}.
For the special case of algebraic triangulated categories, a considerably stronger version was recently obtained as a part of the main theorem of~\cite{LV}.

\begin{prop} \label{prop.definable is preenveloping}
A definable subcategory $\mcV$ of a compactly generated triangulated category $\mcD$ is preenveloping.
\end{prop}

\begin{cor} \label{cor.definable t-str from ideal}
Let $\mcI$ be a suspended ideal of $\mcD^c$. Then $\mcV = \{ V\in\mcD \mid \Hom_\mcD(\mcI,V) = 0 \}$ is a definable co-aisle of a $t$-structure in $\mcD$.
\end{cor}

\begin{proof}
We already know that $\mcV$ is definable, preenveloping and $\mcV[-1]\subseteq\mcV$. Once we prove that $\mcV$ is closed under extensions (hence cosuspended and closed under summands), the conclusion will follow from the dual of~\cite[Proposition 3.11]{SaSt} or~\cite[Proposition 4.5]{AMV}. To that end, suppose that we have a triangle $V'\stackrel{u}\la V\stackrel{v}\la V''\laplus$ with $V',V''\in\mcV$ and suppose that $g\in\mcI$. Since $\mcI=\mcI^2$, we can express $g\dd C\la D$ as $g_1g_2$ with $g_1,g_2\in\mcI$ (a~priori, the elements of $\mcI^2$ are finite sums $\sum_{i=1}^n g_{1,i}g_{2,i}$ with all $g_{1,i}\dd E_i\la D$ and $g_{2,i}\dd C\la E_i$ belonging to $\mcI$, but as $\mcD^c$ is additive, we combine these to $g_1\dd \bigoplus_{i=1}^nE_i\la D$ and $g_2\dd C\la\bigoplus_{i=1}^nE_i$).
Then, given any morphism $h\dd D\la V$, we have that $vhg_1=0$ as $V''\in\mcV$, so that $hg_1=uh'$ for some morphism $h'$. Thus, $hg=uh'g_2=0$ since $V'\in\mcV$.
\end{proof}

To finish the proof of Theorem~\ref{thm.classification definable} we also need the following useful result.

\begin{prop} \label{prop.tor-pair vs t-structure}
Let $\mathbf{t}=(\mcU,\mcV)$ be a $t$-structure in $\mcD$ with $\mcV$ definable, and suppose that $(\mcT,\mcF)$ is the hereditary torsion pair of finite type in $\Mod\mcD^c$ which corresponds to $\mcV$ in the sense of Proposition~\ref{prop.definable via ideals}. Then
\[ \mcV=\y_\mcD^{-1}(\mcF) \qquad\textrm{and}\qquad \mcU[1]=\y_\mcD^{-1}(\mcT). \]
\end{prop}

\begin{proof}
The equality $\mcV=\y_\mcD^{-1}(\mcF)$ is in fact true for any definable class in $\mcD$. Indeed, given any morphism $g$ in $\mcD^c$ and $V\in\mcD$, we use the Yoneda lemma and the fp-injectivity of $\y_\mcD(V)$ (\cite[Lemma 1.6]{K}) to see that $\Hom_\mcD(g,V)=0$ if and only if $\Hom_{\Mod\mcD^c}(\y_\mcD(g),\y_\mcD(V))=0$ if and only if $\Hom_{\Mod\mcD^c}(\Img\y_\mcD(g),\y_\mcD(V))=0$. 
Since $\mcF=\{\Img\y_\mcD(g)\mid g\in\mcI\}^\perp$ in $\Mod\mcD^c$ by Proposition~\ref{prop.definable via ideals}, we infer that $V\in\mcV$ if and only if $\y_\mcD(V)\in\mcF$.

Regarding the other equality, we first prove the inclusion $\mcU[1]\subseteq\y_\mcD^{-1}(\mcT)$. So suppose that $U\in\mcU$.
Then Lemma~\ref{lem.pi for universal homological functor} yields that the injective $\mcD^c$-modules in $\mcF$ are up to isomorphism of the form $\y_\mcD(V)$ for $V\in\PInj\mcV$, and that
\[ \Hom_{\Mod\mcD^c}\big(\y_\mcD(U[1]),\y_\mcD(V)\big) \cong \Hom_\mcD(U[1],V) = 0 \]
for each $V\in\PInj\mcV$. Since $\mcF$ is closed in $\Mod\mcD^c$ under injective envelopes by~\cite[Proposition VI.3.2]{St}, we have $\y_\mcD(U[1])\in\mcT$, as desired.

To prove the remaining inclusion, suppose that $X\in\mcD$ is such that $\y_\mcD(X[1])\in\mcT$ and consider a triangle $U[1]\la X[1]\la V\la U[2]$ with $U\in\mcU$ and $V\in\mcV$. Then we get an exact sequence
\[ \y_\mcD(X[1]) \la \y_\mcD(V) \la \y_\mcD(U[2]). \]
Since $U[2]\in\mcU[1]$, the outer terms are in the hereditary torsion class $\mcT$ by the above and we have $\y_\mcD(V)\in\mcT\cap\mcF=0$. Thus, $V=0$ since $\y_\mcD$ reflects isomorphisms and, consequently, $X\in\mcU$.
\end{proof}

We complement the proposition with some consequences, which will be used either here or later in~\S\ref{subsec:comp gen}.

\begin{cor}\label{cor.U closed under pure epi/mono}
Let $\mathbf{t}=(\mcU,\mcV)$ be a $t$-structure in $\mcD$ with $\mcV$ definable. Then $\mcU$ is closed under pure monomorphisms and pure epimorphisms.
\end{cor}

\begin{proof}
If $X\la U\la Y\laplus$ is a pure triangle with $U\in\mcU$, then
\[ 0\la\y_\mcD(X[1])\la\y_\mcD(U[1])\la\y_\mcD(Y[1])\la 0 \]
is a short exact sequence with the middle term in $\mcT$. Since $\mcT$ is a hereditary torsion class, we have that $\y_\mcD(X[1]),\y_\mcD(Y[1])\in\mcT$, and conclude by Proposition~\ref{prop.tor-pair vs t-structure}.
\end{proof}

\begin{cor} \label{cor.H sends pure epis in U to epis}
Let $\mathbf{t}=(\mcU,\mcV)$ be a $t$-structure in $\mcD$ with heart $\mcH$ and suppose that $\mcV$ is definable. If $p\dd U\la U'$ is a pure epimorphism in $\mcD$ with $U$, $U'\in\mcU$, then $\H(p)\dd\H(U)\la\H(U')$ is an epimorphism in $\mcH$.
\end{cor}

\begin{proof}
This is a direct consequence of Corollary~\ref{cor.U closed under pure epi/mono} and Lemma~\ref{lem:t-generating tria}.
\end{proof}

\begin{cor}\label{cor.description of U for definable}
Let $\mathbf{t}=(\mcU,\mcV)$ be a $t$-structure in $\mcD$ with $\mcV$ definable and let $X\in\mcD$. Then $X\in\mcU$ if and only if each morphism $f\dd C\la X[1]$ with $C$ compact factors as $f=f'\circ g$ with $g\in\mcI$.
\end{cor}

\begin{proof}
We have proved that $X\in\mcU$ if and only if $\y_\mcD(X[1])\in\mcT$ if and only if for each $f\dd C\la X[1]$ with $C\in\mcD^c$, the map $\y_\mcD(f)\dd\y_\mcD(C)\la\y_\mcD(X[1])$ factors through $\mcS=\mcT\cap\modf\mcD^c$.

By Proposition~\ref{prop.definable via ideals}, objects of $\mcS$ are of the form $\Img\y_\mcD(g)$ for some $g\dd C'\la D'$ from $\mcI$. Thus, if $\y_\mcD(f)$ factors through $\Img\y_\mcD(g)$, we have the solid part of the following diagram
\[
\xymatrix{
\y_\mcD(C) \ar[r] \ar@{.>}[d]_-{h} & \Img\y_\mcD(f) \ar[r] \ar@{=}[d] & \y_\mcD(X[1]) \\
\y_\mcD(C') \ar@{->>}[r] & \Img\y_\mcD(g)\; \ar@{>->}[r] & \y_\mcD(D'), \ar@{.>}[u] \\
}
\]
and dotted arrows exist since $\y_\mcD(C)$ is projective and $\y_\mcD(X[1])$ is fp-injective in $\Mod\mcD^c$ (\cite[Lemma 1.6]{K}). Hence, $f$ factors through $gh\in\mcI$. If, conversely, $f$ factors through $g\in\mcI$, then clearly $\y_\mcD(f)$ factors through $\Img\y_\mcD(g)\in\mcS$.
\end{proof}

Now we can finish the proof of the theorem.

\begin{proof}[Proof of Theorem~\ref{thm.classification definable}]
Suppose that $\mathbf{t}=(\mcU,\mcV)$ is a $t$-structure in a compactly generated triangulated category $\mcD$ such that $\mcV$ is definable, and put
\[ \mcI = \{ f\dd C\la D \mathrm{~in~} \mcD^c \mid \Hom_\mcD(f,\mcV) = 0 \}. \]
Then clearly $\mcI[1]\subseteq\mcI$ is a saturated ideal of $\mcD^c$. Suppose now that $g\dd C\la D$ belongs to $\mcI$; in particular $\Hom_\mcD(g,\coa{0}D)=0$. Then $g$ factors through $g'\dd C\la\ai{-1}D$, which in turn factors through $(g''\dd C\la D')\in\mcI$ thanks to Corollary~\ref{cor.description of U for definable}. Since both $g''$ and the composition $D'\la\ai{-1}D\la D$ belong to $\mcI$, we have $g\in\mcI^2$. It follows that $\mcI$ is a suspended ideal.

Conversely, given a suspended ideal, the class
$\mcV = \{ V\in\mcD \mid \Hom_\mcD(\mcI,V) = 0 \}$
is a definable co-aisle by Corollary~\ref{cor.definable t-str from ideal}.

The bijective correspondence then clearly comes up as the corresponding restriction of the bijective correspondence between (1) and (2) in Proposition~\ref{prop.definable via ideals}.
\end{proof}

\subsection{\texorpdfstring{$t$-structures}{t-structures} with right  adjacent \texorpdfstring{co-$t$-structure}{co-t-structure}}
\label{subsec:co-t-structures}

A \emph{co-$t$-structure} (also called a \emph{weight structure}; see~\cite{Bo10,Pauk08}) in a triangulated category $\mcD$ is a pair $\mathbf{c}=(\mcV,\mcW)$ of subcategories which are closed under direct summands and satisfy the following conditions:

\begin{enumerate}
\item[(i)] $\Hom_\mcD(V,W[1])=0$, for all
$V\in\mcV$ and $W\in\mcW$;
\item[(ii)] $\mcV[-1]\subseteq\mcV$ (or $\mcW[1]\subseteq\mcW$);
\item[(iii)] For each $X\in\Ob(\mcD)$, there is a triangle $V\la X\la
Y\laplus$ in $\mcD$, where
$V\in\mcV$ and $Y\in\mcW[1]$.
\end{enumerate}
 
Then one has that $\mcV^\perp=\mcW[1]$ and $\mcV={^{\perp}\mcW[1]}$. The intersection $\mcC=\mcV\cap\mcW$ is called the \emph{co-heart} of $\mathbf{c}$. Given a $t$-structure $\mathbf{t}=(\mcU,\mcV)$, we say that $\mathbf{t}$ has a \emph{right adjacent co-$t$-structure} when the pair $(\mcV,\mcW)=(\mcV,\mcV^\perp[-1])$ is a co-$t$-structure in $\mcD$. Note that the intersection $\mcV\cap\mcV^\perp[-1]$ makes sense even if $(\mcV,\mcV^\perp[-1])$ is not a co-$t$-structure. It is sometimes called the \emph{right co-heart} of $\mathbf{t}$. 
  
  In this subsection we will give a criterion for the heart of a $t$-structure with right adjacent co-$t$-structure to be a Grothendieck category. We will show also that $t$-structures with definable co-aisle in a compactly generated triangulated category always have a right adjacent co-$t$-structure. We will show also that the $t$-structure cogenerated by a pure-injective object in a standard well generated triangulated category (see Proposition~\ref{prop.t-structure from pi}) has a right adjacent co-t-structure. We need the following elementary result of Category Theory whose proof is left to the reader.
  
\begin{lem} \label{lem.inj-cogenerator}
Let $\mcA$ be an AB3* abelian category with a cogenerator and enough injectives. Then it has an injective cogenerator. 
\end{lem}
  
We can now give the first main result of this subsection.
   
   %

\begin{prop} \label{prop.Groth-heart-adjacent-cotstr}
Let $\mcD$ be a triangulated category with products, let $\mathbf{t}=(\mcU,\mcV)$ be a $t$-structure with right adjacent co-$t$-structure $\mathbf{c}=(\mcV,\mcW)$. Denote by $\mcH=\mcU\cap\mcV$ and $\mcC=\mcV\cap\mcW$, respectively, the heart of $\mathbf{t}$ and the co-heart of $\mathbf{c}$. The following assertions hold:
  \begin{enumerate}
    \item $\mcH$ is an AB3* abelian category with enough injectives and the functor $\H\dd\mcD\la\mcH$ restricts to a category equivalence $\mcC\stackrel{\simeq}{\la}\Inj\mcH$.
		\item $\mcH$ has an injective cogenerator (and is AB5, resp.\ is a Grothendieck category) if, and only if, there is an object $Q\in\mcC$ (that is pure-injective, resp.\ accessible pure-injective) such that $\mcC=\Prod(Q)$. 
   \item Suppose that there is a pure-injective object $E$ such that $\PInj\mcD=\Prod_\mcD(E)$. Then $\mcH$ is AB5 with an injective cogenerator (resp. a Grothendieck category) if, and only if, $\mcC$ consists of pure-injective (resp. accessible pure-injective) objects.
   \item When $\mcD$ is standard well generated, $\mcH$ is a Grothendieck category if, and only if, $\mcC$ consists of pure-injective objects. 
  \end{enumerate}
\end{prop}

\begin{proof}
(1) The heart of any $t$-structure in $\mcD$ is AB3* (see \cite[Proposition 3.2]{PS1}). That $\H$ restricts to a fully faithful functor $\mcC\la\Inj\mcH$ follows from Proposition \ref{prop.near-partialsilting-object}(1)
(see also \cite{Bo16} and the dual of \cite[Lemma 3.2]{NSZ}).

By definition of co-$t$-structure, we have for each $X\in\mcV$ a triangle $V[-1]\la X\stackrel{\alpha}{\la} W\la V$, with $V\in\mcV$ and $W\in\mcW$. Consequently $W\in\mcV\cap\mcW=\mcC$ since $X,V\in\mcV$, and hence $\mcC$ is $t$-cogenerating in $\mcV$ (Definition~\ref{def.t-generating}). Hence, whenever $X\in\mcH$, the morphism $\H(\alpha)\dd X\la \H(W)$ is a monomorphism in $\mcH$ by the dual of Lemma~\ref{lem:t-generating tria}. This shows that $\H(\mcC)\subseteq\Inj\mcH$ is a cogenerating class and that $\mcH$ has enough injectives.

It remains to show that $\H\dd\mcC\la\Inj\mcH$ is essentially surjective. To that end, let $E\in\Inj\mcH$. By the above paragraph, we have a split embedding of $E$ into $\H(W)$ for some $W\in\mcC$. The corresponding idempotent morphism $e\dd\H(W)\la\H(W)$ whose image is $E$ uniquely lifts to an idempotent endomorphism $e'\dd W\la W$ since $\H\dd\mcC\la\Inj\mcH$ is fully faithul. Since $\mcD$ has split idempotents by~\S\ref{subsec:triangulated}, $W$ has a direct summand $W'\in\mcC$ whose image under $\H$ is clearly $E$.
	
(2) Note that both $\mcV=\mcU[1]^\perp$ and $\mcW=\mcV[-1]^\perp$ are closed under products in $\mcD$, and consequently so is $\mcC$.
It is clear now that $\mcH$ has an injective cogenerator (which is pure-injective, resp. accessible pure-injective) if, and only if, there is $E\in\Inj\mcH$ (which is pure-injective, resp. accessible pure-injective) such that $\Inj\mcH=\Prod_\mcH(E)$.
This is by~(1) further equivalent to the existence of a $Q$ (that is pure-injective, resp. accessible pure-injective) such that $\mcC=\Prod_\mcD(Q)$.

(3) Bearing in mind Proposition \ref{prop.Positsetski-Stovicek} and the equivalence of categories $\H\dd\mcC\stackrel{\simeq}{\la}\Inj\mcH$, the 'only if' part is clear.
	
As for the 'if' part we then have a class $\Inj\mcH=\H(\mcC)$ of pure-injective (resp.\ accessible pure-injective) injective cogenerators of $\mcH$ and, using Lemma~\ref{lem.inj-cogenerator}, the task is then reduced to prove that $\mcH$ has a cogenerator.
Let us do it. By hypothesis, for any $Q\in\mcC$ we have a section $s_Q\dd Q\la E^{I_Q}$, for some set $I_Q$.
Moreover, from the definition of co-$t$-structure, we get a triangle $V_E\stackrel{\lambda}{\la} E\stackrel{\rho}{\la} W_E[1]\stackrel{+}{\la}$, which in turn yields another one $V_E^{I_Q}\stackrel{\lambda^{I_Q}}{\la} E^{I_Q}\stackrel{\rho^{I_Q}}{\la} W_E^{I_Q}[1]\stackrel{+}{\la}$, where $V_E\in\mcV$ and $W_E\in\mcW$. We then have that $\rho^{I_Q}\circ s_Q=0$ since $Q\in\mcV$ and $W_Q^{I_Q}[1]\in\mcW[1]=\mcV^\perp$. So there is a morphism, necessarily a section, $u_Q\dd Q\la V_E^{I_Q}$ such that $\lambda^{I_Q}\circ u_Q=s_Q$.  But $\H\dd \mcV\la\mcH$ preserves products (see \cite[Lemma 3.1]{PS1}), and so we get a section  $\H(u_Q)\dd \H(Q)\la \H(V_E)^{I_Q}$, where the product in the codomain is taken in $\mcH$. It immediately follows that $\H(V_E)$ is a cogenerator of $\mcH$. 	

(4) It follows from Lemma~\ref{lem.pi for universal homological functor} and the proof of Proposition~\ref{prop.universal AB5 homology} that there is an $E\in\PInj \mcD$ such that $\PInj\mcD=\Prod (E)$.   The result is then a direct consequence of assertion (2) and Lemma~\ref{lem.pi in std well gen}.
\end{proof}


The final result of this subsection shows that $t$-structures with a definable co-aisle have a right adjacent co-$t$-structure. 

\begin{thm} \label{thm.definable-have-adjacent-cotstr}
Let $\mcD$ be a compactly generated triangulated category and $\mathbf{t}=(\mcU,\mcV)$ be a $t$-structure. Consider the following assertions:

\begin{enumerate}
\item The co-aisle $\mcV$ is definable.
\item $\mathbf{t}$ has a right adjacent co-$t$-structure $\mathbf{c}=(\mcV,\mcW)$ such that $\mcV ={^{\perp}\PInj\mcW[1]}$.
\end{enumerate}

The implication $(1)\Longrightarrow (2)$ holds true. Moreover, when $\mcD$ is the base of a stable derivator or the homotopy category of a stable $\infty$-category, both assertions are equivalent and they are also equivalent to
\begin{enumerate}
\setcounter{enumi}{2}
\item The co-aisle $\mcV$ is closed under pure epimorphisms
(see Remark \ref{rem:definable coaisles}).
\end{enumerate}
\end{thm}

\begin{proof}
$(2)\Longrightarrow (3)\Longleftrightarrow (1)$ when $\mcD$ has the mentioned enhancement:
The equality $\mcV ={^\perp\PInj\mcW}$ guarantees that $\mcV$ is closed under pure epimorphisms. Then use Remark~\ref{rem:definable coaisles}.

$(1)\Longrightarrow (2)$ Consider the suspended ideal $\mcI$ corresponding to $\mcV$ by the bijection of Theorem \ref{thm.classification definable}. We then consider the associated TTF triple $(\mcC_\mcI,\mcT_\mcI,\mcF_\mcI)$ in $\Mod\mcD^c$ (see \cite[Subsection 4.2]{PSV}). Recall from \emph{op.\,cit.} that $\mcT_I$ consists of the functors $T\dd(\mcD^c)^\op\la\Ab$ such that $T(s)=0$, for all morphisms $s\in\mcI$. In other words, $\mcT_\mcI$ is the essential image of the forgetful functor $\Mod\frac{\mcD^c}{\mcI}\rightarrowtail\Mod\mcD^c$.

If $\y:=\y_\mcD\dd\mcD\la\Mod\mcD^c$ is the Yoneda functor, it follows from the Yoneda lemma that $\y V\in\mcT_\mcI$ if and only if $\Hom_\mcD(\mcI,V)=0$. That is, we have an equality $\mcV=\y^{-1}(\mcT_\mcI)$ (here we warn the reader that the class $\mcT_\mcI$ may differ from the class $\mcF$ from Proposition~\ref{prop.tor-pair vs t-structure}, despite the fact that $\y^{-1}(\mcT_\mcI)=\mcV=\y^{-1}(\mcF)$).

On the other hand $(\mcT_\mcI,\mcF_\mcI)$ is a hereditary torsion pair in $\Mod\mcD^c$, which implies that $\mcT_\mcI={^{\perp}\big(\Inj{\Mod\mcD^c}\cap\mcF_\mcI\big)}={^{\perp}Y}$, for some object $Y$ such that $\Prod(Y)=\Inj{\Mod\mcD^c}\cap\mcF_\mcI$. If we take now $Q\in\PInj\mcD$ such that $\y Q\cong Y$ then,  by \cite[Theorem 1.8]{K}, we get that  $\mcV={^{\perp}Q}=\bigcap_{i\ge 0}{^{\perp}Q[i]}$. 
The task is then reduced to prove that the pair $\big({^{\perp}Q},({^{\perp}Q[-1]})^\perp\big)$ is a co-$t$-structure in $\mcD$, for then if we put $\mcW:=({^{\perp}Q[-1]})^\perp$ and observe that $Q\in\mcW[1]$, the equality $\mcV={^{\perp}\PInj\mcW[1]}$ becomes obvious.

The key result here (somewhat alike Proposition~\ref{prop.t-structure from pi}) is \cite[Theorem 2.3.4]{Bo} which says that $\big({^{\perp}Q},({^{\perp}Q[-1]})^\perp\big)$ is a co-$t$-structure in $\mcD$ provided that $Q$ is a perfect object in $\mcD^\op$ (see Subsection~\ref{subsec:triangulated}). This is what we are going to verify now. If $f\dd X\la Z$ is a morphism in $\mcD$, we have the following equivalences:

\begin{align*}
\Hom_\mcD(f,Q)=0
&\Longleftrightarrow \Hom_{\Mod\mcD^c}\big(\y(f),Y\big)=0 \\
&\Longleftrightarrow \Hom_{\Mod\mcD^c}\big(\Img(\y(f)), Y\big)=0 \\
&\Longleftrightarrow \Img\big(\y(f)\big)\in\mcT_I.
\end{align*}
The first equivalence is again due to \cite[Theorem 1.8]{K}, the second holds since $Y$ is injective and the last one by the choice of $Y$. Hence, since $\y$ respects products and $\mcT_\mcI$ is closed under products (it is a torsion-free class in $\Mod\mcD^c$), it follows that the class of morphisms in $\mcD$ satisfying $\Hom_\mcD(f,Q)=0$ is closed under  all  products in $\mcD$, or equivalently coproducts in $\mcD^\op$, as required.
\end{proof}

\subsection{Compactly generated \texorpdfstring{$t$-structures}{t-structures} have a locally fp heart}
\label{subsec:comp gen}

Except for the final main result, where we shall work in a more general context, we assume all through this subsection that $\mcD$ is a compactly generated triangulated category and that $\mathbf{t}=(\mcU,\mcV)$ is a compactly generated $t$-structure in it. 

The following is the crucial result. Since $\mcV$ is definable, we also indirectly obtain that the suspended ideal $\mcI$ corresponding to $\mathbf{t}$ via Theorem~\ref{thm.classification definable} consists precisely of the maps in $\mcD^c$ factoring through $\mcU[1]\cap\mcD^c$. In fact, the compactly generated $t$-structures in $\mcD$ are known to bijectively correspond to suspended subcategories of $\mcD^c$ which are closed under  direct summands; see~\cite[Theorem 4.2.1(3)]{Bo16} or, under the existence of a derivator enhancement, also \cite[Theorem 4.5]{StPo}.

\begin{prop} \label{prop.description-compgener-aisle}
Suppose that $\mcD$ and $\mathbf{t}$ are as above and denote $\mcU_0=\mcU\cap\mcD^c$. The following assertions hold:

\begin{enumerate}
\item $\mcU$ is the smallest subcategory of $\mcD$ that contains $\mcU_0$ and is closed under coproducts, extensions and Milnor colimits.

\item For an object $U\in\mcD$, the following assertions are equivalent:

\begin{enumerate}
\item $U\in\mcU$;
\item Any morphism $f\dd C\la U$, with $C$ compact,  factors through some object in $\mcU_0$;
\item There exists a pure epimorphism $\coprod_{i\in I}U_i\la U$, for some family $(U_i)_{i\in I}$ of objects in $\mcU_0$.
\end{enumerate}
\item $\mcU$ is closed under pure monomorphisms in $\mcD$.
\end{enumerate}
In particular $\Add(\mcU_0)$ is a $t$-generating subcategory of $\mcU$ (see Definition \ref{def.t-generating}). 
\end{prop}
\begin{proof}
(1) Let $\overline{\mcU}$ denote the smallest subcategory of $\mcD$ that contains $\mcU_0$ and is closed under coproducts,  extensions and Milnor colimits.   It is clear that $\overline{\mcU}\subseteq\mcU$ since $\mcU$ satisfies all those closure properties and contains $\mcU_0$. On the other hand, by \cite[Theorem 12.1]{KN}, we know that 
$\mcU=\Susp_\mcD(\mcU_0)$, where $\Susp_\mcD(\mcS)$ denotes the smallest suspended subcategory of $\mcD$ that contains $\mcS$ and is closed under coproducts, for each subcategory $\mcS$. Since we have $\mcU_0[1]\subseteq\mcU_0$ we immediately get that $\overline{\mcU}[1]\subseteq\overline{\mcU}$, which implies that $\overline{\mcU}$ is a suspended subcategory closed under taking coproducts. This gives the inclusion $\Susp_\mcD(\mcU_0)=\mcU\subseteq\overline{\mcU}$, which is then an equality. 

(2) $(c)\Longrightarrow (b)$ Fix a pure epimorphism $p\dd\coprod_{i\in I}U_i\la U$ as described in the statement of (2)(c). Then any morphism $f\dd C\la U$, with $C$ compact, admits a factorization $f=p\circ g$, where $g\dd C\la\coprod_{i\in I}U_i$ is some morphism. Due to the compactness of $C$, $g$ factors through some finite subcoproduct of the $U_i$ and such a subcoproduct is in $\mcU_0$.

$(b)\Longrightarrow (c)$ We can always construct a pure epimorphism $q\dd\coprod_{i\in I}C_i\la U$, for some family $(C_i)_{i\in I}$ of compact objects in $\mcD$.
For example, one can take a set $\mathcal{C}$ of representatives of the isoclasses of objects in $\mcD^c$ and take the canonical morphism $\coprod_{C\in\mathcal{C}}C^{(\Hom_{\mcD}(C,U))}\la U$.  If we denote by $\iota_j\dd C_j\la\coprod_{i\in I}C_i$ the canonical $j$-th map to the coproduct, then $q_j:=q\circ\iota_j\dd C_j\la U$ is a morphism from a compact object. By hypothesis, we have a factorization $q_j\dd C_j\stackrel{u_j}{\la}U_j\stackrel{q'_j}{\la}U$, for some $U_j\in\mcU_0$. If now $q'\dd\coprod_{i\in I}U_i\la U$ is the morphism with the $q'_j$ as components, then we have a decomposition $q=q'\circ (\coprod u_i)$, which implies that $q'$ is also a pure epimorphism. 

$(a)\Longleftrightarrow (b)\,\&\,(c)$ Let us denote by $\tilde{\mcU}$ the full subcategory of $\mcD$ consisting of the objects $\tilde{U}$ which satisfy the equivalent conditions (2)(b) and (2)(c).
Since $\tilde{\mcU}\subseteq\mcU$ by Corollary~\ref{cor.U closed under pure epi/mono}, we just need to prove that $\mcU\subseteq\tilde{\mcU}$. 

Using (2)(b), we clearly see that $\tilde{\mcU}$ is closed under taking coproducts. On the other hand, if $\tilde{U}_1\la\tilde{U}_2\la\cdots\la\tilde{U}_n\la\cdots$ is a sequence in $\tilde{\mcU}$, then the canonical morphism $q\dd\coprod_{n\in\mathbb{N}}\tilde{U}_n\la\Mcolim\tilde{U}_n$ is a pure epimorphism. If, for each $n\in\mathbb{N}$, using (2)(c), we fix a pure epimorphism $p_n\dd\coprod_{i\in I_n}U_i\la\tilde{U}_n$, where  $I_n$ is some set and all $U_i$ are in $\mcU_0$, then $\coprod p_n\dd\coprod_{n\in\mathbb{N}}(\coprod_{i\in I_n}U_i)\la\coprod_{n\in\mathbb{N}}\tilde{U}_n$ is also a pure epimorphism. It follows that $q\circ (\coprod p_n)$ is a pure epimorphism, which, by (2)(c), implies that $\Mcolim\tilde{U}_n\in\tilde{\mcU}$. Therefore $\tilde{\mcU}$ is also closed under taking Milnor colimits. 

We next prove that $\tilde{\mcU}$ is closed under extensions, and then assertion (1) will give the desired inclusion $\mcU\subseteq\tilde{\mcU}$.  Let $\tilde{U}_1\stackrel{u}{\la}X\stackrel{v}{\la}\tilde{U}_2\la\tilde{U}_1[1]$ be a triangle in $\mcD$ with  $\tilde{U}_k\in\tilde{\mcU}$, for $k=1,2$,  and let $f\dd C\la X$ be any morphism, where $C\in\mcD^c$. Then $v\circ f$ factors as a composition $C\stackrel{v'}{\la}U_2\stackrel{h}{\la}\tilde{U}_2$, where $U_2\in\mcU_0$. Completing to a triangle we obtain a triangle in $C'\stackrel{u'}{\la}C\stackrel{v'}{\la}U_2\la C'[1]$ in  $\mcD^c$ together with a morphism of triangles

\[
\xymatrix{
  C'  \ar[r]^-{u'} \ar@{.>}[d]_{g} &
  C   \ar[r]^-{v'} \ar[d]_f &
  U_2 \ar[r] \ar[d]^h &
  C'[1] \ar@{.>}[d]^{g[1]}
  \\
  \tilde{U}_1 \ar[r]_-{u} &
  X \ar[r]_{v} &
  \tilde {U}_2 \ar[r] &
  \tilde{U}_1[1].
}
\]

But the morphism $g\dd C'\la\tilde{U}_1$ factors as a composition $C'\stackrel{g_1}{\la}U_1\stackrel{g_2}{\la}\tilde{U}_1$, where $U_1\in\mcU_0$. By taking now the homotopy pushout of $u'$ and $g_1$, we obtain a triangle $U_1\la U\la U_2\la U_1[1]$ in $\mcD$ (and in $\mcD^c$). Since $\mcU_0$ is closed under taking extensions in $\mcD$, we get that $U\in\mcU_0$.  Moreover, we have that $f\circ u'=u\circ g=u\circ g_2\circ g_1$ and, by properties of homotopy pushouts (see \cite[Section 1.4]{N}), we immediately get that $f$ (and $u\circ g_2$) factor through $U$. It follows that $X\in\tilde{\mcU}$. 


(3) This is a direct consequence of Corollary~\ref{cor.U closed under pure epi/mono}.
%
%
\end{proof}

The following are immediate consequences.

\begin{prop} \label{prop.heart as quotient of modules over aisle}
Let $\mcD$ be a compactly generated triangulated category, $\mathbf{t}=(\mcU,\mcV)$ be a compactly generated $t$-structure with heart $\mcH$, let us put  $\mcU_0=\mcD^c\cap\mcU$ and let
\[ \y=\y_{|\mcU}\dd\mcU\la\Mod\mcU_0, \qquad U\rightsquigarrow\y U=\Hom_\mcU(?,U)_{\mcU_0}, \]
be the generalized Yoneda functor. Then the functor $G:=\y_{|\mcH}\dd\mcH\la\Mod\mcU_0$ is fully faithful and has a left adjoint $F\dd\Mod\mcU_0\la\mcH$ that is a Gabriel localization functor. Moreover, there is a natural isomorphism $F\circ\y\cong (\H)_{| \mcU}$. 
\end{prop}
\begin{proof}
Just apply Theorem \ref{thm.t-generating left Kan extensions},  with the $t$-generating subcategory $\mcP=\Add(\mcU_0)$, taking into account that we have a clear equivalence of categories $\widehat{\mcP}=\modf\mcP\stackrel{\cong}{\la}\Mod\mcU_0$ that takes $M\rightsquigarrow M_{| \mcU_0}$.
\end{proof}

\begin{cor} \label{cor.heart has a generator}
Let $\mcD$ be a compactly generated triangulated category  and let $\mathbf{t}=(\mcU,\mcV)$ be a compactly generated $t$-structure in $\mcD$ with heart $\mcH$. Then  $\H(\mcU_0)$ is a skeletally small class of generators of $\mcH$.
\end{cor}
\begin{proof}
Using the notation from the last proof, $\{F(\y P)=\H(P)\mid P\in\mcP\}$ is a class of generators of $\mcH$ since so is $\{\y P\mid P\in\mcP\}$ for $\widehat{\mcP}$. We end the proof by noting that $\H(\mcP)\subseteq \Add(\H(\mcU_0))$.
\end{proof}

For any skeletally small pre-additive category $\mcA$,  we will denote by $\mathbf{FP}_2(\Mod\mcA)$ the subcategory  of $\Mod\mcA$ consisting of the $\mcA$-modules $M$ which admit a projective presentation $P^{-2}\la P^{-1}\la P^0\la M\rightarrow 0$, where the $P^{-k}$ are finitely generated projective $\mcA$-modules. The following lemma is crucial for our purposes. 

\begin{lem} \label{lem.FP2-generation}
Let $\mcT=\Ker(F)$, for the functor $F$ as in Proposition~\ref{prop.heart as quotient of modules over aisle}. The hereditary torsion pair $(\mcT,\mcT^\perp)$ in $\Mod\mcU_0$ is generated by modules in   $\mathbf{FP}_2(\Mod\mcU_0)$. That is, there is a (necessarily skeletally small) class $\mcS$ of modules in  $\mathbf{FP}_2(\Mod\mcA)\cap\mcT$ such that $\Gen(\mcS)=\mcT$.
\end{lem}
\begin{proof}
By using the associated Grothendieck topology in $\mcU_0$ (see \cite[Section 3.2]{PSV} or \cite{Low2}), it is enough to prove that if $U_0\in\mcU_0$ and $N$ is a submodule of the representable $\mcU_0$-module $\y(U_0)$ such that $\y(U_0)/N\in\mcT$,
then there is an epimorphism $M'\twoheadrightarrow\y(U_0)/N$ for some $M'\in\mathbf{FP}_2(\Mod\mcA)\cap\mcT$.

Let then $U_0$ and  $N$ be as in last paragraph, and put $M:=\y(U_0)/N$, which is an object of $\mcT$. Consider an epimorphism $p\dd\coprod_{i\in I}\y(U_i)\twoheadrightarrow N$, where the $U_i$ are in $\mcU_0$. Since $\y$ preserves coproducts, $\coprod_{i\in I}\y(U_i)\cong\y(\coprod_{i\in I}U_i)$. By the Yoneda lemma, the composition 
$\y(\coprod_{i\in I}U_i)\cong \coprod_{i\in I}\y(U_i)\stackrel{p}{\la} N\stackrel{\operatorname{inc}}{\rightarrowtail}\y(U_0)$
is necessarily of the form $\y(g)$, for some morphism $g\dd\coprod_{i\in I}U_i\la U_0$ in $\mcU$, and it follows that $\Coker(\y(g))\cong M$.
As in the proof of Lemma~\ref{lem.identification-of-Serre-subcategory}, we observe that $F(\y(g))\cong\H(g)$ is an epimorphism since $F$ is exact and $F(M)=0$. Upon completing $g$ to a triangle
\begin{equation} \label{eq:pure-tria2}
U_0[-1] \stackrel{u}{\la} U'\stackrel{f}{\la}\coprod_{i\in I}U_i\stackrel{g}{\la}U_0,
\end{equation}
we, thus, have $U'\in\mcU$ by Lemma~\ref{lem:t-generating tria}.
Since $U_0[-1]$ is compact, the morphism $u\dd U_0[-1]\la U'$ factors as $U_0[-1]\stackrel{u'}\la U'_0\stackrel{\alpha}\la U'$ for some $U'_0\in\mcU_0$, by Proposition~\ref{prop.description-compgener-aisle}(2), and we obtain a commutative diagram with triangles in rows, where $U''_0\in\mcU_0$ is the cone of $u'$ and the morphism $\beta$ comes from the axioms of triangulated categories:
\[
\xymatrix{
U_0[-1] \ar[r]^-{u'} \ar@{=}[d] &
U'_0 \ar[r]^-{f'} \ar[d]_-{\alpha} &
U''_0 \ar[r]^-{g'} \ar@{.>}[d]^-{\beta} &
U_0. \ar@{=}[d]
\\
U_0[-1] \ar[r]_-{u} &
U' \ar[r]_-{f} &
\coprod_{i\in I} U_i \ar[r]_-{g} &
U_0.
}
\]

If we apply the restricted Yoneda functor to the last diagram and denote $M'=\Coker(\y(g'))$, we obtain a commutative diagram in $\Mod\mcU_0$ with exact rows and an epimorphism in the rightmost column:
\[
\xymatrix{
\y(U'_0) \ar[r]^-{\y(f')} \ar[d]_-{\y(\alpha)} &
\y(U''_0) \ar[r]^-{\y(g')} \ar[d]^-{\y(\beta)} &
\y(U_0) \ar[r] \ar@{=}[d] &
M' \ar[r] \ar@{.>>}[d] &
0
\\
\y(U') \ar[r]^-{\y(f)} &
\y(\coprod_{i\in I} U_i) \ar[r]_-{\y(g)} &
\y(U_0) \ar[r] &
M \ar[r] &
0
}
\]
Now it follows from Lemma~\ref{lem.identification-of-Serre-subcategory} that $M'\in \mathbf{FP}_2(\Mod\mcA)\cap\mcT$, as required.
\end{proof}

We are now ready for the main result of the section.

\begin{thm} \label{thm.locally-fp-heart}
Let $\mcD$ a triangulated category with coproducts, let $\mathbf{t}=(\mcU,\mcV)$ be a compactly generated $t$-structure in $\mcD$, with heart $\mcH$,  and put $\mcU_0=\mcU\cap\mcD^c$. Then  $\mcH$ is a locally finitely presented Grothendieck category and its subcategory of finitely presented objects is
$\fp(\mcH)=\H(\mcU_0)$. 

When in addition $\mathbf{t}$ restricts to the subcategory $\mcD^c$ of compact objects, the heart $\mcH$ is also locally coherent.
\end{thm}

\begin{proof}
Replacing $\mcD$ by the compactly generated triangulated subcategory $\mcL:=\Loc_{\mcD}(\mcU_0)$ if necessary, we can and shall assume in the sequel that $\mcD$ is compactly generated. This is because the restricted $t$-structure $\mathbf{t}':=(\mcU,\mcL\cap\mcV)$ has the same heart as $\mathbf{t}$.
Note that the composition $\mcL=\Loc_\mcD(\mcU_0)\stackrel{\iota}{\rightarrowtail}\mcD\stackrel{\H}{\la}\mcH$ is the cohomological functor associated to the restricted $t$-structure $\mathbf{t}'=(\mcU,\mcV\cap\mcL)$. Therefore the reduction to the case when $\mcD$ is compactly generated is also valid when proving the last statement of the theorem. 

Let now $\mcG$ be the Giraud subcategory of $\Mod\mcU_0$ associated to the torsion pair $(\mcT,\mcT^\perp )$, where $\mcT=\Ker(F)$ for $F$ from Proposition~\ref{prop.heart as quotient of modules over aisle}. By Lemma \ref{lem.FP2-generation}, we can fix a set $\mcS\subseteq \mathbf{FP}_2(\Mod\mcU_0)$ such that $\mcT=\text{Gen}(\mcS)$. It then follows that $\mcG$ consists of the $\mcU_0$-modules $Y$ such that $\Hom_{\Mod\mcU_0}(S,Y)=0=\Ext_{\Mod\mcU_0}^1(S,Y)$, for all $S\in\mcS$. This implies that $\mcG$ is closed under taking direct limits in $\Mod\mcU_0$. By Proposition \ref{prop.locally-fp-quotient categories}, we get that $\mcH\simeq (\Mod\mcU_0)/\mcT$ is locally finitely presented and that $\fp(\mcH)=\add(F(\modf\mcU_0))$. 

Let us assume that $X=F(M)$, where $M$ is a finitely presented $\mcU_0$-module. There is then a morphism $f\dd U_1\la U_0$ in $\mcU_0$ such that the sequence
\[
\y(U_1)\stackrel{\y(f)}{\la}\y(U_0)\stackrel{p}{\la}M\la 0
\]
is exact, for some epimorphism $p$. Thanks to the natural isomorphism $F\circ\y\cong (\H)_{| \mathcal{U}}$, if we apply $F$ to the last sequence, we get an exact sequence 
\[\H(U_1)\stackrel{H(f)}{\la}\H(U_0)\la X\rightarrow 0.\]
However, we also have a triangle $U_1\stackrel{f}{\la}U_0\stackrel{g}{\la}U'\la U_1[1]$, with its terms in $\mcU_0$, which induces an exact sequence
\[ \H(U_1)\stackrel{\H(f)}{\la}\H(U_0)\stackrel{\H(g)}{\la}\H(U')\la\H(U_1[1])=0. \]
We then get that $X\cong\H(U')$ and so $F(\modf\mcU_0)\subseteq\H(\mcU_0)$.  On the other hand, we have that $\H(U_0)\cong F(\y U_0)\in F(\modf\mcU_0)$, for all $U_0\in\mcU_0$. So $\fp(\mcH)=\add(\H(\mcU_0))$.

We must still prove that every summand $Y$ of an object $X\in\H(\mcU_0)$ lies already in $\H(\mcU_0)$. To that end, let $U_0\in\mcU_0$ be such that $X\cong\H(U_0)$ and denote by $g$ the composition $U_0\stackrel{\operatorname{can}}\la\coa0(U_0)=\H(U_0)\twoheadrightarrow Y$, where the last arrow stands for a split epimorphism. When completing $g$ to a triangle
\[ Y[-1] \la W\stackrel{f}{\la} U_0 \stackrel{g}{\la} Y, \]
we obtain a split exact sequence $0\la\H(W)\stackrel{\H(f)}{\la}\H(U_0)\stackrel{\H(g)}{\la} Y\la0$ in $\mcH$. In particular, $Z:=\H(W)\in\fp(\mcH)$. As we clearly have that $W\in\mcU$ and $\Add(\mcU_0)$ is $t$-generating by Proposition~\ref{prop.description-compgener-aisle}, there is a morphism $q\dd \coprod_{i\in I}U_i \la W$ with $U_i\in\mcU_0$ for all $i\in I$ which induces an epimorphism $\H(q)\dd \H(\coprod_{i\in I}U_i)\cong\coprod_{i\in I}\H(U_i)\la Z$ in $\mcH$. Thus $Z=\sum_{i\in I}q(\H(U_i))$ and, $Z$ being finitely presented in $\mcH$, there is a finite subset $J\subseteq I$ such that $Z=\sum_{i\in J}q(\H(U_i))$. All in all, if we put $h:=f\circ q_{|\coprod_{i\in J}U_i}$, we observe that $\Coker\H(h)=\Coker\H(f)\cong Y$. If we complete $h$ to a triangle
\[ \coprod\nolimits_{i\in J}U_i \stackrel{h}\la U_0 \la U' \la \coprod\nolimits_{i\in J}U_i[1], \]
we obtain an object $U'\in\mcU_0$ with $\H(U')\cong\Coker\H(h)\cong Y$, as desired.

For the final statement, note that when $\mathbf{t}$ restricts to $\mcD^c$, the category $\mcU_0$ has weak kernels. Indeed if $f\dd U\la U'$ is a morphism in $\mcU_0$ and we complete it to a triangle $X\stackrel{g}{\la}U\stackrel{f}{\la}U'\laplus$ in $\mcD^c$, then the composition $\ai 0 X\stackrel{\operatorname{can}}{\la}X\stackrel{g}{\la}U$ is a weak kernel of $f$ in $\mcU_0$. Hence $\Mod\mcU_0$ is a locally coherent Grothendieck category (see \cite[Corollary 1.11]{PSV}). The local coherence of $\mcH$ then follows by \cite[Proposition A.5]{K-spectrum}.
\end{proof}

\begin{rem} \label{rem.Bondarko}
In \cite{Bo} the author has proved, by using different methods, that any compactly generated $t$-structure has a Grothendieck heart with $\H(\mcD^c)$ as skeletally small class of generators. Note that Bondarko's result is a particular case of Theorem \ref{thm.Groth.heart-for-definable-coaisle}. 
\end{rem}

\begin{rem}
If $\mathbf{t}=(\mcU,\mcV)$ is a compactly generated $t$-structure in the homotopy category $\mcD$ of a compactly generated stable $\infty$-category, then the equality $\fp(\mcH)=\H(\mcU_0)$ from Theorem~\ref{thm.locally-fp-heart} also follows from \cite[Corollary 5.5.7.4(5)]{Lur-HTT} (see also the introduction to~\cite[Appendix C.6]{Lur-SAG}).
\end{rem}


\section{Cosilting objects and \texorpdfstring{$t$-structures}{t-structures} with AB5 hearts}
\label{sec.cosilting}

\subsection{Partial cosilting objects}
\label{subsec.partial cosilting objects}

Now we relate the objects $Q$ from Theorem~\ref{thm.AB5 heart} to concepts which appeared in the literature.

\begin{opr} \label{def.partial-cosilting-set}
Suppose that $\mcD$ is a triangulated category with products and $Q$ is an object of $\mcD$. We shall say that

\begin{enumerate}
\item $Q$ is \emph{AMV partial cosilting} (for Angeleri-Marks-Vitoria) when $_{}^{\perp_{>0}}Q$ is a co-aisle of $\mcD$ that contains $Q$. The induced $t$-structure will be said to be an \emph{AMV partial cosilting $t$-structure}.
\item $Q$ is \emph{NSZ partial cosilting} (for Nicol\'as-Saor\'in-Zvonareva) when $(\mcU_Q,\mcV_Q):=(_{}^{\perp_{<0}}Q, (_{}^{\perp_{\leq 0}}Q)^\perp)$ is a $t$-structure in $\mcD$, called in the sequel the \emph{NSZ partial cosilting $t$-structure} associated with $Q$, and $\Hom_\mcD(?,Q)$ vanishes on $\mcV_Q[-1]$. 
\end{enumerate}
 
 The object $Q$ is called \emph{cosilting} when it is (AMV or NSZ) partial cosilting and cogenerates $\mcD$. The associated $t$-structure, which is $(_{}^
{\perp_{<0}}Q,_{}^{\perp_{>0}}Q)$, is the \emph{cosilting $t$-structure} associated to $Q$.
\end{opr}

The NSZ partial cosilting objects are rather generally related to right non-degenerate $t$-structures with Grothendieck hearts.
 
\begin{prop} \label{prop.NSZ}
Let $\mcD$  be a triangulated category with products and coproducts and let $\mathbf{t}=(\mcU,\mcV)$ be a $t$-structure with heart $\mcH$. Consider the following assertions:

\begin{enumerate}
\item $\mathbf{t}$ is the $t$-structure associated with a pure-injective NSZ partial cosilting object $Q$.
\item $\mathbf{t}$ is right non-degenerate, the heart $\mcH$ is  AB5 with an injective cogenerator and $\H\dd\mcD\la\mcH$ preserves coproducts. 
\item $\mathbf{t}$ is right non-degenerate, the heart $\mcH$ is a Grothendieck category and the functor $\H\dd\mcD\la\mcH$ preserves coproducts.  
 \end{enumerate}
 The implications $(1)\Longrightarrow (2)\Longleftarrow (3)$ hold true. When $\mcD$ satisfies Brown representability theorem, the implication $(2)\Longrightarrow (1)$ also holds. When $\mcD$ is standard well generated, all assertions are equivalent.  
\end{prop}
\begin{proof}
$(3)\Longrightarrow (2)$ is clear and both implications $(1)\Longrightarrow (2)$ and $(2)\Longrightarrow (1)$ are included  in the proof of  \cite[Corollary 4.1]{NSZ}, bearing in mind that $\mcH$ is as in~(2) exactly when the NSZ partial cosilting object $Q$ representing the functor $\Hom_\mcH(\H(?),E)$, for the injective cogenerator $E$ of $\mcH$, is pure-injective (see Proposition \ref{prop.near-partialsilting-object}). 

$(1)=(2)\Longrightarrow (3)$ (when $\mcD$ is standard well generated) follows by the truth of implication $(1)\Longrightarrow (4)$ in Theorem  \ref{thm.AB5 heart} in this case. 
\end{proof}

If the category $\mcD$ is standard well generated, we can say much more.

\begin{prop} \label{prop.NSZ extended}
Let $\mcD$ be an standard well generated triangulated category and $\mathbf{t}=(\mcU,\mcV)$ be a $t$-structure such that the heart $\mcH$ is a Grothendieck category and the functor $\H\dd\mcD\la\mcH$ preserves coproducts. Then the object $Q\in\mcD$ from Theorem~\ref{thm.AB5 heart}(1) is pure-injective NSZ partial cosilting.
Moreover, if $\mcH_Q$ is the heart of the NSZ partial cosilting $t$-structure $\mathbf{t}_Q=(\mcU_Q,\mcV_Q)$ and $(\H)'\dd\mcD\la\mcH_Q$ is the associated cohomological functor, then there is an equivalence of categories $\Psi\dd\mcH\stackrel{\cong}{\la}\mcH_Q$ such that $\Psi\circ\H\cong (\H)'$.
\end{prop}

\begin{proof}
Suppose that $Q$ is obtained from $\mathbf{t}$ via Theorem~\ref{thm.AB5 heart}. Then $\mathbf{t}_Q:=(\mcU_Q,\mcV_Q):=(_{}^{\perp_{<0}}Q, (_{}^{\perp_{\leq 0}}Q)^\perp)$ is a $t$-structure in $\mcD$ thanks to Proposition \ref{prop.t-structure from pi}. Since clearly $\mcV_Q\subseteq\mcV$ and $\Hom_\mcD(?,Q)$ vanishes on $\mcV[-1]$, it vanishes on $\mcV_Q[-1]$. It follows that $Q$ is NSZ partial cosilting. On the other hand,  by the proof of Theorem \ref{thm.classify homological functors} or that of Proposition \ref{prop.near-partialsilting-object}, we have equivalences of categories
\[ \Inj \mcH=\Prod_\mcH\big(\H(Q)\big)\cong\Prod_\mcD(Q)\cong\Prod_{\mcH_Q}\big((\H)'(Q)\big)=\Inj{\mcH_Q}, \]
where $(\H)'\dd\mcD\la\mcH_Q$ is the cohomological functor associated to $\mathbf{t}_Q$. Then, by the dual of Corollary \ref{cor.Cont(P)}, we conclude that $\mcH$ and $\mcH_Q$ are equivalent via an equivalence $\Psi\dd\mcH\stackrel{\cong}{\la}\mcH_Q$ that takes $\H(Q)$ to $(\H)'(Q)$. But then we have functorial isomorphisms
\begin{align*}
\Hom_{\mcH_Q}\big((\Psi\circ\H)(?),(\H)'(Q)\big) &\cong
\Hom_{\mcH_Q}\big((\Psi\circ\H)(?),(\Psi\circ\H)(Q)\big) \\&\cong
\Hom_{\mcH}\big(\H(?),\H(Q)\big) \\&\cong
\Hom_\mcD(?,Q) \\&\cong
\Hom_{\mcH_Q}\big((\H)'(?),(\H)'(Q)\big).
\end{align*}
By Yoneda's lemma and the fact that $(\H)'(Q)$ is an injective cogenerator of $\mcH_Q$, we get a natural isomorphism $\Psi\circ\H\cong (\H)'$ (recall Lemma~\ref{lem.Cont(P)}). 
\end{proof}

\begin{rem} \label{rem.NSZ shift}
A word of warning is due in connection with the last proposition. A NSZ partial cosilting $t$-structure is always right non-degenerate. Therefore if the $t$-structure $\mathbf{t}$ of last proposition is right degenerate, then $\mathbf{t}\neq\mathbf{t}_Q$. 
\end{rem}

The final result of this subsection shows that, up to suitable localization, any smashing $t$-structure with an AB5 heart is given by a pure-injective NSZ partial cosilting object. If the $t$-structure is already left non-degenerate, the resulting $t$-structure will be non-degenerate and hence given by a pure-injective cosilting object.

\begin{prop} \label{prop.smashing-AB5heart}
Let $\mcD$ be a triangulated category with coproducts and $\mathbf{t}=(\mcU,\mcV)$ be a smashing $t$-structure of $\mcD$ with heart $\mcH$. The following assertions hold:
\begin{enumerate}
\item $\mcL:=\Loc_\mcD(\mcU)$ is a smashing localizing subcategory of $\mcD$, i.e. the inclusion functor $\mcL\rightarrowtail\mcD$ has a right adjoint which preserves coproducts. 
\item If $\mcD$ satisfies Brown representability theorem, then so does $\mcL$.
\item If $\mcD$ is standard well generated, then so is $\mcL$.
\item Suppose  that $\mcD$  satisfies Brown representability theorem (resp. $\mcD$ is standard well generated). The heart $\mcH$ is a complete AB5 abelian category with an injective cogenerator (resp. $\mcH$ is a Grothendieck category) if, and only if,  there exists a pure-injective NSZ partial cosilting object $Q$ of $\mcL$ such that $\mathbf{t}'=(\mcU,\mcV\cap\mcL)$ is the $t$-structure in $\mcL$ associated to $Q$. 
\end{enumerate}
\end{prop}
\begin{proof}
(1) We start by proving that the inclusion $\mcL\rightarrowtail\mcD$ has a right adjoint, for which we will check that any object 
 $D\in\mcD$ fits into a triangle $L\la D\la Y\la L[1]$, with $L\in\mcL$ and $Y\in\mcL^\perp$. For each integer $n\geq 0$, we have a triangle
\[ 
\Delta_n\dd \quad \tau_\mathbf{t}^{\leq n}D\la D\la\tau_\mathbf{t}^{>n}D\la\tau_\mathbf{t}^{\leq n}[1]
\] 
with respect to the $t$-structure $(\mcU[-n],\mcV[-n])$. We have an obvious functorial morphism of triangles from $\Delta_n$ to $\Delta_{n+1}$ with identity map on $D$. Using an argument similar to that of  \cite[Theorem 12.1]{KN}, we get a commutative diagram
\[
\xymatrix{
D \ar[r]^-{1} \ar[d] &
D \ar[r]^-{1} \ar[d] &
D \ar[r]^-{1} \ar[d] &
\cdots \ar[r]^-{1} &
D \ar[r]^-{1} \ar[d] &
\cdots
\\
\tau_\mathbf{t}^{>0}D \ar[r]_-{f_1} &
\tau_\mathbf{t}^{>1}D \ar[r]_-{f_2} &
\tau_\mathbf{t}^{>2}D \ar[r]_-{f_3} &
\cdots  \ar[r]^-{f_n} &
\tau_\mathbf{t}^{>n}D \ar[r]_-{f_{n+1}} &
\cdots
}
\]

\noindent
Bearing in mind that $D$ is isomorphic to the Milnor colimit of the upper sequence, using Verdier's $3\times 3$ Lemma as in the argument in [op.cit], we get a triangle
\[
L\la D\la\Mcolim\tau_\mathbf{t}^{>n}D\laplus,
\]
where $L$ fits into a triangle
\[
\coprod_{n\geq 0}\tau_\mathbf{t}^{\leq n}D\la\coprod_{n\geq 0}\tau_\mathbf{t}^{\leq n}D\la L\la\coprod_{n\geq 0}\tau_\mathbf{t}^{\leq n}D[1].
\]
We then clearly have  that $L\in\mcL$ and the task is reduced to check that $\Mcolim\tau_\mathbf{t}^{>n}D$ is in $\mcL^\perp =\mcU^{\perp_{\mathbb{Z}}}$. For that, note that for each $r\ge 0$ we have a triangle
\[
\coprod_{n\geq r}\tau_\mathbf{t}^{>n}D\stackrel{1-f}{\la}\coprod_{n\geq r}\tau_\mathbf{t}^{>n}D\la\Mcolim\tau_\mathbf{t}^{>n}D\la\coprod_{n\geq r}(\tau_\mathbf{t}^{>n}D)[1]
\]
by~\cite[Lemma 1.7.1]{N}. Consequently, we have $\Mcolim\tau_\mathbf{t}^{>n}D \in \mcU[-r]^\perp$ for any $r\ge0$ since all the coproducts in the above triangle belong to $\mcU[-r]^\perp$. Finally, as $\mcU[-r]\subseteq\mcU$ for $r\le0$, we have $\Mcolim\tau_\mathbf{t}^{>n}D \in \mcU[-r]^\perp$ also in this case, as we wished to prove.

In order to see that $\mcL$ is smashing, it suffices to prove that $\mcL^\perp$ is closed under coproducts. To that end, notice that $\mcL^\perp = \mcU^{\perp_\mathbb{Z}} = \bigcap_{n\in\mathbb{Z}}\mcV[n]$ and the conclusion follows since $\mcV$ is closed under coproducts.

(2) Let $\Phi\dd\mcD\la\mcL$ be right adjoint to the inclusion functor $\iota\dd\mcL\rightarrowtail\mcD$, so that the unit $\eta\dd1_\mcL\la\Phi\circ\iota$ is a natural isomorphism. Let $H\dd\mcL\la\Ab$ be a contravariant cohomological functor which sends coproducts to products. Then $H\circ\Phi\dd\mcD\la\Ab$ has the same property. By the Brown representabilty property of $\mcD$, we get an object $D_H\in\mcD$ such that $H\circ\Phi$ and $\Hom_\mcD(?,D_H)$ are naturally isomorphic. We then have natural isomorphisms $$H\stackrel{H(\eta)}{\la}H\circ\Phi\circ\iota\cong\Hom_\mcD(\iota (?),D_H)\cong\Hom_\mcL(?,\Phi (D_H)),$$ the last one of them due to the adjunction $(\iota ,\Phi )$. Therefore $H$ is representable and so $\mcL$ satisfies Brown representability theorem. 

(3) By~\cite{K-Brown-rep}, well generated triangulated categories have a set of perfect generators. Then, by \cite[Corollary 2.4]{NS-recollements}, we have an induced TTF triple $(\mcL,\mcL^\perp,\mcL^{\perp\perp})$ in $\mcD$. But then, by properties of TTF triples, we have  triangulated equivalences $\mcL\cong\mcL^{\perp\perp}\cong\mcD/\mcL^\perp$. If $\Psi\dd\mcD\la\mcL^\perp$ is the left adjoint to the inclusion functor $\mcL^\perp\la\mcD$, then $\mcL^\perp=\Loc_\mcD(\Psi (\mcX))$, where $\mcX$ is any set of perfect generators of $\mcD$ (see \cite[Lemmas 2.2 and 2.3]{NS-recollements}). It is now routine to check that if $\mcD=\mcC/\Loc_\mcC(\mcS)$, for some compactly generated triangulated category $\mcC$ and some set $\mcS\subseteq\mcC$, and  we choose any set $\mcX'$ of objects of $\mcC$ that is mapped onto $\Psi (\mcX)$ by the quotient functor $q\dd\mcC\la\mcD$, then $\mcD/\Loc_\mcD(\Psi(\mcX))$ is equivalent to $\mcC/\Loc(\mcS\cup\mcX')$. Therefore $\mcL\simeq\mcC/\Loc(\mcS\cup\mcX')$ is standard well generated. 

(4) When $\mcD$ satisfies Brown representability theorem (resp. is standard well generated), the restricted $t$-structure $\mathbf{t}'=(\mcU,\mcV\cap\mcL)$ in $\mcL$ is smashing, clearly right non-degenerate and its heart is again $\mcH$. By using Proposition \ref{prop.NSZ} and assertions~(2) and~(3), we conclude that $\mcH$ is complete AB5 with an injective cogenerator (resp. a Grothendieck category)  if, and only if, $\mathbf{t}'$ is the $t$-structure associated to a pure-injective NSZ partial cosilting object of $\mcL$. 
\end{proof}


\subsection{Left non-degenerate \texorpdfstring{$t$-structures}{t-structures}}

As long as we are interested only in the cohomological functor of a $t$-structure, NSZ partial cosilting objects are very convenient, as was shown in \S\ref{subsec.partial cosilting objects}. This is in particular illustrated by Proposition~\ref{prop.NSZ extended} which says that for a given $t$-structure whose associated cohomological functor is nice enough, there exists a (possibly different) NSZ partial cosilting $t$-structure with the same cohomological functor.

If we are concerned in how precisely the heart sits in $\mcD$, however, we need more refined tools. Here we assume that the $t$-structure in question is left non-degenerate, which can be achieved in various situations (see the end of Section \ref{sec.t-structures and localization}). We stress again that in that case, the cohomological functor of a $t$-structure preserves coproducts if and only if the $t$-structure is smashing (Remark~\ref{rem.H preserves coproducts}).

The final result of the paper explains the role of AMV partial cosilting objects. The following result may be seen as a derivator-free generalization of the equivalence $(1)\Longleftrightarrow (4)$ of \cite[Theorem 4.6]{Lak}

\begin{prop} \label{prop.AMV}
Let $\mcD$ have coproducts and products and let $\mathbf{t}$ be a left non-degenerate $t$-structure with heart $\mcH$. Consider the following assertions:

\begin{enumerate}
\item $\mathbf{t}$ is the $t$-structure associated with a pure-injective AMV partial cosilting object.
\item $\mathbf{t}$ is smashing and $\mcH$ is an AB5 abelian category with an injective cogenerator. 
\item $\mathbf{t}$ is smashing and $\mcH$ is a Grothendieck category.
\end{enumerate}
The implications $(1)\Longrightarrow (2)\Longleftarrow (3)$ hold true. When $\mcD$ satisfies Brown representability theorem,  the implication $(2)\Longrightarrow (1)$ also holds. When $\mcD$ is standard well generated all assertions are equivalent.
\end{prop}
\begin{proof}
$(3)\Longrightarrow (2)$ is clear and, assuming that $\mcD$ is standard well generated, the implication 
$(2)\Longrightarrow (3)$ is a direct consequence of the implication $(4)\Longrightarrow (5)$ in Theorem \ref{thm.AB5 heart}.

$(1)\Longrightarrow (2)$ The equality $\mcV={^{\perp_{>0}}Q}$ implies that 
\[ \Hom_\mcD(V[-1],Q)=\Hom_\mcD(V,Q[1])=0, \]
for all $V\in\mcV$. Then $Q$ satisfies conditions (1)(a) and (1)(c) of Theorem \ref{thm.AB5 heart}. Moreover, if $M\in\mcH$ is an object such that $\Hom_\mcD(M,Q)=0$, then we have $\Hom_\mcD(M,Q[n])=0$, for all $n\geq 0$, since $\mcH\subseteq\mcV=_{}^{\perp_{>0}}Q$. It then follows that $\Hom_\mcD(M[1],Q[j])=0$, for all $j>0$, so that $M[1]\in\mcV$. It follows that $M\in\mcH\cap\mcV[-1]\subseteq\mcU\cap\mcV[-1]=0$. Hence also condition (1)(b) of the mentioned theorem holds. It follows that $\mcH$ is complete AB5 with an injective cogenerator. Finally, it is clear that $\mcV={^{\perp_{>0}}Q}$ is closed under taking coproducts, so that $\mathbf{t}$ is smashing. 

$(2)\Longrightarrow (1)$ (when $\mcD$ satisfies Brown representability theorem) Since $\H$ clearly preserves coproducts, Theorem~\ref{thm.AB5 heart} tells us that there exists a $Q\in\mcV$ satisfying conditions (1)(a)--(c) of that theorem. In particular we have that $\mcV\subseteq{^{\perp_{>0}}Q}$. It remains to check that the reverse inclusion also holds. For this, just apply the argument in the proof of \cite[Lemma 4.4]{Lak}, based on \cite[Theorem 3.6]{AMV}.
\end{proof}


\bibliographystyle{alpha}
\bibliography{Saorin-Stovicek-Locallyfp-heart}


\end{document}